\documentclass[11pt]{amsart}

\usepackage[english]{babel}
\usepackage[utf8]{inputenc}
\usepackage[T1]{fontenc}
\usepackage{lmodern}
\usepackage{amssymb}
\usepackage{latexsym}
\usepackage{amsmath}
\usepackage[dvips]{graphics}
\usepackage{graphicx}  
\usepackage{verbatim}
\usepackage{xcolor}
\newtheorem{prop}{Proposition}[section]
\newtheorem{definition}[prop]{Definition}
\newtheorem{lemma}[prop]{Lemma}
\newtheorem{theorem}[prop]{Theorem}
\newtheorem{corollary}[prop]{Corollary}

\newtheorem{remark}[prop]{Remark}
\newtheorem{example}[prop]{Example}
\newtheorem{assum}[prop]{Assumption}

\newcommand{\eqnsection}{\renewcommand{\theequation}{\thesection.\arabic{equation}}
      \makeatletter \csname @addtoreset\endcsname{equation}{section}\makeatother}

\def\vep{\vep}

\def\tr{{\mbox{Tr}} }
\def\NN{\mathbb N}
\def\RR{\mathbb R}
\def\CC{\mathbb C}
\def\bN{{\mathbf N}}
\def\bn{{\mathbf n}}
\def\bX{{\bf X}}
\def\bY{{\bf Y}}

\def\bU{{\bf U}}
\def\bO{{\bf O}}

\def\nn{{\noindent}}

\def\ra{\rightarrow}

\def\N{{\mathbb N}}
\def\R{{\mathbb R}}

\def\E{{\mathbb E}}
\def\P{{\mathbb P}}
\def\Q{{\mathbb Q}}
\def\ts{\times}

\def\vep{\varepsilon}

\def\mun{\hat{\mu}^N}

\def\Q{\mathbb Q}
\def\Da{\mathcal D}

\def\HNC{\mathcal H^N(\CC)}

\def\ccl{\CC\langle X_1,\dots,X_\ell\rangle}

\DeclareMathOperator{\Id}{Id}
\DeclareMathOperator{\Tr}{Tr}
\DeclareMathOperator{\Diag}{Diag}

\DeclareMathOperator{\Diff}{Diff}

\title{Matrix models at low temperature}
\date{\today}
\author{Alice Guionnet}
\address{UMPA, CNRS  UMR 5669, ENS Lyon,46 all\'ee d'Italie,
	69007 Lyon, France } \email{Alice.Guionnet@umpa.ens-lyon.fr}
\author{\'Edouard Maurel-Segala}
\address{Université Paris-Saclay, CNRS, Laboratoire de mathématiques d’Orsay, 91405, Orsay, France} \email{edouard.maurel-segala@universite-paris-saclay.fr }

\begin{document}

\begin{abstract}
	In this article  we investigate the behavior of multi-matrix unitary invariant models under a potential $V_\beta=\beta U+W$ when the inverse temperature $\beta$ becomes very large. We first  prove, under mild hypothesis on the functionals $U,W$ that as soon at these potentials are "confining" at infinity, the sequence of spectral distribution of the matrices are tight when the dimension goes to infinity. Their limit points are  solutions of Dyson-Schwinger's equations. Next we investigate a few specific models, most importantly the "strong single variable model" where $U$ is a sum of potentials in a single matrix and the "strong commutator model" where $U = -[X,Y]^2$.
\end{abstract}
\maketitle

\tableofcontents

\section{Introduction}
Matrix integrals and their large $N$  limit   play an important role in both  mathematics and physics.  Random matrices were  first introduced in statistics by Wishart \cite{wishart} to analyze a large array of noisy data, a viewpoint  that is now particularly relevant and useful for  principal component analysis and statistical learning \cite{MontaBA,AKJ}.  Wigner \cite{wigner} and Dyson \cite{dyson} later suggested  that the statistics of their eigenvalues model very well those of the high energy levels in heavy nuclei. Even more surprisingly, Montgomery \cite{Montgo}  showed that random matrices are closely related to the zeros of the Riemann Zeta function, a conjecture that nowadays provides great intuition for many mathematical results, see e.g \cite{KeatingSnaith,ABB}.  Random matrices are moreover intimately related to integrable systems to which they  furnish key examples. Random matrices have  also played  a central role in  the theory of operator algebras since  Voiculescu \cite{voicstflour,voi91}  proved that they are asymptotically free. In physics, they were  introduced into quantum field theory by 't Hooft  \cite{thooft}, then used in two-dimensional quantum gravity and  string theory \cite{BIPZ,David,KKM}, and related with algebraic geometry \cite{eynard,BEO,EyOr}. In both free probability, statistical learning  and physics, a central question is to estimate, in the large $N$-limit,  a matrix integral of the form
$$I_{N}^{\ell}(V)=\int \exp\{-N\tr V(X_{1}^{N},X_{2}^N,\ldots,X_{\ell}^{N})\} dX_{1}^{N}\cdots dX_{\ell}^{N}$$ where the integration holds over  Hermitian $N\times N$ matrices, $V$ is  apolynomial and $dX_{i}^{N}$ denotes the Lebesgue measure on this space. 
It turns out that this was done only in few cases. The case where the integral holds only over one matrix has been thoroughly studied. Indeed, in this case one can diagonalize the matrix and find that
$$I_{N}^{1}(V)=C_{N}\int \prod_{1\le i<j\le N}|\lambda_{i}-\lambda_{j}|^{2} e^{-N\sum V(\lambda_{i})} \prod_{1\le i\le N} d\lambda_{i}\,,$$
where $C_{N}$ is a universal constant.  $F_{N}^1({V})=N^{-2}\ln I_{N}^{1}(V)$ can be shown to converge as $N$ goes to infinity. Its limit was computed by using a saddle point method or large deviations techniques \cite{BAG97} as soon as $V$ is continuous and goes to infinity fast enough, yielding
$$F^{1}(V):=\lim_{N\rightarrow\infty}\frac{1}{N^{2}}\ln \int \prod_{1\le i<j\le N}|\lambda_{i}-\lambda_{j}|^{2}e^{-N\sum V(\lambda_{i})} \prod_{1\le i\le N} d\lambda_{i}$$
\begin{equation}\label{fe}\quad =-\inf\{\int\int (\frac{1}{2}V(x)+\frac{1}{2}V(y)-\ln |x-y|)d\mu(x)d\mu(y)\}\end{equation}
where the supremum is taken over probability measures $\mu$ on the real line. Large $N$ corrections to this limit can also be computed, at least in some nice regimes  \cite{johansson, BGK, Bog1, ChEy1}.
When more than one matrix is involved, even the very existence of the limit of the renormalized free energy $F_{N}^{\ell}(V)=N^{-2}\ln I_{N}^{\ell}(V)$ is unknown. The existence of such a limit in great generality would already be of importance for free entropy theory \cite{Vo02} since it would  show that  one can replace limsup by liminf in the definition of microstates free entropy, which would entail important properties such as summability under freeness. One class of multi-matrix integrals is still amenable to asymptotic analysis: those where matrices interact via a quadratic interaction such that the interaction can be expressed in terms of the Harich-Chandra-Itzykson-Zuber integral or more generally spherical integrals \cite{GZ3,Ghu21,EyMe,Mehrig, GJSZ,BH98}. Apart from  these special cases, the case where the potential $V$ is a small modification of  a quadratic potential, which is also called a  ``high temperature'' regime,  is also well understood. Indeed, it has been  foreseen that matrix integrals expand as generating functions for enumerating maps \cite{thooft, BIPZ} at a formal level. In \cite{GM,GM2}, it was shown that 
this expansion holds asymptotically under general hypotheses, so that   the free energy limit can be expressed  as a generating function for the enumeration of planar maps. The central tool to derive
 these asymptotic estimates is based on the fact that the matrices remain bounded and that the limit points of their  moments satisfy the so-called Dyson-Schwinger's equations. It can  then be  shown that these equations  asymptotically have a unique solution  when the potential is a small perturbation of the quadratic potential, a solution that expands analytically in the small parameters of the potential. Such a uniqueness fails in the general case, even in the case of a single  matrix  \cite{BoG2,BGK}. In the latter case, it can be seen that for general potentials the limiting distribution of the spectrum has a disconnected support and that the Dyson-Schwinger's equations have a unique solution only if the mass in each connected component is fixed. The  limiting distribution is the unique minimizer in \eqref{fe}, but the Dyson-Schwinger equations describe critical points,  which in fact are much more numerous. 
Therefore, except for  these  few cases, the asymptotic analysis of most models is still a wide open problem and  even efficient  numerical methods have  to be found \cite{KZh}. 
In this article, we  start  the study of  multi-matrix models in the  ``low temperature''  regime  where  the potential dominates the measure. We will indeed consider the case where we can decompose $V$ into 

$$V_\beta = \beta U + W$$
with two polynomials $U,W$ and $\beta$ very large. The driving idea is that the situation where $\beta$ is infinite should somehow be easier to analyze, as should the situation where $\beta$ is large, as a perturbation of the latter.

We will first give sufficient conditions for $U$ and $W$ such that the integral $I_N^\ell(V_\beta)$ is finite and mostly supported on bounded matrices. We will also show that this integral is concentrated on matrices satisfying the so-called Dyson-Schwinger's equations. However, these equations  generally have multiple solutions  in  the low temperature situation, even in the one-matrix case. We then characterize special cases. When $U$ is strictly convex, we show that the analysis can be reduced to the high temperature situation studied in \cite{GM}. When $U$ is the sum of polynomials in one matrix, we can also estimate $I_N^\ell(V_\beta)$ for sufficiently large $\beta$.
 The crux of our study will be to consider the case where $-U$ is the square of the  commutator in two matrices $X$ and $Y$ and $W$ decomposes as the sum of two polynomials in $X$ and $Y$ respectively. This type of  models  is also considered in \cite{KZh, KKN}: It can be solved when $W$ is quadratic but not in general.
In the limit $N$ and then $\beta$ goes to infinity, our matrix model will concentrate on operators $x$ and $y$ which commute. However, there are many ways to achieve this: for example  these operators could have the same eigenbasis or one of them could be the multiple of the identity. Our results show that the optimal scenario also depends  on $W$ and 
give the asymptotics of  the free energy$F_{N}^{2}(V_{\beta})$ for $N$ and then $\beta$ going to infinity. We also  give a general criterion for the  tightness of these multi-matrix models. Our paper considers the Hermitian case but could easily be  extended to symmetric matrices (see  e.g. \cite{Ma} and  \cite{CGM} for other generalizations of the loop equations  techniques to orthogonal  and symmetric matrices). 

\bigskip

{\bf Acknowledgments:} We are grateful to Amir Dembo for many discussions at a preliminary stage of this work. This work was supported in part by the ERC Project LDRAM : ERC-2019-ADG Project 884584.

\section{Notations and main results}
We let $\ccl$ be the space of non-commutative polynomials in $\bX=(X_1,\dots,X_\ell)$ with complex coefficients. 
We will in general consider $V$ to be a polynomial but for some applications we will need that the potentials belong to a more general class of smooth functions given by:
$$\mathbb C\langle X_{1},\ldots, X_{\ell}\rangle^{c}=\{ P(X_{1},\ldots,X_{\ell})=\sum_{{i=1}}^{k} z_{i }X_{j_{1}^{i}}\cdots  X_{j_{d_{i}}^{i}}+\sum_{j=1}^{\ell}\vartheta_{j} (X_{j})\}$$
where $z_{j}$ are complex numbers, $k, j^d_{i}$ , the $j_{m}^{i}$ are integer numbers in $\{1,\ldots, \ell\}$ and $\vartheta_{j}$ are  smooth real-valued functions. $\mathbb C\langle X_{1},\ldots, X_{\ell}\rangle^{c}_{sa}$ will denote the subset of self-adjoint elements of $\mathbb C\langle X_{1},\ldots, X_{\ell}\rangle^{c}$, namely functions $V\in\mathbb C\langle X_{1},\ldots, X_{\ell}\rangle^{c}$ such that
$V=V^{*}$,  where $*$ is the involution such that for any integer number $d$, any $j_{1},\ldots,j_{d}\in [1,\ell]^{d}$, any complex number $z$ and any smooth function $\vartheta:\mathbb R\mapsto \mathbb R$, we have 
$$\left(  z X_{j_{1}}\cdots  X_{j_{d}}\right)^{*}=\bar z X_{j_{d}}X_{j_{d-1}}\cdots X_{j_{1}}, \quad \vartheta(X_j)^{*}= \vartheta(X_{j})\,.$$
$\mathbb C\langle X_{1},\ldots, X_{\ell}\rangle^{c}$ is equipped with a unit denoted by $1$.
A polynomial $P$ is cyclically invariant iff it  is the sum of terms of the form
$$z \sum_{\ell} X_{i_{\ell}}\cdots X_{i_{k}}X_{i_{1}}\cdots X_{i_{\ell-1}}$$
for some complex number $z$ and $(i_{1},\ldots,i_{k})\in [1,\ell]^{k}$, $k\in\N$.
$\ccl^{c}$  is equipped with the unit $I$.

Let $V_\beta=\beta  V+W$ be a one parameter family of non-commutative polynomials so that $V$ and $W$ belong to  $\mathbb C\langle X_{1},\ldots, X_{\ell}\rangle^{c}_{sa}$ and $\beta$ is  a non-negative real number. 
We study the matrix model
given by the $\ell$-tuple $\bX^N=(X^N_1,\dots,X^N_\ell)$ of $N\times N$ Hermitian matrices  with law
\begin{equation}\label{defQ}
d\Q^N_{V_\beta}(\bX^N)=\frac{1}{Z^N_{V_\beta}}\exp\left(-N\Tr V_\beta(\bX^N)\right)d\bX^N\end{equation}
when the dimension goes to infinity and then $\beta$, the analogue of the  inverse temperature in statistical physics, goes to infinity.  $d\bX^{N}$ is the Lebesgue measure on the set  $\HNC$  of $N\times N$ Hermitian matrices and we have denoted in short $X^N=(X^N_1,\dots,X^N_\ell)$. 
Moreover, $Z^{N}_{V_{\beta}}$ is the partition function given by 
\begin{equation}\label{defZ}
Z^N_{V_\beta}=\int \exp\left(-N\Tr V_\beta(\bX^N)\right)d\bX^N\end{equation}
We may and shall assume that $V$ is cyclically invariant, since the trace is left invariant by cyclic permutation, namely 
 we may replace any monomial $q(X)=X_{i_{1}}\cdots X_{i_{k}}$ in $V$ by 
 $$c(q)(\bX):=\frac{1}{k}\sum_{p=1}^kX_{i_{p}}\dots X_{i_k}X_{i_1}\dots X_{i_{p-1}}\,.$$
 
 Our first result gives sufficient conditions on the potential $V$  to insure that matrices distributed according to $\Q^{N}_{V}$ stay bounded with large probability. 
This will prove that the sequence of moment distribution of the $\bX^N=(X^N_i)_{1\le i\le \ell}$ is tight under $\Q^{N}_{V}$.
 Let us denote $\mathcal D_i$  the cyclic derivative with respect to the $i$th variable so that $\mathcal D_i$ is linear on the space of non commutative polynomials and
 $$\mathcal D_i (X_{i_1}\cdots X_{i_k})=\sum_{p:i_p=i} X_{i_{p+1}}\cdots X_{i_k}X_{i_1}\cdots X_{i_{p-1}}\,.$$
   We extend $\mathcal D_{i}$ to smooth functions in one variable by setting $\mathcal D_{i}\vartheta(X_{j})=1_{i=j}\vartheta'(X_{j})$.
 We will denote 
 $$ X_i.\mathcal D_i V=\frac{1}{2}\left(X_i\mathcal D_i V+\mathcal D_i V X_i\right)\,.$$
 We denote by $\|A\|_{\infty}$ the operator norm of a matrix $A$:
 $$\|A\|_{\infty}=\sup_{\| u\|_{2}=1}\langle Au, Au\rangle^{1/2}\,.$$
We can finally define the notion of trapping potentials that will insure that matrices stay bounded under $\Q^N_{V}$.

\begin{definition} Let $\eta$  be a non-negative real number and  $A$ be a  finite real number and $I=(I_{1},\ldots,I_{d})$ be a partition of $\{1,\dots,\ell\}$.
Let us denote for $1\leqslant j\leqslant d$, $Z^N_j=\sum_{i\in I_j} (X^N_i)^2$.

We say that a potential $V\in \mathbb C\langle X_{1},\ldots, X_{\ell}\rangle^{c}_{sa}$ is {\bf $(\eta,A,I)$-trapping} if  for every $k,N\in\N^*$, for every  $\bX^N=(X^N_1,\dots,X^N_\ell)\in\HNC^{\ell}$,
$$\Tr\left(\sum_{j=1}^{d}(Z^N_{j})^k\sum_{i\in I_{j}}X_{i}^{N}.\mathcal D_{i}V(X^{N})\right)\geqslant\Tr \left(\eta\sum_{j=1}^{d}(Z^{N}_{j})^{k+1}-A(1+\sum_{j=1}^{d}(Z^{N}_{j})^{k})\right)\!\!.
$$
\end{definition}

 Our main interest in this definition is the following theorem, which shows that when the potential $V$ is $(\eta,A,I)$-trapping, the spectrum of the $\ell$ tuple of matrices  $\bX^N$ stays bounded with large probability.  
 \begin{theorem}\label{Ttosupport} Assume that there exists a positive real number $\eta$ and a finite real number $A$ such that 
$V$  in $ \mathbb C\langle X_{1},\ldots, X_{\ell}\rangle^{c}_{sa}$ is {\bf $(\eta,A,I)$-trapping. }
Then with $L=32\max\left((1+|A| )d \eta^{-1},1\right)$,  under $\Q^N_{V}$, almost surely
$$\limsup_{N\ra\infty} \max_{1\le i\le \ell} \| X^N_i\|_\infty\le L .$$
Moreover, for every $\kappa\in (0,1)$, we have 
\begin{eqnarray}\label{zx}
\limsup_{N\rightarrow\infty}  N^{\kappa-1}\ln \Q^N_V\left( \max_{1\le i\le \ell} \| X^N_i\|_\infty\ge L^{\frac{1}{2}+\frac{1}{\kappa}} \right)\le -\frac{1}{4} .\end{eqnarray}

\end{theorem}
\begin{remark}
If $X=CY+aI$ with $C$ invertible and bounded   it is enough to prove that the potential $V$ is trapping  for the variables $Y$ since  we then get the conclusion of Theorem \ref{Ttosupport} for the variables $X$ as well since
 {
 $$\max_{1\le i\le \ell}\|X^{N}_{i}\|_\infty\le \|C\|_\infty\max_{1\le i\le \ell}\|Y^{N}_{i}\|_\infty +|a|$$}
\end{remark}
In this article,  we  will study the empirical distribution $\mun$ of the matrices $\bX^{N}$ under $\Q^{N}_{V}$. For a $\ell$-tuple $\bY=(Y_{1},\ldots,Y_{\ell})$ of self-adjoint $N\times N$ matrices, we denote by $\hat\mu_{\bY}$ the empirical distribution defined as the linear form on $\mathbb C\langle X_{1},\ldots, X_{\ell}\rangle^{c}$ such that for every $P\in \mathbb C\langle X_{1},\ldots, X_{\ell}\rangle^{c}$,
$$\hat\mu_{\bY}(P):=\frac{1}{N}\Tr \left(P(Y_{1},\ldots,Y_{\ell})\right)\,.$$
Throughout this article, we denote in short $\mun=\hat\mu_{\bX^{N}}$ with $\bX^{N}$ the random matrices with law $\Q^{N}_{V}$. We see $\mun$ as a non-commutative law \cite[Section 5.2]{AGZ}, namely as a linear form $\tau$ on
 $\mathbb C\langle X_{1},\ldots, X_{\ell}\rangle^{c}$ such that for any $P,Q\in \mathbb C\langle X_{1},\ldots, X_{\ell}\rangle^{c}$,
 $$\tau(PP^{*})\ge 0, \tau(1)=1,\tau(PQ)=\tau(QP)\,.$$
 We equip the space of non-commutative laws with its weak topology.
An important consequence of Theorem \ref{Ttosupport} is that the empirical distribution $\mun$ is tight \cite[Lemma 5.4.6]{AGZ}.
\begin{corollary} \label{CorTight}If $V$ is {\bf $(\eta,A,I)$-trapping}, the empirical distribution $\mun$  is sequentially tight. 
\end{corollary}
We hereafter investigate examples of  potentials which are $(\eta,A,I)$-trapping. Note that  they may be replaced by cyclically equivalent potentials without loss of generality. 

\begin{theorem}\label{conf-to-T}

\begin{enumerate}
\item
If $U$ is {\bf $(\eta,A,I)$-trapping} and  $W$ is {\bf $(\eta',A',I)$- trapping} and $\beta\geqslant 0$, $\beta U+W$ is {\bf $(\beta\eta+\eta',\beta A+A',I)$-trapping}.
\item 
If $V$ is $(\eta,A,\{\{1\},\dots,\{\ell\}\})$-trapping then for any non negative real numbers  $\alpha_{ij}$, the potential 
$$V(\bX)-\sum_{1\leqslant i<j\leqslant \ell}\alpha_{ij}[X_{i},X_{j}]^2$$
is also $(\eta,A,\{\{1\},\dots,\{\ell\}\})$-trapping. Here,  $[X_{i},X_{j}]=X_{i}X_{j}-X_{j}X_{i}$ denotes  the commutator.
\item
We say that a potential $V$ in $\ccl$ is {\bf confining } if there is a partition $(I_1,\ldots,I_d)$ of $\{1,\ldots,\ell\}$   so that we can write $V$ as
$$V(X)=\sum_{i=1}^\ell P_i(X_i)+ \sum_{u=1}^p d_uq_u(X)+\sum_{v=1}^r R_v \left(\sum_{i\in I_{m_{v}}} \alpha_{vi} X_i\right)$$
where $P_i$ are polynomials whose term of highest degree has degree $2D_j\geqslant 2$ with a positive coefficient, $R_v$ are polynomials 
 with  even highest degree with positive coefficient {and the $\alpha_{ij}$ are real numbers,} $1\leqslant m_v\leqslant d$, the $q_u$ are monomials of degree at most $2\min_{i\in I_m}D_i-1$ for every $m$ such that there exists $i$ in $I_m$ such that $X_i$ appears in $q_u$.

If $V$ is confining, it is {\bf $(\eta,A,I)$-trapping}  for some $\eta,A$ with $\eta>0$.

\end{enumerate}
\end{theorem}

\begin{remark} \label{remarktrapping}
Theorem  \ref{conf-to-T}
 gives a quite general and technical condition in order to be confining and therefore trapping. Let us show it  encompasses many natural examples.
\begin{enumerate}
\item If $V(\bX)=\sum_{i=1}^\ell P_i(X_i)+ \sum_{u=1}^p d_uq_u(\bX)$ with the same hypothesis on the $P_i$ and the $q_v$ than in  Theorem \ref{conf-to-T}
then $V$ is confining with the trivial partition $I=\{\{1\},\dots,\{\ell\}\}$.

For example $V(X,Y)=c_XX^{2D_X}+c_YY^{2D_Y}+V_X(X)+V_Y(Y)+W(X,Y)$ is confining if $c_X,c_Y>0$ and $V_X$ is of degree at most $2D_X-1$, $V_Y$ is of degree at most $2D_Y-1$ and $W $is of degree at most $2\min(D_X,D_Y)-1$.
Another example of confining polynomial is for instance
$$P(X_{1},X_{2},X_3)=X_{1}^{4}+X_{2}^{4}+X_{1}X_{2}^{2}+X_{3}^{2}.$$
\item 
Any polynomial which is a sum of terms of type $\sum P_i(\sum_{j} v_{ij}X_j)$ with $P_{i}$ a polynomial going to $+\infty$ at $\pm\infty$ and $v$ of rank $\ell$ can be rewritten as a confining polynomial with partition $\{1,\dots,\ell\}$ after a linear change of variables. This implies that
polynomials of the form
$$V(\bX) = \sum_iP_i(\sum_j v_{ij}X_j)+\sum_{ij}u_{ij}X_iX_j+ \sum_{u=1}^p d_uq_u(\bX)$$
can be rewritten as a confining polynomial with partition $\{1,\dots,\ell\}$ after a linear change of variable if $u$ is a symmetric positive matrix, the $P_i$'s are going to $+\infty$ at $\pm\infty$ and the $q_u$ are of degree at most $\min\deg P_i -1$.
\item  
We would have liked  trace convex polynomials to be  confining , where $V$ is trace convex iff $\bX\rightarrow \Tr V(\bX)$ is convex on the space of $\ell$-tuples of $N\times N$ Hermitian matrices for any dimension $N$. 
Unfortunately,  this is not the case as
$X^2+Y^2+X^2Y^2$ is trace convex but not confining (and not trapping). { Note however that such polynomials are known to satisfy the conclusions of Theorem \ref{Ttosupport} thanks to Brascamp-Lieb inequalities, see \cite[Section 6.5]{GuioFlour}}
\end{enumerate}
\end{remark}
  Our next result shows that,  if $V$ is trapping,  any limit point of $\mun$ satisfies the so-called Dyson-Schwinger's equations. To describe the later we first define the non-commutative derivative $\partial_{i}$ with respect to the $i$th variable so that for every  choices of monomials $q(\bX)= X_{i_{1}}\ldots X_{i_{k}}$ with $i_{j}\in \{1,\ldots,\ell\}$ we have 
 $$\partial_{i}X_{i_{1}}\ldots X_{i_{k}}=\sum_{j: i_{j}=i}X_{i_{1}}\ldots X_{i_{j-1}}\otimes X_{i_{j+1}}\cdots X_{i_{k}}\,.$$
 This definition extends to smooth functions $\vartheta$ of one variable by setting $$\vartheta(x,y)=\frac{\vartheta (x)-\vartheta(y)}{x-y}\,.$$
 
\begin{theorem}\label{theoDSint}(Dyson-Schwinger's equations)  
	Under the assumptions of Theorem \ref{Ttosupport}, any limit point $\tau_{V}$ of $\mun$ under $\Q_{V}^{N}$ satisfies the Dyson-Schwinger's equations: for any $P\in \mathbb C\langle X_{1},\ldots,X_{\ell}\rangle ^{c}$, any $i\in \{1,\ldots,\ell\}$,
	\begin{equation}\tau_{V}\left(\mathcal D_{i} V P\right)=\tau_{V}\otimes\tau_{V}\left(\partial_{i} P\right)\,.\end{equation}
	Moreover, if $V=\beta U+W$ with $U$ and $W$ $(\eta,A,I)$ trapping for some $\eta>0$ and $A$ finite, then there exists a finite $B$ (depending only on $A$ and $\eta$) such that for $\beta>0$ large enough, every $i\in \{1,\ldots,\ell\}$,  and every  integer number $k$, 
	\begin{equation}\label{crit}\tau_{V_{\beta}}\left(|\mathcal D_{i} U|^{2k}\right)\le (\frac{B}{\beta})^{k}\,.\end{equation}
\end{theorem}
The last result \eqref{crit} shows, as could be expected, that under reasonable assumptions the limit points $\tau_{V_{\beta}}$ concentrate on the critical points of the potential $U$ when $\beta$ goes to infinity. Its proof is a direct consequence of Proposition \ref{proploc}.

The Dyson-Schwinger's equations are 
  not sufficient to characterize the limit in general since they may  have several solutions. Efficient  numerical methods have recently been proposed to investigate these solutions \cite{KZh}.

  We therefore study special cases where the potential is given by $V_{\beta}=\beta U+W$ for some fixed polynomials $U$ and $V$ and $\beta$ large.  In Section \ref{convex}, we study the case where $U$ is strictly convex. We then show that the matrices following $\Q^{N}_{V_{\beta}}$ converge to multiples of the identity when $N$ and then $\beta$ go to infinity and moreover for $\beta$ large but finite the empirical distribution $\hat\mu^{N}$ converges almost surely and its moments can be expanded in powers of $\beta^{-1/2}$. In Section \ref{single}, we consider the case where $U(\bX)=\sum U^{i}(X_{i})$ is the sum of polynomials in one variable that can achieve their minimal value at several points and $W$ is small when the spectrum of the $X_{i}$ are close to these minimizers. We then show, see Theorem \ref{low},  that for $\beta$ large enough, the spectrum of the matrices $X^{N}_{i}$ under $\Q^{N}_{V}$ converges as $N$ goes to infinity. 
  To simplify the statement of our result in this introduction, let us assume  that $U^{i}$ takes its minimum value at $\{x^{i}_{j}\}_{1\le j\le m_{i}}$ and that these minimizers are non degenerate in the sense that  $\min_{j}(U^{i})''(x^{i}_{j})>0$.
 
{
  \begin{theorem}\label{lowint}Under some technical hypotheses over $U$ and $W$ (see Assumption \ref{assumsingle2}), with $V_{\beta}=\beta U+W$
for $\beta$ large enough, the empirical distribution  $\hat\mu_{N}$ of $\bX^{N}$ converges  $\Q^{N}_{V_{\beta}}$-
almost surely towards  a non-commutative law $\tau_{\beta}$. Moreover, for any polynomial $P$,  $\tau_{\beta}(P)$  expands as a converging series in 
$\beta^{-\frac{1}{2}}$. $\tau_{\infty}$ is the law of $\ell$ free variables with law  {$\frac{1}{m_{i}} \sum_{j=1}^{m_{i}}\delta_{x^{i}_{j}}$} for $i\in \{1,\ldots,\ell\}$.
\end{theorem}}
  Furthermore, we show in Theorem \ref{theosum} that if the minimizers of $U^{i}$ are degenerate, 
    the mass of the Dirac at a minimizer $x$ depends on the flatness of $U^{i}$ close to $x$, namely on the exponent $k$ such that $U^{i}(y)-U^{i}(x)\simeq c(x-y)^{2k}$ for $y$ close to $x$, but not on the constant $c$ provided it is positive.

    Finally we consider  in Section \ref{commutator} the case where $U$ is the square of the commutator and more precisely
 
 \begin{equation}\label{potcom}V_{\beta}(\bX)=-\beta (X_{1}X_{2}-X_{2}X_{1})^{2}+V_{1}(X_{1})+V_{2}(X_{2})\,.\end{equation}
 Such a matrix-model can be considered as a toy model for Yang Mills theory \cite{KZh}. The authors of   \cite{KZh} say that this model is  ''unsolvable''.
 Our main result is that the non-commutative distribution of $X_{1}^{N},X_{2}^{N}$ under $\Q^{N}_{V_{\beta}}$ converges when $N$ goes to infinity and then $\beta$ goes to infinity towards two commutative variables whose  spectrum  lies at the minimizers of $V_{1}$ and $V_{2}$. The central point of our result is that the mass in these minimizers is nontrivial and again depends  on the flatness of the potentials in a neighborhood of the minimizers. In fact, we might have expected one of the matrices to converge  to a multiple of the identity to make the commutator small, as would be the case for instance  if we had a three matrix model 
 with $U(X_{1},X_{2},X_{3})=- (X_{1}X_{2}-X_{2}X_{1})^{2}- (X_{1}X_{3}-X_{3}X_{1})^{2}$ where making $X_{1}$ close to the identity would have a higher probability than making it commute with both $X_{1}$ and $X_{2}$ (see a more detailed discussion in section \ref{commutator}). 
 More precisely, we have the following statement. 
 {
 \begin{theorem}\label{theocommutator}We let  { $V_{1}$ and $V_{2}$ be two non-negative polynomials of one variable, going to $+\infty$  at infinity. Let
	 $z_1^{i}<z_2<\dots<z^{i}_{k}$ be the zeroes of $V_{i}$  and assume there exists  $c_1^{i},\dots,c_k^{i}\in\mathbb R^{+*}$ 
	such that, for every $j\in \{1,\ldots, k\}$,
	$$V_{i}(x)= c_j^{i}(x-z_j)^2+o(x-z_j)^2\,.$$ 
	Let $V_{\beta}$ be given by \eqref{potcom} and $X_{1},X_{2}$ following $\Q^N_{V_\beta}$. Then,  for every $\epsilon>0$ small enough,  almost surely,  for every $j\in \{1,\ldots,k\}$ and $i\in\{1,2\}$, we have 
	$$\lim_{\beta\rightarrow\infty}\limsup_{N\rightarrow\infty} \hat \mu_{X_{i}}^N([z_j^{i}-\epsilon, z_j^{i}+\epsilon])=\frac{(c_j^{i})^{-\frac{1}{2}}}{\sum_\ell (c^{i}_\ell)^{-\frac{1}{2}}}=:\gamma^{i}_{j}\,.$$
	Moreover,  for any monomial $P(X_{1},X_{2})=X_{1}^{k_{1}}X_{2}^{\ell_{1}}X_{1}^{k_{2}}X_{2}^{\ell_{2}}\cdots X_{1}^{k_{p}} X_{2}^{\ell_{p}}$, we have
	$$\lim_{\beta\rightarrow\infty}\limsup_{N\rightarrow\infty} \int \frac{1}{N} \Tr(P(X_{1}^{N},X_{2}^{N}))d\mathbb Q^{N}_{V_{\beta}}(X_{1}^{N},X_{2}^{N})=\nu_{1}(x^{\sum k_{i}})\nu_{2}(x^{\sum\ell_{i}})$$
	with $\nu_{i}=\sum_{j}\gamma_{j}^{i}\delta_{z_{j}^{i}}$ for $i=1,2$. 
	The same results hold if the above  limsup is replaced by a liminf.
	}
\end{theorem}}
The last result shows that $X_{1}^{N}$ and $X_{2}^{N}$ asymptotically commute but also their eigenspaces become independent. Moreover, the asymptotic distribution of each of the matrix is a sum of dirac masses at the minimizers of the potentials $V_{i}$ with explicit weights depending on the second derivative of the potential at these minimizers.

\section{Tightness of matrix models at finite temperature}\label{SectionTight}
In this section, we prove Theorems \ref{conf-to-T} and  \ref{Ttosupport}. This is based on new techniques to bound probabilities of rare events developed in the appendix, see Section \ref{appineq}. 
\subsection{ Proof of  Theorem \ref{conf-to-T}}

The first point is clear.
For the second point, we just observe that
$-[X_{1},X_{2}]^2=-(X_{1}X_{2}-X_{2}X_{1})^2$ satisfies the $(0,0,\{\{1\},\dots,\{\ell\}\})$-trapping condition.
To see this, remark that it amounts to check that for any integer number $k$:
$$\Tr (X_{1}^{2k+1}X_{2}X_{1}X_{2}+X_{2}^{2k+1}X_{1}X_{2}X_{1})\leqslant \Tr (X_{1}^{2k+2}X_{2}^2+X_{2}^{2k+2}X_{1}^2).$$
This is a direct consequence of Lemma \ref{matrix-inequ}.
Now if $V$ is $(\eta,A,\{\{1\},\dots,\{\ell\}\})$-trapping then by the first point for any $\alpha>0$,
$$V(\bX)-\alpha [X_{1},X_{2}]^2$$
is also $(\eta,A,\{\{1\},\dots,\{\ell\}\})$-trapping. The generalization to any number of square commutators is immediate.

We finally prove the third point.
Let us consider  a confining potential $V$ of the form :

$$V(X)=\sum_{i=1}^\ell P_i(X_i)+ \sum_{u=1}^p d_uq_u(X)+\sum_{v=1}^r R_v \left(\sum_{i\in I_{m_v}} \alpha_{vi} X_i\right)$$

Observe that 
$$\sum_{i\in I_m} X_i.\mathcal D_iR_v \left(\sum_{i\in I_{m_v}} \alpha_{vi} X_i\right)=1_{m=m_v}\tilde R_v \left(\sum_{i\in I_{m_v}} \alpha_{vi} X_i\right)$$
where $\tilde R_v(X)= X R'_v(X)$ is also a polynomial with a positive even term of highest degree, therefore it is bounded from below by some constant $-c_v$.
Hence,  for $j\in \{1,\ldots,m\}$, we  have (since $(Z^N_m)^k$ is non-negative)
$$\sum_{m=1}^d \Tr \left((Z^N_m)^k \sum_{i\in I_m} X_i^{N}.\mathcal D_i R_v(\sum_{{k\in I_{m_{v}}}} \alpha_{jk} X_k^{N})\right)
	\geqslant -c_v\Tr (Z^N_{m_v})^k ,$$
i.e. $R_v \left(\sum_{i\in I_{m_v}} \alpha_{vi} X_i\right)$ satisfies the $(0,c_v,I)$-trapping condition.
With this in mind, we have only to  worry about  the other part of the polynomial $V$ and then apply the first point of the theorem. {We can assume up to rescaling that the polynomials $P_{i}$ are monic. }
We consider a monomial $q_k$ in the decomposition of $V$ and denote it by $q$ for simplicity. We denote by $S(q)$ the smallest subset of $i\in [1,\ell]$ so that $q$ only depends 
  on $X_i, i\in S(q)$.  Let $s(q)$ be the set of $j$ such that $I_j\cap S(q)\neq \emptyset$ and define $D_q=\min_{i\in s(q)}\min_{i\in I_j} D_i$.
\begin{equation}\label{bound0}
\Tr \left(\sum_{j=1}^{d}(Z^N_{j})^k\sum_{i\in I_{j}}X_{i}^{N}.\mathcal D_{i}q(X^{N}) \right)=\sum_{j\in s(q)} 
 \Tr \left((Z^N_{j})^k\sum_{i\in I_{j}}X_{i}^{N}.\mathcal D_{i}q(X^{N}) \right).\end{equation}
Any monomial  $X_{i_1}\cdots  X_{i_t}$ coming in the linear decomposition of  $\mathcal D_i q$ has degree $t$ bounded above by $2D_q-1$ by assumption.
Observe that,  by H\"older's non commutative inequality with parameters $(p_{0},p_{1},\ldots, p_{t})=(1+\frac{t}{2k},2k+t,\ldots, 2k+t)$, we find
\begin{eqnarray*}
|\Tr (((Z^N_j)^{1/2})^{2k} X^N_{i_1}\dots  X^N_{i_t})|
&\leqslant &  \left(\Tr(Z^{N}_j)^{k+\frac{t}{2}}\right)^{\frac{2k}{2k+t}}\prod_{m=1}^{t}\left(\Tr |X_{i_{m}}|^{2k+t}\right)^{\frac{1}{2k+t}}\\
&\leqslant& \sum_{j\in s(q) }\Tr (Z^N_j)^{k+\frac{t}{2}}\\
&\leqslant & \varepsilon\Tr \sum_{j\in s(q)} (Z^N_j)^{k+D_q}+\varepsilon^{-1}\Tr \sum_{j\in s(q)} (Z^N_j)^k
\end{eqnarray*}
where in the second line we used that  for $s\in [1,t]$,  $$\Tr(|X_{i_s}|^{k})\le \max_{j\in s(q)}\Tr ((Z^{N}_j)^{k/2})$$ for every  integer number $k$. 
In the last line we picked 
$\varepsilon\in (0,1)$ to be chosen later and used the inequality valid for $t\in [1,2D_q-1]$ 
$$x^{k+t/2}\leqslant (\varepsilon^{2})^{D_q-t/2}x^{k+D_q}1_{\varepsilon^2x>1}+\varepsilon^{-t}x^k1_{\varepsilon^2x\leqslant1}$$
 and took the worst case scenario for each term. Using this inequality for each term coming from the $ q$,
 we deduce that there exists a finite constant $C(q)$ such  that
 \begin{equation}\label{bound2}\Tr (\sum_{j=1}^d (Z^N_j)^k\sum_{i\in I_j} X_i^N.\mathcal D_i q(X^N))\qquad\qquad \end{equation}
 $$\qquad\qquad \qquad \qquad \geqslant - C(q)\left( \varepsilon\sum_{j\in s(q)}\Tr (Z^N_j)^{k+D_q}+\varepsilon^{-1}\Tr\sum_{j\in s(q)} (Z^N_j)^k\right)\,.$$
 Moreover,  putting $\bar D_j=\min_{i\in I_j} D_i$, and 
 since $x^{2D_{i}}\ge x^{2\bar D_j}-1$ for every real number $x$  and $i\in I_j$ as  $D_{i}\geq \bar D_j$ for $i\in s(q)$, we have
  \begin{align}\sum_{j=1}^d \Tr ((Z^{N}_j)^{k} \sum_{k\in I_j} X_k.  \mathcal D_k \sum_{i=1}^\ell X_{i}^{2D_{i}}))&\geq 2\sum_{j=1}^d \bar D_{j }\Tr ((Z^{N}_{j})^{k} \sum_{i\in I_j} X_{i}^{2\bar D_{j}})\nonumber\\
  & - 2\sum_j \bar D_j |I_j| \Tr(Z_j^N)^k
\label{bound5}\end{align}
But, 
using  Lemma \ref{matrix-inequ-holder}, we find that for every $j\in \{1,\ldots,d\}$, 
\begin{eqnarray}
	\ell^{\bar D_{j}}
\Tr ((Z^N_j)^k \sum_{i\in I_j} X_{i}^{2\bar D_{j}})
&\geqslant&\Tr((Z^N_j)^{k+\bar D_j})\label{bound3}
\end{eqnarray}
Putting  \eqref{bound2} and \eqref{bound3} together we find  finite constants $C',C''$ such that
\begin{align*}
\Tr ((Z^N_j)^k \sum_{i\in I_j} X_i.\mathcal D_i V)&\geqslant  \sum_{j=1}^d \frac{2\bar D_j}{{\ell^{\bar D_{j}}}} ( \Tr((Z^N_j)^{k+\bar D_j}) - C'\varepsilon^{-1} \Tr(Z_j^N)^k)\\
&\quad -C\varepsilon  \sum_{r=1}^n \sum_{j\in s(q_r)}\Tr (Z^N_j)^{k+D_{q_r}}\\
&\ge {\frac{1}{\ell^{\max \bar D_{j}}}}\sum_{j=1}^d  \Tr((Z^N_j)^{k+1}) -C''\sum_{j=1}^d   \Tr((Z^N_j)^{k})
\end{align*}
where we finally choose $\varepsilon$ small enough, used that $D_{q_r}< \bar D_i$ for every $i\in s(q_r)$ and that $\min D_i\ge 1$ to find a finite constant $C''$ so that the above holds. This completes the proof.

\subsection{Proof of Theorem \ref{Ttosupport}}
 
Our proof uses bounds based on change of variables formula that we prove in the appendix, see Lemma \ref{matrix-ODE}.
We fix $L>0$ and  let for a $\ell$-tuple of $N\times N$ Hermitian matrices $\bX^{N}=(X_{1}^{N},\ldots, X_{\ell}^{N})$,
$$T(\bX^{N})=\inf\{ k: \frac{1}{{d} N}\Tr \sum_{j=1}^d (Z_j(\bX^{N}))^k\ge  L^k\}\,,$$ 
where we recall that $Z_{j}(\bX^{N})=\sum_{i\in I_{j}} (X_{j}^{N})^{2}$. 
Note that by  Jensen's inequality, $k\rightarrow  \frac{1}{d N}\Tr \sum_{j=1}^d L^{-k}(Z^N_j)^k$ is increasing. 
To prove the second inequality in \eqref{zx}, it is enough to show that for $\kappa\in (0,1)$ and  $N$ large enough
\begin{equation}\label{zxc}
\Q_V^N\left( \bX^{N}: T(\bX^{N})\le \frac{\kappa \ln N }{2\ln (2L)}\right)\le \frac{\kappa \ln N}{2\ln L} e^{-\frac{1}{2} N^{1-\kappa}}\,.
\end{equation}  
Indeed, observe that for any $M>0$, any integer number $k$, we have
$$\left\{\bX^{N}:\max_i\|X_i^N\|_\infty \ge M\right\}\subset \left\{\bX^{N}: \frac{1}{N}\Tr\left(   \sum_{j=1}^d (Z_j(\bX^{N}))^{k}\right)\ge \frac{1}{N} M^{2k}\right\}\,.$$
But on $\{T(\bX^{N})> \frac{\kappa\ln N}{2\ln 2L}\}$ we know that for every $k\le  \frac{\kappa\ln N}{2\ln 2L}$
$$\frac{1}{d N}\Tr\left(   \sum_{j=1}^d (Z_j(\bX^{N}))^{k}\right)\le L^k\,.$$
Therefore, on $\{\max_i\|X_i^N\|_\infty \ge M\}\cap \{T(\bX^{N})> \frac{\kappa\ln N}{2\ln 2L}\}$, we must have 
$d L^k\ge \frac{1}{N} M^{2k}$ for every  $k\le  \frac{\kappa\ln N}{2\ln 2L}$ . This is impossible unless
$$\ln \frac{M^2}{L}\le \inf_{k\le  \frac{\kappa\ln (N)}{2\ln 2L}} \frac{\ln (dN)}{k}=\frac{2}{\kappa}(1+\frac{\ln d}{\ln N}) \ln 2L,$$
for instance when $ M=(2L)^{\frac{1}{2}+\frac{1}{\kappa}}$ and $N\ge d^{2}$. Then,  we deduce that $\{T(\bX^{N})> \frac{\kappa\ln N}{2\ln 2L}\}\cap \left\{\max_i\|X_i^N\|_\infty \ge (2L)^{\frac{1}{2}+\frac{1}{\kappa}}\right\}$ is empty.
Therefore we conclude
$$\Q_V^N\left(\bX^{N}: \max_i\|X_i^N\|_\infty \ge (2L)^{\frac{1}{2}+\frac{1}{\kappa}}\right)\le \Q_V^N\left(\bX^{N}:T(\bX^{N})\le  \frac{\kappa \ln N}{2\ln 2L}\right)\,,$$
from which the theorem follows
from  \eqref{zxc}. To prove the latter, 
we first notice that
\begin{equation}\label{zxcv}
\Q_V^N\left(\bX^{N}: T(\bX^{N})\le \frac{\kappa \ln N}{2\ln 2L} \right)\le \sum_{k=0}^{\frac{\kappa \ln N}{2\ln 2L}} q_k^N
\end{equation}
where
$q_k^N=\Q_N^V\left( \bX^{N}\in S_{k,a}\right)$ if $S_{k,a}$ is the measurable set of  $N\times N$ Hermitian matrices $\bX^{N}=(X_{1}^{N},\ldots, X_{\ell}^{N})$ given by 
$$S_{k,a}\!=\!\left\lbrace  
\frac{1}{d N}\Tr \sum_{j=1}^d (Z_j(\bX^{N}))^{k}>a_{k}\right\rbrace\bigcap\!\bigcap_{1\leqslant u\leqslant k-1}\! \left\lbrace  \frac{1}{d N}\Tr \sum_{j=1}^d (Z_j(\bX^{N}))^u\leqslant a_{u}\right\rbrace$$
with  a sequence $(a_u)_{u\in \mathbb N}$ to be chosen later. 
We prove by induction over $k\le \frac{\kappa \ln N}{2\ln 2L} $ that we can find  non-negative real numbers $a_k, k\le \frac{\kappa \ln N}{2\ln 2L}$ such that for $N$ large enough
\begin{equation}\label{qw}q_k^N \le e^{-\frac{1}{2} N^{1-\kappa}}\end{equation}
and moreover that for every $k$, $a_k\le L^k$,  with $L$ as in the theorem. The same reasoning would give an upper bound of the expected order $e^{-c N}$ but only up to $k$ of order $\sqrt{\ln N}$ which is not sufficient to bound the spectral radius. 
Note that when $k=0$ the bound is satisfied  for any $a_0>1$.
Assume that up to rank $k$, we have found $a_0,\dots,a_k$  such that $a_u\le L^u$ for every $u\le k$ and $a_u$ verifies \eqref{qw}.
We next  construct $a_{k+1}$ so that the same property holds. 
We look at the dynamic of Lemma \ref{matrix-ODE} up to time $t_N=N^{-1-\kappa}$ on each set of variables indexed by $I_j$, namely for $i\in I_j$ we set, for $\bX^{N}=(X^{N}_{1},\ldots,X^{N}_{\ell})$ a $\ell$-tuple of Hermitain $N\times N$ matrices : 
$$
\left\{\begin{array}{ccl}
X^N_i(0)&=&X^N_i\\
\partial_{t}X_i^N(t) &=& -((Z_j^N(t))^kX_i^N(t)+X_i^N(t)(Z_j^N(t))^k)\\
Z^N_j(t) &=& \sum_{i\in I_j}(X^N_i(t))^2
\end{array}
\right.
$$
As proved in  Lemma \ref{change-variable-matrix} with $g_{i}(X)=((Z_j(X))^kX_i+X_i(Z_j(X))^k)$ for $i\in I_{j}$, we have 
\begin{align}
q^N_{k+1}&=
Q^N_V(\bX^{N}\in S_{k,a})\nonumber\\
&\leqslant \sup_{\bX^{N}\in S_{k,a}}\bigg\{\exp(-2\int_0^t \bigg(N\Tr\sum_j (Z^N_j(s))^k\sum_{i\in I_j} X^N_i(s).\Da_i V(X^N(s))\nonumber\\
&\qquad\qquad{-}\Tr\otimes\Tr\sum_{j, i\in I_j}\partial_i (X^N_i(s)(Z_j^N(s))^k)\bigg) ds)\bigg\}\label{tur}
\end{align}
We next bound the term in the exponential by using the hypothesis  that  $V$ is $(\eta,A,I)$-trapping which implies:
\begin{eqnarray}
N\Tr\hspace{-0.5cm}&& \sum_{j=1}^d(Z_j^N(s))^k\sum_{i\in I_j} X^N_i(s).\Da_i V(X^N(s)){-}\Tr\otimes\Tr\sum_{j, i\in I_j}\partial_i (X^N_i(s)(Z_j^N(s))^k)\nonumber \\
&&\qquad\qquad\geqslant  N \sum_{j=1}^d \Tr ( \eta  (Z^N_j(s))^{k+1}-A(1+Z^N_j(s)^k))\nonumber\\
&&\qquad\qquad{{-} }\sum_{j=1}^d \sum_{p=0}^k\Tr(Z^N_j(s))^p\Tr(Z^N_j(s))^{k-p}\label{bo1}\\
&&\qquad\qquad {{-} }\sum_{j=1}^d\sum_{i\in I_j}\sum_{p=0}^{k-1}\Tr X^N_i(s)(Z_j^N(s))^p\Tr X^N_i(s)(Z_j^N(s))^{k-1-p}\nonumber
\end{eqnarray}
Now observe that $\Tr\otimes\Tr ((X\otimes 1 - 1\otimes X)^2Z^a\otimes Z^b)$ is non-negative since
$(X\otimes 1 - 1\otimes X)^2$ and $Z^a\otimes Z^b$ are non-negative operators in $\HNC^{\otimes 2}$ for every integer numbers $a,b$. As a consequence, we find that
\begin{align*}
&\Tr (X^N_i(s))^2(Z_j^N(s))^p\Tr (Z^N(s))^{k-1-p}
\!+\! \Tr (Z_j^N(s))^p\Tr (X^N_i(s)^2(Z_j^N(s))^{k-1-p})\\
&\geqslant
2\Tr X^N_i(s)(Z_j^N(s))^p\Tr X^N_i(s)(Z_j^N(s))^{k-1-p}\,.
\end{align*}
Thus, we deduce the bound
\begin{align*}
&\sum_{i\in I_j}\sum_{p=0}^{k-1}\Tr X^N_i(s)(Z_j^N(s))^p\Tr X^N_i(s)(Z_j^N(s))^{k-1-p}\\&\leqslant \sum_{p=0}^k\Tr(Z_j^N(s))^p\Tr(Z_j^N(s))^{k-p}\,.
\end{align*}
%$$
%\sum_{i\in I_j}\sum_{p=0}^{k-1}\Tr X^N_i(s)(Z_j^N(s))^p\Tr X^N_i(s)(Z_j^N(s))^{k-1-p}%\leqslant \sum_{p=0}^k\Tr(Z_j^N(s))^p\Tr(Z_j^N(s))^{k-p}\,.$$
%{from which we deduce that the sum of the two last terms in \eqref{bo1} is non-negative. }
Now Lemma \ref{matrix-ODE} tells us that $s\rightarrow \Tr(Z_j^N(s))^n$ is decreasing on  $[0, \frac{1}{N^{\frac{k}{k+1}}}]$, so that we deduce if $t\le  \frac{1}{N^{\frac{k}{k+1}}}$, that for $s\le t$, 
\begin{eqnarray}
&&N\sum_{\substack{j\in [1,d]\\i\in I_j}
} \left(\Tr(Z^N_j(s))^kX^N_i(s)\Da_i V(X^N(s))-\Tr\otimes\Tr\partial_i (X^N_i(s)(Z_j^N(s))^k)\right)\nonumber\\
&&\qquad \geqslant   N  \sum_{j=1}^d \Tr (\eta  (Z^N_j(s))^{k+1}-A(1+(Z_j^N(0))^k))\label{bo2}\\
&&\qquad -2\sum_{p=0}^k\sum_{j=1}^d \Tr(Z^N_j(0))^p\Tr(Z_j^N(0))^{k-p}\nonumber
\end{eqnarray}
Finally, on the set $\{\frac{1}{N}\Tr \sum_{j=1}^d (Z_j^N(0))^{k+1}>a_{k+1}\}$ there is at least one $j$ such that $\frac{1}{N}\Tr  (Z_j^N(0))^{k+1}>a_{k+1}/d$, so that by Lemma \ref{matrix-ODE} for any $\kappa>0$,
 all $s\in [0,\frac{1}{N^{1+\kappa}}]$, we have
$$\frac{1}{N}\Tr \sum_{j=1}^d (Z^N_{j}(s))^{k+1}\geqslant \frac{a_{k+1}}{d}-\frac{4(k+1)}{N^{\kappa+\frac{1}{k+1}}} (\frac{a_{k+1}}{d})^{\frac{2k+1}{k+1}}$$
Therefore, on the set $S_{k,a}$, \eqref{tur} and  \eqref{bo1} give,  with  $t=t_{N}=N^{-1-\kappa}$, that
$q^N_{k+1}$ is bounded as follows:
\begin{equation}\label{ty}q^N_{k+1}\leqslant  e^{-N^{1-\kappa}F_{k+1}(a) }\end{equation}
with
$$F_{k+1}(a):=2
( \eta (\frac{a_{k+1}}{d}-4(k+1)N^{-\kappa-\frac{1}{k+1}}(\frac{a_{k+1}}{d})^\frac{2k+1}{k+1})-A(1+a_k)-2\sum_{p=0}^{k}a_pa_{k-p}).$$
We next  define {$a_0=1$} and by induction set
$$a_{k+1}=\frac{d}{\eta}(1+|A|(1+a_k)+2\sum_{p=0}^{k}a_pa_{k-p})\,.$$
We next  show by induction that for every integer numbers $k$, 
\begin{equation}\label{ind}a_k\le  C_k (\frac{4(1+|A|)d}{\eta})^k\end{equation} with $C_k$ the Catalan numbers.
 We assume without loss of generality that $ (1+|A|)d\ge \eta$ (since an $(\eta,A)$-trapping potential is also $(\eta',A)$ trapping for any $\eta'= (1+|A|)d<\eta$). \eqref{ind} is true for $k=0$. We assume by induction it  is true up to $k$,  then 
\begin{eqnarray*} a_{k+1}&\le & \frac{d}{\eta}\left(1+|A|(1+ 4^{k} C_k (\frac{(1+|A|)d}{\eta})^k )+\frac{1}{2}4^{k+1} C_{k+1} (\frac{(1+|A|)d}{\eta})^k \right)\\
&\le& C_{k+1} (\frac{4(1+|A|)d}{\eta})^{k+1} \frac{1}{(1+ |A|)}  \left( \frac{1+|A|}{  C_{k+1} (\frac{4(1+|A|)d}{\eta})^k} +\frac{|A| C_k}{4 C_{k+1}}
+\frac{1}{2}\right)\\
&\le&
  C_{k+1} \left(\frac{4(1+|A|)d}{\eta}\right)^{k+1}\end{eqnarray*}
 provided 
 \begin{equation}\label{tot}\frac{1}{(1+ |A|)}  \left( \frac{1+|A|}{ 4^{k+1} 
 C_{k+1} (\frac{(1+|A|)d}{\eta})^k} +\frac{|A| C_k}{4 C_{k+1}}
+\frac{1}{2}\right)\le 1.\end{equation}
 Since we assumed $(1+|A|)d\ge \eta$ the above inequality is clearly true for every $k\ge 0$.  Hence, the induction is satisfied. As a consequence, since $C_k\le 4^k$, we deduce that
 $a_{k}\leqslant 4^{2k} (\frac{(1+|A|)d}{\eta})^k=L^k$.
   Moreover, for such a choice of  $a_k, k\ge 0$, we can lower bound the exponent in \eqref{ty}:
\begin{equation}\label{turt}F_{k+1}(a) \ge 
1-2\eta \frac{4(k+1)L^{2k+1}}{N^{\kappa+\frac{1}{k+1}}}d^{\frac{2k+1}{k}}\,.\end{equation}
The right hand side is bounded from below by $1/2$ for $N$ large enough provided $k$ is small enough :
$$2k\le \kappa \frac{\ln N}{\ln (2L)}\,.$$
 From \eqref{turt},\eqref{ty} and \eqref{zxcv},  \eqref{zxc} and \eqref{zx} follow readily.

\section{Asymptotic localization at low temperature}

In this section we use Schwinger Dyson's equation to prove in the general case asymptotic localization properties at low temperature.

Let us consider a potential $V$ which is $(\eta,A,I)$-trapping for some $\eta>0$, a finite positive real number $A$ and a partition $I$ of $\{1,\ldots,\ell\}$. By Corollary \ref{CorTight}, we know that  the empirical distribution $\mun$ is tight almost surely under $\Q^{N}_{V}$. Let $\tau_{V}$ be an almost sure limit point.
\begin{theorem}\label{theoDS}(Dyson-Schwinger's equations)
	Under the assumptions of Theorem \ref{Ttosupport}, any limit point $\tau_{V}$ of $\mun$ under $\Q_{V}^{N}$ satisfies the Dyson-Schwinger's equations: for any polynomial $P$ in $\mathbb C\langle X_{1},\ldots,X_{\ell}\rangle$ , any $i\in \{1,\ldots,\ell\}$, 
	\begin{equation}\label{ds}\tau_{V}\left(\mathcal D_{i} V P\right)=\tau_{V}\otimes\tau_{V}\left(\partial_{i}P\right)\,.\end{equation}
\end{theorem}
When $V$ is a small perturbation of a quadratic polynomial, it was shown in \cite{GM} that the Dyson-Schwinger's equations have a unique bounded solution. This is not true in general, even in the one-matrix case. Indeed, if $V$ has two wells, the limit of the empirical measure of the eigenvalues of the matrix may have  disconnected support and there is a solution to the Dyson-Schwinger equation for any choice of masses in the different connected pieces of the support \cite{BoG2}.
\begin{proof}{
	We use the change of variables of Lemma \ref{change-variable-matrix}so that $\partial_{t}(X^{N}_{j})_{t}=-g_{j}(X^{N}_{t})$ with
	$g_j(X)=1_{j=i}P(X^{\varepsilon})$  for a  non-commutative polynomial $P$, and  $X^{\varepsilon}=(f_\varepsilon(X_1),\ldots, f_\varepsilon(X_\ell))$ with $f_\varepsilon (x)=(\varepsilon+x^{2})^{-1/2} x$ for some $\varepsilon>0$.  
	$g$ is  uniformly bounded (with a uniform bound depending on $\varepsilon$) so that for $t\le \frac{1}{N}$, for every $j\in \{1,\ldots,\ell\}$
	\begin{equation}\label{exp1} X^{N}_{j}({t})=X^{N}_{j}+O_\varepsilon(\frac{1}{N}).\end{equation}
	Let us set  for some $\delta>0$, 
	\begin{align*}
	A_{\delta}=\bigcap_{i}\left\{\bX^{N}\! \right.& \in\! (\mathcal H^N(\mathbb C))^\ell: \\
	&\left.\frac{1}{N}\Tr(\mathcal D_{i} V (\bX^{N})g_{i}(\bX^{N}) )\!-\!\frac{1}{N}\Tr \otimes\frac{1}{N}\Tr(\partial_{i} g_{i}(\bX^{N}))\ge \delta\right\}.
	\end{align*}
	It is easy to see that since $g_{i}$ is uniformly Lipschitz, by \eqref{exp1}, we have uniformly 
	on $A_\delta$,  
		$$\int_0^{\frac{1}{N}}(N\Tr \Da_{i} V(X^N_s)  g(X^N_s)+\Tr\otimes \Tr \partial_{i} g(X^N_s))ds\ge N\delta +O_{\varepsilon}(1)\,.$$
		Hence, Lemma \ref{change-variable-matrix} implies that
	$$\Q^{N}_{V}(A_{\delta})\le e^{-N\delta+ O_\varepsilon(1)}\,.$$
	The same holds when we replace $g$ by $-g$ and therefore
	almost surely by Borel Cantelli's lemma 
	$$\limsup_{N\rightarrow  \infty  }\left|  \frac{1}{N}\Tr(\mathcal D_{i} V.g)-\frac{1}{N}\Tr \otimes\frac{1}{N}\Tr(\partial_{i} g)\right|=0\,.$$
	On the other hand,  put $B_D=\{\bX^{N}\in (\mathcal H^N(\mathbb C))^\ell: \frac{1}{N}\Tr( X_i^{2D})\le L^{2D}\}$ for $L$ chosen as in Theorem \ref{Ttosupport}, with $2D$ larger or equal than the sum of the degree of $V$ and  the maximum degree of the $P_{i}$'s. Theorem \ref{Ttosupport} implies that $\Q^{N}_{V}(B_D^c)\le e^{-N/^{1-\kappa}}$ for $\kappa>0$ and  $N$ large enough. On the other hand, on $B_D$, it is not hard to see that
	$$\left|  \frac{1}{N}\Tr(\mathcal D V.(g-P))|\le O_L(\varepsilon), \mbox{ and }\,  |\frac{1}{N}\Tr \otimes\frac{1}{N}\Tr(\partial (g-P))\right|\le O_L(\varepsilon)\,,$$
	with $ O_L(\varepsilon)$ going to zero uniformly on $N$ as $\varepsilon$ goes to zero. As a consequence, 
	almost surely, we have 
	$$\limsup_{N\rightarrow  \infty  }\left|  \frac{1}{N}\Tr(\mathcal D_{i}V.P)-\frac{1}{N}\Tr \otimes\frac{1}{N}\Tr(\partial_{i} P)\right|=0\,.$$
	This implies that any almost sure limit point of $\mun$ satisfies the limiting Dyson-Schwinger's equations \eqref{ds}.}
	
\end{proof}

We next consider  general potentials in the  low temperature situation where $\beta=1/T$ is large and 
$$V_{\beta}=\beta U+W$$
with $U,W$ self-adjoint polynomials independent of $\beta$.  
%We will assume that the empirical distribution of the matrices under $\Q_{V_{\beta}}^{N}$ is tight as in 

We will first  study the limit points of $\mun$ under $\Q_{V_\beta}^N$  and   show  that they concentrate near  the critical points of $U$, namely the points where $\mathcal D_{i}U$ is small, when $\beta$ is large enough.
Theorem \ref{Ttosupport} 
allows us to localize the state we are looking at on a fixed support. 
Our result takes the following form:

\begin{prop}[{\bf Control of the support, Low temperature}]\label{proploc}

Assume that we are given $L<\infty ,\beta_0>0$ and for every $\beta>\beta_0$ a  non-commutative distribution $\nu_\beta$ 
 solution of the Dyson-Schwinger equation \eqref{ds} with  potential $V_\beta=\beta U+W$ and with support in $[-L;L]$ in the sense that for every $i\in[1,\ell]$, all $k\in\mathbb N$
 $$\nu_\beta(X_i^{2k})\le L^{2k}\,.$$
Then there exists $B>0$ such that for $\beta>\beta_0$,  for every integer number $k$
$$\max_{{1\le i\le \ell}}\nu_\beta\left(\left|\mathcal D_{i} U\right|^{2k}\right)\le \left(\frac{B}{\beta}\right)^{k}\,.$$
Moreover one can take $B=4\max_i\{\|\mathcal D_i U\mathcal D_i W\|_L+ 2\|\partial_i \mathcal D_i U\|_L\}$ where  $\|P\|_L=\sup_{\max_i\|x_i\|\le L}\|P(x)\|$  if  the supremum is taken over all $C^*$-algebras and operators $\{x_1,\ldots,x_\ell\}$ with norm bounded by $L$, and we set $\|\sum_{i} P_i\otimes Q_i\|_L=\sum_i\|P_i\|_L\|Q_i\|_L$. 
\end{prop}

\begin{proof}
	According to the Dyson-Schwinger equation, \eqref{ds} in Theorem \ref{theoDS}, applied to $P_{i}=(\Da_i U)^{2k-1}$,
	
	\begin{align*}
		&\nu_\beta( (\Da_i U)^{2k-1}( \Da_i U+\frac{1}{\beta} \Da_i W))=\beta^{-1}\nu_\beta\otimes\nu_\beta\left( \partial_i (\Da_i U)^{2k-1}\right)\\
		&\qquad=\beta^{-1}\sum_{j=1}^{2k-1}\nu_\beta\otimes \nu_\beta \left(
		((\Da_i U)^{j-1}\otimes 1)(\partial_i \Da_i U)(1\otimes(\Da_i U)^{2k-1-j})\right)
	\end{align*}
	We use this equality to prove by induction that for every $p\ge 0$, with $B\ge 4\max_i\{\|\mathcal D_i U\mathcal D_i W\|_L+2 \|\partial_i \mathcal D_i U\|_L\}$ :
	
	\begin{equation}\label{ines}\nu_\beta((\Da_i U)^{2p})\leqslant C_p(\frac{B}{4\beta})^p\end{equation}
	where $C_{p}$ denotes the Catalan numbers. The result follows since $C_p\le 4^p$. 
	The property is obviously satisfied for $p=0$. Assume that it is true up to $p=k-1$. This implies that for any $j\leqslant 2k-2$, for any  polynomial $P$
	
	$$|\nu_\beta((\Da_i U)^{j}P)|\!\leqslant\! \|P\|_L \nu_\beta(|\Da_i U|^j)
	\!\leqslant \!\|P\|_L \nu_\beta(|\Da_i U|^{2\lceil \frac{j}{2}\rceil})^{\frac{j}{2\lceil \frac{j}{2}\rceil}}\!\leqslant\!  \|P\|_L C_{\lceil \frac{j}{2}\rceil}(\frac{B}{4\beta})^{\frac{j}{2}}\,.$$
	Moreover
	$$|\nu_\beta((\Da_i U)^{2k-1}\Da_i W)|\le \|\Da_{i}U\Da_{i}W\|_{L}\nu_\beta(|\Da_i U|^{2k-2})$$
	Plugging this estimate in the Dyson-Schwinger equation and using the induction hypothesis we find that  :
	\begin{align*}
		&\left((\frac{B}{4\beta})^k C_k\right)^{-1} \nu_\beta((\Da_i U)^{2k})\\
		&\leq 
		\frac{4C_{k-1}}{BC_k}
		\|\Da_{i}U\Da_i W\|_{L} 
		+ \frac{4}{BC_k}\|\partial_i \mathcal D_i U\|_L
		\sum_{j=1}^{2k-1}C_{\lceil \frac{j-1}{2}\rceil}C_{\lceil \frac{2k-1-j}{2}\rceil}\\
		&\leqslant \frac{4}{B}\left(\max_i\{\|\mathcal D_i U\mathcal D_i W\|_L+2 \|\partial_i \mathcal D_i U\|_L\}\right).
	\end{align*}
	 Hence, the induction hypothesis is satisfied as soon as  $B$ is greater than $4\max_i\{\|\mathcal D_i U\mathcal D_i W\|_L+2 \|\partial_i \mathcal D_i U\|_L\}$.
	
\end{proof}

\begin{remark}
Note that  if $U$ is homogenous of degree $d$ and  cyclically invariant,
$$U(\bX)= \frac{1}{d} \sum_i X_i \Da_i U\,.$$
\end{remark}
Therefore if we have a bound on the $\nu_\beta(X_i^{2k})$'s, the bound on $\nu_\beta((\Da_i U)^{2k})$ from the previous Lemma implies a control on $\nu_\beta( U^{2k})$. This is the content of the next Corollary:
\begin{corollary}\label{homo}
With the same hypothesis and notation than in Proposition \ref{proploc}, and if furthermore $U$ is homogenous of degree $d$  (in the sense that all its monomials have the same total degree) and cyclically invariant, then
 for $\beta>\beta_0$,  for every integer number $k$,
$$\nu_\beta\left( U^{2k}\right)\le \left(\frac{L^{2}B}{\beta}\right)^{k}$$
with $B=4\max_i\{\|\mathcal D_i U\mathcal D_i W\|_L+2 \|\partial_i \mathcal D_i U\|_L\}$.
\end{corollary}
In the case above, we see that  under $\nu_\beta$, $U$ converges to zero. In a more general situation, we have only the following concentration of the mean of $U$.

\begin{prop}\label{bound-first-trace}
Let  $V_\beta=\beta U+W$ be a self-adjoint potential such that,  for some $\beta_0>0$, the partition function \eqref{defZ} is finite :
\begin{equation}\label{ass0}
\limsup_{N\rightarrow\infty} \frac{1}{N^2}\ln Z^N_{V_{\beta_0}}< +\infty\end{equation}
 and there exist {real } numbers $\alpha_1\dots,\alpha_\ell$  such that
$$m:=\inf_{\bX\in \HNC^\ell}\frac{1}{N}\Tr U(\bX)=\frac{1}{N}\Tr(U(\alpha_1\Id,\dots,\alpha_\ell\Id))=U(\alpha_1,\ldots,\alpha_\ell).$$
Then, there exists a finite constant $C$ such that for $\beta>\beta_0$ and any non-negative real number $M$,
$$\limsup_{N\ra\infty}\frac{1}{N^{2}}\ln \Q^{N}_{V}\left(\frac{1}{N}\Tr  U(\bX^{N})\ge m+\frac{M+\ell\ln\beta}{\beta}\right)<C-M\,.$$

As a consequence,
if we assume furthermore that we have  an accumulation point $\tau_\beta$ of the empirical distribution for every $\beta$ sufficiently large. Then any accumulation point $\tau$ of  $\tau_\beta$ when $\beta\to\infty$ is such that $\tau(U)=m$.
\end{prop}
\begin{proof}
Let $2D$ be the degree of $U$.
We can assume without loss of generality, that up to some translation and the addition of a constant term,  that $\alpha_1=\dots=\alpha_\ell=0$ and that $m=\Tr(U(\alpha_1\Id,\dots,\alpha_\ell\Id))=0$.

\noindent
{\bf First step : upper bound on $U$.}
Since $\Tr U(\bX)$, as a function of the coefficients of the matrices $\bX$, is a polynomial function  that reaches its local minimum in $0$  and whose value is $0$, then necessarily $\Tr U$ has no constant term and no terms of degree $1$. Therefore $U$ is a sum of monomials of degree at least $2$. Let $q=X_{i_1}\dots X_{i_k}$ be a monomial of degree $k\leqslant 2D$ in $U$, and observe that by non-commutative H\"older's inequality:
$$|\Tr q(X)|\leqslant( (\Tr |X_{i_1}|^k)\dots (\Tr |X_{i_k}|^k))^{\frac 1 k}\leqslant \max_i \Tr |X_i|^k\leqslant \sum_i \Tr(X_i^2+X_i^{2D})$$
since  for every $2\leqslant k \leqslant 2D$, $|x|^k\leqslant x^2+x^{2D}$.
By combining such inequalities we deduce that there is some finite constant $0<A<\infty$(which only depend on $U$) such that for every $\bX\in (\HNC)^\ell$:
$$0\leqslant \Tr U(\bX)\leqslant A\Tr \sum_i \Tr(X_i^2+X_i^{2D}).$$

\noindent
{\bf Second step : lower  bound on the partition function $Z^N_{V_\beta}$.}
First make the change of variables  $\sqrt\beta X_i = Y_i$ and then use Jensen's inequality:
\begin{align*}
	\frac{1}{N^2}\ln Z_{V_\beta}^{N}&\geqslant \frac{1}{N^2}\ln\int e^{-N\Tr W}e^{-N\beta A \Tr \sum_i \Tr(X_i^2+X_i^{2D})}dX_i\\
	&\geqslant -\ell\ln\beta+\frac{1}{N^2}\ln\int e^{-N\Tr W(\bY/\sqrt{\beta})}\prod_{i}e^{-N A \Tr \Tr(Y_i^2+\beta^{1-D}Y_i^{2D})}dY_i\\
	&\geqslant -\ell\ln\beta +\frac{1}{N^2}\ln Z^N_{A \sum_i Y_i^2}\\
	&-\int \left( \frac{1}{N}\Tr W(\frac{\bY}{\sqrt{\beta}})+ A\beta^{1-D}\sum_{i}\frac{1}{N}\Tr(Y_i^{2D})\right)
	\prod_{i }\frac{e^{-N A \Tr(Y_i^2)}dY_i}
	{Z^N_{A Y^2}}
	\\
	&\geqslant -\ell\ln\beta + O(1)
\end{align*}
where we used that  moments of Gaussian matrices are bounded, as well as the partition function which was computed by Selberg \cite{AndersonSelberg, Selberg}. Here,
the term $O(1)$ is uniform in $\beta>0$.

\nn
{\bf Third step : conclusion.}
Finally, by Tchebychev's inequality we find when $\beta\ge \beta_0$

\begin{align*}
	&\frac{1}{N^2}\ln\Q^{N}_{V_\beta}\left(\frac{1}{N}\Tr  U(\bX^{N})\ge \frac{\ell\ln\beta+M}{\beta}\right)\\
	&\leqslant \frac{1}{N^2}\ln\frac{Z^N_{V_{\beta_0}}}{Z^N_{V_\beta}} \int e^{-N^2(\beta-\beta_0)  \frac{\ell\ln\beta+M}{\beta}}d \Q^{N}_{V_{\beta_0}}\\
	&\leqslant O(1)+\ell\ln\beta-(\beta-\beta_0) \frac{\ell\ln\beta+M}{\beta}
	=O(1)-M
\end{align*}
where we used that  $\frac{1}{N^2}\ln\frac{Z^N_{V_{\beta_0}}}{Z^N_{V_\beta}}\le \ell\ln\beta + O(1)$ by the second step  and \eqref{ass0}.
\end{proof}

\section{Strong uniformly convex  potentials}\label{convex}

The three remaining sections are devoted to the study of particular cases of low temperature asymptotics.
We begin by the simplest case as in a way it can be converted to the well known high temperature regime. Indeed, in this section we look at a $\ell$-uple of random $N\times N$ Hermitian matrices $\bX^N=(X^N_1,\dots,X^N_\ell)$ under the law :

\begin{equation}\label{gibbs}\Q^N_{V_{\beta}}(d\bX^N)=\frac{1}{Z^\beta_N}\exp\left(-N \Tr V_\beta(\bX^N)\right)d\bX^N\end{equation}
where $V_\beta$ is a  self-adjoint polynomial that satisfy the following convexity assumption.
\begin{assum}\label{assum1well}Let $C$ be a real number.
We say that a self-adjoint polynomial $V\in \mathbb C\langle X_{1},\ldots, X_{\ell}\rangle$ is $C$-convex if for any dimension $N$, 
the Hessian of the function $\varphi_{V}(\bX^{N})= \tr V(\bX^N)$ on the set 
of $\ell$-tuple $\bX^{N}$ of $\ell$ $N\times N$ Hermitian matrices is bounded from below uniformly  :
$${\rm{ Hess }}(\varphi_{V}(\bX^N))\ge C.Id$$
for any $\ell$-tuple of $N\times N$ Hermitian matrices $\bX^{N}$ and any integer number $N$.
Then we assume that
$$V_{\beta}(\bX)=\beta U(\bX)+W(\bX)$$
where $U$ and $W$ are self-adjoint polynomials such that:
\begin{itemize}
\item 
$U$ is  $c$-convex for some $c>0$ (therefore $\bX:(\mathcal H_{N}(\mathbb C))^{\ell}\mapsto \Tr U(\bX)$ is strictly convex),
\item 
$W$ is $A$-convex for some $A$ finite (but possibly negative). 
\end{itemize}\end{assum}
It is not hard to check the following example.
\begin{example}
If  $U(\bX) = \sum_i P_i(X_i)$ with polynomials $P_{i}$ with real coefficient such that $P_{i}''(x)\ge c>0$ for every $x \in \mathbb R$ and
$W(X)=\tilde{W}(X)+\sum_i a_iX_i^{2D}$ for an integer number $D\ge 1$, some polynomial $\tilde W$ with degree strictly smaller than $2D$ and positive
 real numbers $(a_i)_{1\le i\le \ell}$,  then $(U,W)$ satisfies our hypothesis with constant $c$ and some finite $A$. \end{example}

We discuss in this section the convergence of the non-commutative distribution for $\beta$ large enough.
\begin{theorem}\label{thmconv} Under Assumption  \ref{assum1well},  there exists  a finite positive constant
	$\beta_0$ such that for $\beta>\beta_0$, 
	 the non-commutative distribution of $\bX^{N}$ converges almost surely under $\Q_{V_\beta}^{N}$  towards a non-commutative law $\tau_{V_\beta}$: for any polynomial $P\in\mathbb C\langle X_{1},\ldots,X_{\ell}\rangle$
	$$\lim_{N\ra\infty}\frac{1}{N}\Tr\left(P( \bX^{N})
	\right)=\tau_{V_\beta}(P)\qquad a.s\,.$$
	
	\end{theorem}
	This theorem is in fact already proven in \cite{GS08,dabrowski,Jekel} as for $\beta$ large enough $V_{\beta}$ is convex. However, we can also 
	 characterize the limit $\tau_{V_\beta}$. This is due to the fact that up to rescaling our model can be identified with a high temperature model as studied in \cite{GM}.
Observe that $\varphi_{U}$ achieves its minimum value at a unique point in any dimension by uniform strict convexity. In fact we can prove that these minimizers are characterized as follows.
\begin{lemma}Assume  that $U$ is $c$-convex with $c>0$. Let ${\bf \alpha}=(\alpha_{1},\ldots,\alpha_{\ell})\in \mathbb R^{\ell}$ be the unique minimizer of $U$ on $\mathbb R^{\ell}$. 
Then, for any integer number $N$, $\varphi_{U}(\bX^{N})$ is minimized on the set of $\ell$-tuples of  $N\times N$ Hermitian matrices at the point  $(\alpha_{1}I_{N},\ldots,\alpha_{\ell}I_{N})$ where $I_{N}$ is the identity in the space of $N\times N$ matrices.
\end{lemma}
\begin{proof}
	We first fix the dimension $N$. Then, by strict convexity we know  that $\varphi_{U}$ is minimized at a unique $\ell$-tuple $(A_{1}^{N},\ldots, A_{\ell}^{N})$ of Hermitian matrices. But by invariance of $\varphi_{U}$ by the change of variables $(X_{1},\ldots,X_{\ell})\rightarrow (OX_{1}O^{*}, \ldots, OX_{\ell}O^{*})$ for any unitary matrix $O$, and uniqueness of the minimizer, we see that for every $i\in [1,\ell]$, $OA^{N}_{i}O^{*}=A_{i}^{N}$. This implies that $A^{N}_{i}$ is a multiple of identity, $A^{N}_{i}=\alpha_{i}^{N} I$. But then $\varphi_{U}(A_{1}^{N}, \cdots, A_{\ell}^{N})=NV(\alpha_{1}^{N}, \cdots,\alpha_{\ell}^{N})$ so that $(\alpha_{1}^{N},\ldots,\alpha_{\ell}^{N})$ minimizes $U$ and therefore equals $(\alpha_{1},\ldots,\alpha_{\ell})$.
\end{proof}

To prove Theorem \ref{thmconv}, we  first do a change of variables related with
the matrix
\begin{equation}\label{cov} C_{U}(ij)=\partial_{x_{i}}\partial_{x_{j}} U(x_{1},\ldots,x_{\ell})|_{x=\alpha}\ge c. Id \end{equation}
In the following, for any $\ell$-tuple of self-adjoint operators ${\bf X}=(X_{1},\ldots,X_{\ell})$, and any matrix $M\in \mathcal H^{\ell}(\mathbb C)$, we denote by $M{\bf X}$ the $\ell$-tuple of self-adjoint operators $(\sum_{j=1}^{\ell}M_{ij}X_{j})_{1\le i\le \ell}$.  We let 
$${\bf Y}=\sqrt{\beta} C_{U}^{1/2}({\bf X}-{\bf \alpha} I)\,.$$ In these variables, let us write the decomposition 
	\begin{equation}\label{rew}\beta  U(\bX)+W(\bX)=\sum_{i=1}^{\ell} Y_{i}^{2}+\sum_{i=1}^{p} \frac{c_{i}}{\sqrt{\beta}^{n_{i}}} q_{i}(\bY)\end{equation}{{
where the $n_{i}$ are integers numbers greater or equal to one  and the $q_{i}$'s are monomials coming either from $W$ (and then  with degree $n_{i}$) or from $U$ (and then with degree $n_{i}+2$). }}

	We next characterize the distribution of ${\bf Y}=\sqrt{\beta} C_{U}^{1/2}({\bf X}-{\bf \alpha} I)$ under $\tau_{V_\beta}$ as a generating function for the enumeration of  planar maps. 
	Recall \cite{GM}  that a star of type $q(X)=X_{i_1}\cdots X_{i_r}$ is a vertex
  embedded into the plane with $r$ colored half-edges rooted at the first half-edge of color $i_1$, second of color $i_2$, until the last one of color $i_r$ (following the orientation of the plane). A planar 
map build over $k_i$ stars of type $q_i$, $1\le i\le p$, is a matching of the half-edges of the same color resulting into a connected graph with genus zero. 
$M((k_i,q_i)_{1\le i\le p})$ denotes the number of planar maps build over $k_i$ stars of type $q_i$, $1\le i\le p$, counted with labelled half-edges.  If $M((k_{i},q_{i})_{1\le i\le p})$ denotes the number of planar maps build over $k_{i}$ stars of type $q_{i}$, we have

\begin{prop}\label{expprop}  Under Hypothesis \ref{assum1well}, with ${\bf \alpha}=(\alpha_{1},\ldots,\alpha_{\ell})\in \mathbb R^{\ell}$ be the unique minimizer of $U$ on $\mathbb R^{\ell}$ and $C_{U}$ given by \eqref{cov}, 
 there exists  a finite positive constant
	$\beta_0$ such that for $\beta>\beta_0$,  
	such that	for any monomial $q$
	$$\tau_{V_{\beta}}(q (\sqrt{\beta} C^{1/2}_{U} (\bX-\alpha I)))=\sum_{}\prod_{1\le i\le p}\frac{(c_{i}\sqrt{\beta}^{-n_{i}})^{k_{i}}}{k_{i}!}M((1,q), (k_{i},q_{i})_{1\le i\le p})\,.$$
\end{prop}
Proposition \ref{expprop} implies Theorem \ref{thmconv}.
We can also use \cite{GS} to see that $\tau_{\beta}$
is a push-forward of free semi-circle laws by smooth functions (namely series which are uniformly converging when evaluated at self-adjoint operators with norm bounded by $3$). Hence we deduce that
\begin{corollary}\label{corconv}
	Under the assumptions of Theorem \ref{thmconv},  $\tau_{\beta,U,W}$ is described as follows. If  $S_{1},\ldots,S_{d}$ are  $\ell$ free semi-circular variables, $\tau_{V_{\beta}}$ is the non-commutative distribution  of
	$$X_{i}=\alpha_{i} I+\frac{1}{\sqrt{\beta}}S_{i}+\frac{1}{\beta} F_{i}(S_{1},\ldots,S_{\ell}), 1\le i\le \ell,$$
	where
	$F_{1},\ldots, F_{\ell}$ are functions which expand as converging power series in $1/\sqrt{\beta}$:
	$$F_{i}(X_{1},\ldots,X_{\ell})=\sum_{n\ge 0} \frac{1}{\sqrt{\beta}^n}
	P_{n}(X_{1},\ldots,X_{\ell})$$
	with non commutative polynomials $P_{n}$ such that $\sum_{\ell,n} \sqrt{\beta}^{-n}\|P_{n}\|$ is finite. Here, the operator norm is taken uniformly over all $\ell$-tuple of self-adjoint operators bounded by $3$.  In particular the $C^{*}$ and the von Neumann algebra of $\tau_{V_{\beta}}$ is the same than those of free semi-circular provided $\beta$ is finite and large.
\end{corollary}

\section{Strong single variables potentials}\label{single}
In this section we study more closely the model where the 
strong potential constrains the eigenvalues of $p\le\ell$ matrices, but this constraint is not convex.
 More precisely we assume:
\begin{assum}\label{assumsingle}
$$V_{\beta}(\bX)= \beta (U^{1}(X_1)+\dots+U^{p}(X_p))+W(\bX)$$
where $p\le \ell$, $U^{i}$  and  $W$ are self-adjoint polynomials such that:
\begin{itemize}
\item  There exists $\beta_0>0$, $\eta>0$, $A<\infty$ and a partition $I$ so that  for every $\beta\ge \beta_0$, 
$V_\beta$ is $(\eta,A,I)$-trapping. 
\item For $i\in\{1,\ldots, p\}$, $U^i$ achieves its minimal value 
at  real numbers  $x_1^i,\dots,x_{m_i}^{i}$ with $U^i(x)
\sim_{x_j} c_j^i(x-x_j^i)^{2k_j^i}$ when $x$ belongs to a small neighborhood of $x_j^i$, for some positive constant $c_j^i$ and integer numbers $k_j^i$. 
\end{itemize}\end{assum}
%{\bf Question:could we assume $V_{0}^{i}$ $C^{2}$ rater than polynomial}

 We then prove that the spectrum of the matrix $X_i^N$  following $\Q^{N}_{V_{\beta}}$ concentrates in a neighborhood of the minimizers of $U^i$ for $i\in \{1,\ldots,p\}$. To state our result, we will denote $\hat\mu^N_{X}=N^{-1}\sum_{i=1}^N \delta_{\lambda_i(X)}$ the empirical measure of the eigenvalues of a $N\times N$ normal matrix $X$. Moreover, we recall that $\|X_{i}\|_{\infty}$ denotes the operator norm of the matrix $X_{i}$ and  we set, for $\bX=(X_{1},\ldots,X_{\ell})$, $\|\bX\|_{\infty}=\max_{1\le i\le \ell}\|X_i\|_\infty$ .
\begin{theorem}\label{theosum} Assume $V_{\beta}$ satisfies Assumption \ref{assumsingle}. 
For a positive real number $R$,  we denote by $$K_{\beta,R}^i=\cup_{1\le j\le m_i} [x_j^i- R(\frac{\beta}{\ln\beta})^{-\frac{1}{2k_j^i}}, x_j^i+R(\frac{\beta}{\ln\beta})^{-\frac{1}{2k_j^i}}]\,.$$
Then,   for every $\kappa\in (0,1)$, there exists $R>0$ and $\beta_0>0$,% $\eta>0$ and $A<\infty$ 
so that  for every $\beta\ge \beta_0$, 
 for  all $i\in\{1,\ldots,p\}$, 
\begin{equation}\label{c1}
\limsup_N \frac{1}{N^{1-\kappa}}\ln \Q^N_{V_\beta}\left(\hat\mu^N_{X_i^N}((K_{\beta,R}^i)^c)\neq 0\right)<0.\end{equation}
Besides,    if  $K_i = \sum_j k_j^i$,  there exists $\beta_{0}$ finite such that for $\beta\ge \beta_{0}$, for every $j\in\{1,\ldots m_{i}\}$, we have almost surely:
$$\limsup_N \sqrt{\ln{\beta}}|\hat \mu^N_{X_i^N}([x_j^i-\max_{j}R(\frac{\beta}{\ln\beta})^{-\frac{1}{2k_j^i}},x_j^i+\max_{j}R(\frac{\beta}{\ln\beta})^{-\frac{1}{2k_j^i}}]) - \frac{k_j^i}{K_i}|<\infty\,.$$
\end{theorem}
%Recall from Corollary \ref{tight} that $\hat\mu^N$ denotes the empirical distribution of the $\ell$ t-uple $\bX^{N}=(X_1^N,\ldots,X_\ell^N)$ and $\hat\mu^N_{X_i^N}$ denotes its restriction to functions depending only on $X_i^N$. More generally, for a non-commutative law $\tau$ of $\ell$ variables $(X_1,\ldots,X_\ell)$, we denote in short $\tau_{X_i}$ its restriction to the algebra generated by $X_i$. 

To describe more precisely the asymptotic  distribution of our matrix model, recall \cite{BAG} that in the case of  a single matrix model   and a potential $V$ (corresponding to $\ell=1$), the empirical measure of the eigenvalues $\hat\mu^{N}_{X}$ converges almost surely to  the unique minimizer $\mu_{V}$ of 
\begin{equation}\label{energy}E_{V}(\mu)=\frac{1}{2}  \int\int \left(V(x) +V(y)-\frac{1}{2}\ln|x-y|\right)d\mu(x)d\mu(y)\,.\end{equation}

 Such a convergence is also true in the multi-matrix model satisfying Hypothesis \ref{assumsingle} in the limit where $\beta$ goes to infinity in the following sense.  For $r>0$ and an integer number $k$,  let $\sigma^{k}_{r}=\mu_{r x^{2k}}$ be the  unique probability measure minimizing $E_{rx^{2k}}$. Observe that when $k=1$, $\sigma^{k}_r$ is the semi-circle law with variance { $r^{-1}$. }

\begin{corollary}\label{corsum} Assume $V_{\beta}$ satisfies Assumption \ref{assumsingle}. Then there exists $\beta_0$ finite such that for $\beta\ge\beta_0$ the empirical distribution $\hat\mu^N$ of $\bX^{N}=(X_1^N,\ldots,X_\ell^N)$ is tight almost surely. Denote  $\nu^\beta$  a limit point.  Then, there exists $R\in (0,+\infty)$ such that for $\beta$ large enough, for any $i\in \{1,\ldots,p\}$  
$$\nu^\beta_{X_i} (K_{\beta,R}^i)=1\,.$$
Moreover,   the filling fraction 
$\alpha^i_j(\beta):=\nu_{X_i}^\beta([x_j^i- R(\frac{\beta}{\ln\beta})^{-\frac{1}{2k_j^i}}, x_j^i+R(\frac{\beta}{\ln\beta})^{-\frac{1}{2k_j^i}}])$
satisfies 
$$|\alpha^i_j(\beta)-\frac{k_j^i}{K_i}|\leqslant \frac{R}{\sqrt{\ln{\beta}}}\,.$$
Finally, the distribution of $X_i$  under $\nu_\beta$ is close to $\sigma^{k_{j}^{i}}_{r^\beta_{i,j}}$ near $x_j^i$:  for any smooth compactly supported function $f$
 $$\left|\alpha^i_j(\beta)^{-1}\int \left(f( \beta^{\frac{1}{2k_j^i}}(x-x_j^i))\right)d \nu^\beta_{X_i}
(x)-\int f(x) d \sigma^{k_j^i}_{r^\beta_{i,j}}(x) \right|\leqslant \frac{\|f\|_{1/2}}{\min_j \beta^{\frac{1}{2k_j^i}}}$$
where 
$r^\beta_{i,j}=\frac{\alpha^i_j(\beta)}{2c_j^i}$ and $\|f\|_{1/2}=(\int_{0}^{\infty}t|\hat f_t|^2 dt)^{1/2}$.
\end{corollary}

The proof of   Theorem \ref{theosum} follows again from  a wise change of variables and Lemma \ref{change-variable-matrix}. We first construct this change of variables in the next subsection: it is based on the explicit joint law of the eigenvalues of a one dimensional matrix model and change of variables involving the eigenvalues of a single matrix. Indeed,
the interaction between the matrices being of lower strength, such change of variables will be sufficient for our purpose. 

\subsection{Proofs of Theorem  \ref{theosum} and  Corollary \ref{corsum}}

To simplify the notations we will take $p=1$, $k^{1}_{j}=k_{j}$, $x^{1}_{j}=x_{j}$. The  proof extends readily to the general case.

{\bf{Proof of Theorem \ref{theosum}} } We denote by $\{x_1,\ldots,x_k\}$ the minimizers of $U^1$.
Without loss of generality we may assume that the minimum value of $U^1$ is zero and it is achieved at the origin so that $U^1(0)=0$. We set for $\eta>0$
$$I = \{x\in\mathbb R :U^1(x)\in [0,\eta]\}$$ 
to be  a neighborhood of the minimizers of $U^1$. $\eta$ is  a constant depending on $\beta$ that we will choose later. We let  
$\varepsilon>0$ be
such that $U^1(x)<\eta/2$ for $|x|<\varepsilon$. We shall prove the theorem by sending all eigenvalues outside of $I$ in $[-\vep,\vep]$. 
To this end we first remark that for each $M>0$
\begin{eqnarray}\label{we}
&&\Q^N_{V_\beta}\left(\left\{ \hat\mu^N_{X_1}(I^{c})\neq 0\right\} \cap\{\|\bX\|_\infty\le M\}
\right)\\
&&= \sum_{n=0}^{N-1} \Q^N_{V_\beta}\left(\left\{ \hat\mu^N_{X_1}(I)=\frac{n}{N}\right\} \cap\{\|\bX\|_\infty\le M\}\,.\nonumber
\right)
\end{eqnarray}
It is sufficient to show that each term in the above right hand side decays like some stretched exponential  provided  $M\ge L$ is chosen large enough as in Theorem \ref{Ttosupport} so that \begin{equation}\label{bM}\Q^N_{V_\beta}\left(\{\|\bX\|_\infty\ge M\}
\right)\le e^{-\frac{1}{2} N^{1-\kappa}}\end{equation} for some $\kappa\in (0,1)$.  We assume without loss of generality that $M\ge 1$.
To bound each term in the right hand side of \eqref{we} we perform the previous change of variables, keeping fixed the $n$ eigenvalues in $I$, as follows. 
Let $(x_1,\dots,x_n)$ be the eigenvalues of $X_1$ inside $I$ and let us apply the change of variables $\phi$ constructed  in Lemma \ref{Jacobian2}
on the remaining eigenvalues, with $\alpha=0$ and $\varepsilon$ small enough so that $U(x)\in [0,\eta/2]$
 for  $x$ in  $[-\varepsilon,\varepsilon]$. 
Then, we have constructed $\phi$ so that
$$J_\phi \geqslant \left(\frac{\vep}{4e}\right)^{3N(N-n)}e^{-2N\Tr \ln (1+|X_{1}|1_{X_{1}\notin I})-2(N-n)\Tr \ln (1+|X_{1}|1_{X \in I})}$$
and the spectrum of $U^1(\phi(X_{1}))$ is included in  $[0;\eta]$.
Note also that each eigenvalue $y$ outside of $I$ is send in $[-\varepsilon;\varepsilon]$ and thus $U^1(y)-U^1(\phi(y))\ge  \eta/2+\ln(1+|y|1_{|y|>A_U})$ for every $y\in I^c$ and a fixed large $A_U$ which depends only on $U^1$ (this comes from the fact that $U^1$ grows at least like a quadratic at infinity). 
As a consequence, we get

\begin{eqnarray*}
\frac{e^{-N\beta \Tr(U^1(X_1))}}{e^{-N\beta \Tr(U^1(\phi(X_1)))}}&\leqslant & e^{-\beta N (N-n)\eta / 2-\beta N\Tr \ln (1+|X_{1}|1_{X_{1}\notin I,|X_1|>A_U})}\end{eqnarray*}
We can extend the change of variables $\phi$ to the other matrices by just setting it equal to  the identity $\phi(X_i)=X_i, i\ge 2$. Noticing that $X_i-\phi(X_i)$ has rank at most $N-n$, we find that 
 if $W$ is a polynomial of degree $2D$, we can find a finite constant $C_1$ which only depends on $W$ such that
$$
\frac{e^{-N \Tr(W(X))}}{e^{-N \Tr(W(\phi(X)))}}\leq \exp\{ C_1(\|\bX\|_\infty^{2D} +1)(N-n)N\}
$$
Therefore, Lemma \ref{change-variable}
gives  (up to eventually change a bit the constant $C_V$)
\begin{align*}
&Q^N_{V_\beta}\left(\left\{ \hat\mu^N_{X_1}(I)=\frac{n}{N}\right\} \cap\{\|\bX\|_\infty\le M\}
\right)\\
&\leqslant 
e^{-\beta N (N-n)\eta / 2 +C_1 (M^{2D}+1) (N-n)N}
{\left(\frac{\vep}{4e}\right)^{-3N(N-n)}}\\
&\leqslant  e^{N(N-n)(C(M)-\beta\eta/2 -3\ln(\varepsilon))}\end{align*}
where  $C(M)$ is some finite constant which only depends on $M,U^1$ and $W$ {and in particular does not depend on $\beta$}.
Finally, we take  $\eta = \beta^{-1}L\ln\beta$ with $L$ large enough. Because $U^1$ is $C^{2}$, $\varepsilon$ is of order $\eta$ and in particular 
 $\ln \varepsilon = o(\eta\beta)$. Therefore $C(M)-\beta\eta/2 -3\ln(\varepsilon)$ will be large and negative, say smaller than $-2^{{-1}}L\ln \beta$,   for large $\beta$. Consequently by \eqref{we}, we find
 $$ Q^N_{V_\beta}\left(\left\{ \hat\mu^N_{X_1}(I^{c})\neq 0\right\} \cap\{\|\bX\|_\infty\le M\}
\right)\le  \sum_{n=0}^{N-1}   e^{-N(N-n)\frac{L}{2} \ln \beta }\le 2 e^{-N \frac{L}{2} \ln \beta }$$
which  implies by \eqref{bM} that there exists $\kappa\in (0,1)$
$$ Q^N_{V_\beta}\left(\left\{ \hat\mu^N_{X_1}(I)=1 \right\} 
\right)\ge 1-e^{-N^{1-\kappa}}$$
and concludes the proof that almost surely $\hat\mu^N_{X_1}(I^c)$ vanishes for $N$ large enough.  It also implies
  \eqref{c1} because $U^{1}$ is $C^{2}$ and with finitely many minimizers,
for $\beta$ large enough,  the set $I=\{x: |U^{1}(x)|<\beta^{-1}L\ln\beta\} $ is included in the set 
$$\cup_j \{ x: |x-x_j|<L'( \beta(\ln\beta)^{-1})^{-\frac{1}{2k_j}}\}$$
for some $L'$ large enough. 
Therefore, we deduce that almost surely
$$\hat\mu^N_{X_1}\left( \cup_{1\le j\le k} [x_j -L' (\frac{\beta}{\ln\beta})^{-\frac{1}{2k_j}}, x_j+ L' (\frac{\beta}{\ln\beta})^{-\frac{1}{2k_j}}]\right)=1\,\,\mbox{ for $N$ large enough }$$
We now estimate the filling fractions, namely the masses around each minimizers.  From the previous considerations, if we set

$$A=\left\{ \hat\mu^N_{X_1}\left( \cup_{j=1}^{k} [x_j -L' (\frac{\beta}{\ln\beta})^{-\frac{1}{2k_j}}, x_j+ L' (\frac{\beta}{\ln\beta})^{-\frac{1}{2k_j}}]\right)=1\right\}\cap
\left\{ |\bX\|_\infty\le M\right\}$$
then $\Q_V^N(A^c)$ goes to zero exponentially fast. We will therefore restrict ourselves to the 
set $A$. 
We diagonalize $X_1$ and write  for a unitary matrix $V$

$$X_1{=}V^*x V +V^*\Lambda(\beta) V$$
with $x=\Diag (x_1,\ldots, x_1,x_2, \ldots,x_{2},\ldots ,x_p)$ where $x_i$ has multiplicity $N_i$ and 
$$\Lambda(\beta)=
\left(
\begin{array}{ccccc}
\beta^{-1/2k_{1}}\Lambda_1& 0&0&\cdots&0\\
0&\beta^{-1/2k_{2}}\Lambda_2&0&\cdots&0\\
0&0&\ddots&0&0\\
\vdots&\vdots&\vdots&\vdots&\vdots\\
0&0&0&0&\beta^{-1/2k_{1}}\Lambda_p
\end{array}
\right)$$
with $\Lambda_i$ a $N_{i}\times N_i$  real diagonal matrices with entries in $[\pm L'(\ln\beta)^{\frac{1}{2k_j}}]$. 
Recalling our hypothesis on the minimizers $x_{i}$ of $U^{1}$, 
$$\beta  U^{1}(x_{i}+\beta^{-1/2k_{i}}x)=c_{i}x^{2k_{i}}+\beta P_{i}(\beta^{-1/2k_{i}}x)$$
with  $c_i$  a positive constant and $P_{i}$ a polynomial with a zero of order  $p_{i}$ strictly greater than $2k_i$ at $0$.
Note that on the set $A$: 
$$ \beta \Tr U^{1}(X_1) = \sum_i c_i \Tr (\Lambda_i^{2k_i})+\Tr \beta P_{i}(\beta^{-1/2k_{i}}\Lambda_i)$$
the second term is small when $\beta$ is large since, for some integer number $p_i>2k_i$, it can be  bounded by  $CN \beta (\ln \beta/ \beta)^{p_i/2k_{i}}$ . Setting $\bY=(X_2,\ldots,X_\ell)$,
we   can give a new expression for $1_Ad\Q^N_{V_\beta}$ which is equal to
\begin{align}
&=\frac{1_{A}}{Z^N_{V_\beta}}e^{-\beta N\Tr U^1(X_1)-N\Tr W(X_1,\bY)}dX_1d\bY\nonumber\\
&=\frac{1_{A}}{Z^N_{V_\beta}}e^{- N\left(\sum_i c_i \Tr (\Lambda_i^{2k_i}) +\Tr \beta P_{i}(\beta^{-1/2k_{i}}\Lambda_i)
-\Tr W(x+\Lambda(\beta),\bY)\right)}\Delta^2(x+\Lambda(\beta))d\Lambda d\bY\nonumber\\
&=\frac{1 }{Z^N_{V_\beta}}1_Ae^{- N\left(\sum_i c_i \Tr (\Lambda_i^{2k_i}) +o(\beta) N-\Tr W(x,\bY)\right)}\Delta^2(x+\Lambda(\beta))d\Lambda d\bY\label{dev}
\end{align}
where $o(\beta)$ goes to zero  like $\beta^{-\frac{1}{2\max k_{i}}}$ as $\beta$ goes to infinity while $M$ is fixed. 

In the square of the  VanderMonde $\Delta^2(x+\Lambda(\beta))$ there are factors with the differences coming from two different global minima:
%$$\prod_{x_i+\beta^{-1/2k_i}y \in Sp(\Lambda_i)}\prod_{x_j+\beta^{-1/2k_j}z \in Sp(\Lambda_j)}|x_i-x_j+\beta^{-1/2k_i}(y-z)|^2= |x_i-x_j|^{N_i N_j} e^{o(\beta)N_{i}N_{j}}$$
\begin{align*}
&\prod_{x_i+\beta^{-1/2k_i}y \in Sp(\Lambda_i)}\prod_{x_j+\beta^{-1/2k_j}z \in Sp(\Lambda_j)}|x_i-x_j+\beta^{-1/2k_i}(y-z)|^2\\
&= |x_i-x_j|^{N_i N_j} e^{o(\beta)N_{i}N_{j}}
\end{align*}
where $|o(\beta)|\le C\beta^{-\frac{1}{2\max k_{i}}}$,
and the factors coming from the same minimum:
$$\prod_{x_i+\beta^{-1/2k_i}y \in Sp(\Lambda_i)}\prod_{x_i+\beta^{-1/2k_i}z \in Sp(\Lambda_i)}\beta^{-\frac{1}{k_i}}|y-z|^2=\beta^{-\frac{N_i(N_i-1)}{2k_i}} \Delta(\Lambda_i)^2$$
Finally we get, \begin{equation}\label{jha}1_A\Q^N_{V_\beta}=\frac{1}{Z^N_{V_\beta}} 1_A \prod_i \beta^{-\frac{N_i(N_i-1)}{2k_i}} e^{ N^2o(\beta)} \prod_i e^{- N c_i \Tr (\Lambda_i^{2k_i})} \Delta(\Lambda_i)^2
d\Lambda_{i}  e^{O(N^{2})}d\bY\end{equation}
where $X{=}V^*x V +V^*\Lambda(\beta) V$ in the above RHS and $O(N^{2})$ bounds $N\tr W$ uniformly on $A$.
We deduce a lower bound of the partition function by conditioning by $N_i/N=n_i$ to see that
\begin{equation}\label{gf}Z^N_{V_\beta}\ge  \beta^{-N^2\sum_i\frac{n_i^2}{2k_i}+O(N^2)} \end{equation}
where we used \cite{BAG,AGZ}  to see that for every $i$, with the notations of \eqref{energy}, we have
$$\lim_{N_i\rightarrow\infty} \frac{1}{N_i^2}\ln \int e^{- N c_i \Tr (\Lambda_i^{2k_i})} \Delta(\Lambda_i)^2d\Lambda_i=-\inf\{E_{n_{i}^{-1}c_{i}x^{2k_{i}}}\} $$
where the supremum is taken over all probability measures  $\mu$ on the real line.  In particular, the limit does not depend on $\beta$ large enough, even though we integrate a priori on $[-L'(\ln\beta)^{\frac{1}{2k_j}} ; L'(\ln\beta)^{\frac{1}{2k_j}}]$, because the supremum
 of $E_{n_{i}^{-1}c_{i}x^{2k_{i}}}$ is taken at a compactly supported probability measure.
We can optimize over the choice of $n_i$ in \eqref{gf}, that is minimize
$(n_i)\mapsto \sum_i\frac{n_i^2}{2k_i}$ while $\sum n_i=1$.  We find that we should take $n_i=n_i^*= k_i/\sum k_j$. 
Plugging back this lower bound on the partition function into \eqref{jha} we find 
$$\Q^N_{V_\beta}\left( \cap_i\{\frac{N_i}{N}=n_i\}\cap A\right)\le \exp\{ -N^2 \ln \beta( \sum \frac{n_i^2}{k_i}-\sum \frac{(n_i^*)^2}{k_i}) +O(N^2)\}
$$
Noticing that $ \sum \frac{n_i^2}{k_i}-\sum \frac{(n_i^*)^2}{k_i} =\sum_i (n_i-n_i^*)^2/k_i$, we conclude that there exists  $R$  finite so that for  $\beta$ big enough
$$\Q^N_{V_\beta}\left( \{\max_{i}|\frac{N_i}{N}-n_i^*| >R\sqrt{\ln \beta}^{-1} \}\cap A\right)\le e^{-N^2}$$
which completes the proof when $p=1$. It is straightforward to generalize the proof to more general $p$ as we can do the change of variables for any matrices separately to localize the spectrum of each of the $p$ matrices and then expand around the minimizers the spectrum of each of these matrices. 

\bigskip

{\bf Proof of  Corollary \ref{corsum}.} The first points are direct consequence of Theorem \ref{theosum}. The main point is hence to show convergence of the rescaled distribution near a minimizer. But, coming back to \eqref{dev}, if we condition by the event $\bN=(N_{1},\ldots, N_{p})=N(n_{1},\ldots,n_{p})=N\bn$, 
$X_{1}$ is fixed up to terms of order $\beta^{-\min\frac{1}{2k_{i}}}$, and therefore the interaction term $N\tr W(X_{1},Y)$ is fixed up to order $\beta^{-\min\frac{1}{2k_{i}}} N^{2}$. 
Hence, the previous proof shows that for any measurable set $B$ of the set of probability measures on the real line, 
$$\Q^N_{V_\beta}\left( \{\hat\mu_{\Lambda_i}\in B\}\cap A|\bN\right)=\frac{1 }{\tilde Z^N_{\beta}}\int_{\hat\mu_{\Lambda_i}\in B} e^{- N\left(c_i \Tr (\Lambda_i^{2k_i}) +o(\beta) N\right)}\Delta^2(\Lambda_i)d\Lambda dY$$ and therefore again by \cite{BAG} the right hand side goes exponentially fast to zero unless 
$$\inf_{\mu\in B}\{E_{c_{i}n_{i}^{-1}x^{2k_{i}}}(\mu)\}\le \inf  E_{c_{i}n_{i}^{-1}x^{2k_{i}}}+o(\beta)\,.$$
Put $r_{i}=c_{i}/n_{i}$ and 
using that  $\sigma^{k_i}_{r_{i}}$ satisfies
$$r_{i}x^{2k_{i}}-\int \ln|x-y| d\sigma^{k_{i}}_{r_{i}}(y)\ge  C$$
everywhere, with equality $\sigma_{r_{i}}^{k_{i}}$ almost surely, we find  that the above  inequality  holds iff for every $\mu\in B$,
$$D(\mu,\sigma^{k_{i}}_{r_{i}})^{2}:=
-\int\ln|x-y|d(\mu-\sigma^{k_{i}}_{r_{i}})(x)d(\mu-\sigma^{k_{i}}_{r_{i}})(y)\le  o(\beta)$$
Moreover, for any probability measures $\mu,\nu$, we have \cite{BAG}
$$D(\mu,\nu)=\left(\int_0^{+\infty} \frac{1}{t}\left| \int e^{itx} d(\mu-\nu)\right|^2 dt\right)^{1/2}\,.$$
so that $D$ defines a distance on the set of probability measures and moreover, 
 Cauchy-Schwartz inequality shows  that
$$|\int fd(\mu-\nu)|\le \|f\|_{{1/2}}D(\mu,\nu)$$
and the fact that $o(\beta)$ is of order at most $\beta^{-\frac{1}{2\max k_{i}}}$ completes the proof.

\subsection{Low temperature  expansion }
In this section we consider a special case of the previous one where   the interaction term vanishes at the minimizers of $U^{1}$. This additional hypothesis will allow us   to obtain results when $\beta$ is sufficiently large but finite. 

\begin{assum}\label{assumsingle2}
$$V_{\beta}(X_{1},\ldots, X_{\ell})=\beta \sum_{i=1}^{p} U^{i}(X_{i})+\sum_{i=p+1}^{\ell}V^{i}(X_{i})+Z(X_{1},\ldots, X_{\ell})
$$
where 
\begin{itemize}
\item $( U^{i})_{1\le i\le p}$ and $W(\bX)=\sum_{i=p+1}^{\ell}V^{i}(X_{i})+Z(X_{1},\ldots,X_{\ell})
$ satisfy Assumption \ref{assumsingle}. Moreover $V^{i}, p+1\le i\le \ell$ is uniformly convex and  the $k_{j}^{i}$ are equal to one. 
\item $Z$ can be decomposed, up to cyclic invariance as $\sum_{i=1}^{p}Z_i(X_{1},\ldots,X_{\ell})$ with some { polynomials $Z_{i}(X_{1},\ldots, X_{\ell})$ whose trace vanishes  at any $\ell$-tuple $(X_{1},\ldots,X_{\ell})$ of operators such that the eigenvalues of $X_{i}$ belong to  the set $\{x^{i}_{j}, 1\le j\le m_{i}\}$ for any  choice of operators $X_{j},j\neq i$. }
\end{itemize}
\end{assum}
{
\begin{remark}
\begin{itemize}
\item Potentials of the form  $$Z_{i}(X_{1},\ldots,X_{\ell})=\prod_{j=1}^{m_{i}}(X_{i}-x_{j}^{i})Q^{i}(X_{1},\ldots,X_{\ell})$$ 
with polynomials $Q^{i}$ fulfill the second point of Assumption \ref{assumsingle2}.
\item 
By closer inspection of the proof, we see that it is enough to assume that for any $\ell$ tuple of unitary matrices $O_{1},\ldots,O_{\ell}$ and diagonal matrices $D_{1},\ldots, D_{\ell}$ such that the eigenvalues of
$D_{i}$ belong to the set $\{x^{i}_{j}, 1\le j\le m_{i}\}$ and for $\epsilon$ small enough $\hat\mu_{D_{i}}(\{x^{i}_{j}\})$ belongs to $[1/m_{i}-\epsilon,1/m_{i}+\epsilon]$,
$$\frac{1}{N}\Tr \left(Z(O_{1}D_{1}O_{1}^{*},\ldots, O_{\ell} D_{\ell} O_{\ell}^{*})\right)= 0\,.$$
Also, our result could extend to more general  integer numbers  $k_{i}^{j}$ but we chose to take $k_{i}^{j}=1$ for simplicity. 
\item Moreover,  we might have generalized our result to functions $U^{i},V^{i},Z$ which are not polynomials in the sense  that tightness could be proved directly  in this setting without the trapping conditions. For consistency we however restrict ourselves to polynomials functions.
\end{itemize}
\end{remark}}
Then, we  prove the following low temperature (namely large $\beta$) expansion, which is based on the fact that, up to rescaling, our hypothesis insures that the interaction is small with overwhelming probability under our hypothesis so that we can use \cite{CGM, GN}. 
\begin{theorem}\label{low}Assume Assumption \ref{assumsingle2} holds. Then, for $\beta$ large enough, the empirical distribution  $\hat\mu_{N}$ of $\bX^{N}$ converges  $\Q^{N}_{V_{\beta}}$-
almost surely towards $\tau_{\beta}$. Moreover, the moments  under $\tau_{\beta}$  expand as a converging series in 
$\beta^{-\frac{1}{2}}$. $\tau_{\infty}$ is the law of $\ell$ free variables, the  $p$th first having law {$\frac{1}{m_{i}
}\sum_{j=1}^{m_{i}} \delta_{x^{i}_{j}}$}  and the $\ell-p$ last  ones following  the minimizer $\mu_{V^{i}}$ of  the energy $E_{V_{i}}$.
\end{theorem}

\begin{proof} {By Theorem \ref{theosum}, we can restrict the domain of integration to
$$\mathcal B=\cap_{i\le p}\{ \hat\mu^N_{X_i^N}(K_{\beta,R}^i)= 1 \}\cap_{1\le j \le m_{i}}\{ |\hat \mu^N_{X_i^N}([x_j^i-\eta,x_j^i+\eta]) - \frac{1}{m_{i}}|\le \frac{C}{ \sqrt{\ln (\beta)}}\}$$
where $C$ is finite and $\eta> \beta^{-d}$ with $d> \frac{1}{2}$
 fixed. We next condition $\Q^{N}_{V_{\beta}}$ by the values of $N_{j}^{i}=N\hat \mu^N_{X_i^N}([x_j^i-\eta,x_j^i+\eta])$, $1\le j\le m_{i}$, and  
diagonalize the matrices $X_{i}^{N}=O_{i}D_{i}O_{i}^{*}$ with unitary matrices $O_{i}$. Under this conditioning, $D_{i}$ is the block diagonal matrix 
with $N_{j}^{i}$ entries close to $x_{j}^{i}$. We let $D_{i}^{\infty}$ be the matrix with $N_{j}^{i}$ entries  exactly equal to $x_{j}^{i}$ : it  is therefore fixed. We denote by $\Lambda_{i}=\beta^{d}(D_{i}-D_{i}^{\infty})$ with  $d\in (0,\min_{i}\frac{1}{2} )$
so that  the matrices $\Lambda_{i}$  stay bounded  on $\mathcal B$ uniformly on $\beta$  and let $\bY_{i}=(X_{j},j\neq i)$. 
Moreover, by assumption, $\tr Z_{i}(O_{i}D^{\infty}_{i} O_{i}^{*},\bY_{i})$ vanishes  so that
\begin{eqnarray*}&&\tr Z_{i}(\bX)=\tr (Z_{i}(O_{i}D_{i} O_{i}^{*}, \bY_{i})-Z_{i}(O_{i}D^{\infty}_{i} O_{i}^{*}, \bY_{i}))\\
&&=\beta^{-d}
\tr Z_{i}\left(O_{i}\Lambda_{i} O_{i}^{*}\int_{0}^{1}\Da_{X_{i}} Z_{i}( O_{i}D_{i}^{\infty} O_{i}^{*}+\alpha \beta^{-d}O_{i}\Lambda_{i} O_{i}^{*}, \bY_{i})\right) d\alpha\,.\end{eqnarray*}
Consequently, since this is true for every $i$, we find a polynomial $P^{\beta}$ with coefficients which are polynomials in $\beta^{-d}$  and such that
$$\Tr Z((O_{j}D_{j}O_{j}^{*})_{1\le i\le j})=\beta^{-d} \Tr( P^{\beta}((O_{j}D_{j}^{\infty} O_{j}^{*}, O_{j} \Lambda_{j}O_{j}^{*})_{j\le p}, (O_{j}D_{j}O_{j}^{*})_{j\ge p+1}))$$
}
 Following \cite[Theorem 2.3]{FG} and \cite{GN}, when $\beta$ is large enough, we know that
%$$\int e^{N \tr Z(O_{1}D_{1} O_{1}^{*}, \ldots, O_{\ell}D_{\ell}O_{\ell}^{*})}dO_{1}\cdots dO_{\ell}
%=e^{N^{2}\beta^{-d}F_{\beta}((\hat\mu_{D_{i}^{\infty},\Lambda_{i}})_{1\le i\le p}, (\hat \mu_{D_{i}})_{ p+1\le i\le \ell})+o(N^{2})}
%$$
\begin{align*}
&\frac{1}{N^2}\int e^{N \tr Z(O_{1}D_{1} O_{1}^{*}, \ldots, O_{\ell}D_{\ell}O_{\ell}^{*})}dO_{1}\cdots dO_{\ell}\\
&=\beta^{-d}F_{\beta}((\hat\mu_{D_{i}^{\infty},\Lambda_{i}})_{1\le i\le p}, (\hat \mu_{D_{i}})_{ p+1\le i\le \ell})+o(1)
\end{align*}

where $F_{\beta}$ is a smooth function of the joint distributions $\hat\mu_{D_{i}^{\infty},\Lambda_{i}}$ of the diagonal matrices $(D_{i}^{\infty},\Lambda_{i})$, $1\le i\le p$, and  of the  empirical measures of the eigenvalues $\hat\mu_{D_{i}}, i\ge p+1$ of the remaining matrices. Moreover, $F_{\beta}$ expands as a converging series in $\beta^{{-1/2}}$. Note that by construction for any polynomial $P(X,Y)=P_{1}(X)Q_{1}(Y)\ldots P_{r}(X) Q_{r}(Y)$,
$$\hat\mu_{D_{i}^{\infty},\Lambda_{i}}(P)=\sum_{j=1}^{m_{i}}\frac{N_{j}^{i}}{N} \prod_{k=1}^{r}P_{k}(x_{j}^{i}) \hat\mu_{\Lambda_{i}^{j}}(\prod Q_{k})$$
is in fact completely determined by the empirical measure of $D_{i}$, that is of $X_{i}=D_{i}^{\infty}+\beta^{-d} \Lambda_{i}$ for $\beta$ large enough. $ \hat\mu_{\Lambda_{i}^{j}}$ is the empirical measure of the eigenvalues of $X_{i}$ close to $x_{j}^{i}$: for any bounded continuous function $f$
$$\hat\mu_{\Lambda_{i}^{j}}(f)=\hat\mu_{X_{i}}(f||X_{i}-x_{j}^{i}|\le \eta)=
\frac{\hat\mu_{X_{i}}(f(X_{i})1_{|X_{i}-x_{j}^{i}|\le \eta})}{ \hat\mu_{X_{i}}({|X_{i}-x_{j}^{i}|\le \eta})}\,.$$
 Hence $\hat\mu_{D_{i}^{\infty},\Lambda_{i}}$ is a smooth function of $\hat\mu_{X_{i}}$. 
 We therefore can  write
$$F_{\beta}((\hat\mu_{D_{i}^{\infty},\Lambda_{i}})_{1\le i\le p}, (\hat \mu_{D_{i}})_{ p+1\le i\le \ell})=G_{\beta, (x^{i}_{j})_{1\le j\le m_{i}}}((\hat\mu_{\Lambda_{i}^{j}})_{\substack{1\le j\le m_{i}\\ 1\le i\le p}}, (\hat\mu_{D_{j}})_{p+1\le j\le \ell})\,.$$
We  deduce by  large deviations techniques, see \cite{BAG,AGZ}, that if $N_{i}/N$ goes to $n_{i}$,
$$\Q^{N}_{V}(\mathcal B\cap\{\max_{1\le i\le k}d(\hat\mu_{X_{i}},\mu_{i})\le\delta\}|\frac{N_{j}^{i}}{N}=n_{i}^{j})\simeq e^{-N^{2}I_{n}((\mu_{i})_{1\le i\le \ell})+o(N^{2})}$$
where $$I_{n}(\mu_{i}, 1\le i\le \ell)=\sum_{i=1}^{p}\sum_{j=1}^{m_{i}}  n_{j}^{2}E_{
c_{j}^{i} (n_{j}^{i})^{-1}x^{2k_{j}^{i}}}(\mu_{\Lambda_{i}^{j}}) +\sum_{j=p+1}^{\ell} E_{
 V^{i}}(\mu_{i})$$
 $$\qquad\qquad -\beta^{-d}G_{\beta, (x^{i}_{j})_{1\le j\le m_{i}}}((\mu_{\Lambda_{i}^{j}})_{1\le j\le m_{i}},  i\le p, (\mu_{j})_{{p+1}\le j\le {\ell}})$$
where finally $\mu_{i}=\sum n_{j}^{i}\mu_{\Lambda_{i}^{j}}$, or in otherwords $\mu_{\Lambda_{i}^{j}}$ is the law  $\mu_{i}$ conditioned to be in  the neighborhood  of the minimizing  point $x_{j}^{i}$.  The rate  function for $\Q^{N}_{V_{\beta}} $ is deduced by optimizing over the $n_{i}$. 
Minimizers of such rate functions were  studied in \cite[Lemma 3.1]{FG}: it was shown that when the potentials $V^{i}$ are uniformly convex and $k^{i}_{j}=1$ and $\beta$ is large enough then there  exists  a unique minimizer because the rate function stays strictly convex under transport maps. This implies convergence of $\hat\mu_{X_{1}},\ldots\hat\mu_{X_{\ell}}$ towards some $\mu_{1}^{\beta},\ldots,\mu_{\ell}^{\beta}$ when $\beta$ is large enough. To deduce the convergence of the joint distribution of $\bX=(O_{1} D_{1}O_{1}^{*},\ldots,O_{\ell}D_{\ell}O_{\ell}^{*})$ it is enough to show the joint convergence of the distribution of $\bO=(O_{1},\ldots,O_{\ell})$ knowing that the distributions of $D_{1},\ldots,D_{\ell}$ converges towards $\mu_{1}^{\beta},\ldots,\mu_{\ell}^{\beta}$. But the law of $\bO$ is
absolutely continuous with respect to Lebesgue measure and with density
$$d\P^{N}_{V_{\beta}}(\bO)=\frac{1}{Z_{V_{\beta}}^{N}}e^{N \beta^{-d} \Tr( P^{\beta}((O_{j}D_{j}^{\infty} O_{j}^{*}, O_{j} \Lambda_{j}O_{j}^{*})_{j\le p}, (O_{j}D_{j}O_{j}^{*})_{j\ge p+1}))
)} d\bU\,,$$
with  $P_{\beta}$ a polynomial  with  coefficients  bounded independently of $\beta$ such that 
$$\beta^{-d} \Tr P_{\beta}(O_{j}D_{j}^{\infty} O_{j}^{*}, O_{j} \Lambda_{j}O_{j}^{*})_{j\le p}, (O_{j}D_{j}O_{j}^{*})_{j\ge p+1}))$$ 
$$=\Tr ( Z((O_{j}(D_{j}^{\infty}+\Lambda_{j})O_{j}^{*})_{j\le p}, (O_{j}D_{j}O_{j}^{*})_{j\ge p+1}))\,.$$
We can now use \cite{GN,CGM} to deduce that when $\beta$ is large enough and the laws  of $\hat\mu_{X_{1}},\ldots\hat\mu_{X_{\ell}}$ converging  towards some $\mu_{1}^{\beta},\ldots,\mu_{\ell}^{\beta}$, the joint distribution  of 

\noindent
$(O_{i}D_{i}^{\infty}O_{i}^{*}, O_{i} \Lambda_{i}
O_{i}^{*}, i\le p, U_{p+1}D_{p+1}U_{p+1}^{*}, \ldots, U_{\ell}D_{\ell}U_{\ell}^{*})$ converges towards a non-commutative  law $\tau_{\beta}$ whose moments are iteratively determined by Dyson-Schwinger  equations. 
Moreover, the moments  under $\tau_{\beta}$  expand as a converging series in 
$\beta^{-1/2}$  and  $\tau_{\infty}$ is the law of $\ell$ free variables, the  $p$th first having law {$\frac{1}{K
_{i}}\sum k^{i}_{j}\delta_{x^{i}_{j}}$}  and the $\ell-p$ last  ones following  the minimizer $\mu_{V^{i}}$ of  the energy $E_{V^{i}}$.

\end{proof}

Notice that  an  analogue of  Corollary \ref{corconv} should also hold in the setting of Theorem \ref{low} but its  proof  would necessitate to extend \cite{GS} to the case of perturbative potentials depending on random matrices and deterministic projections, a project that we do not pursue here.

\section{Strong  commutator  interaction} \label{commutator}

In this section we consider the case where the strong interaction is given by a commutator.
\begin{assum} \label{assumcom} There exists  polynomials $V_{1}$ and $V_{2}$ such that
$$V_{\beta}(\bX)=-\beta (X_{1}X_{2}-X_{2}X_{1})^{2}+V_{1}(X_{1})+V_{2}(X_{2})\,.$$
Moreover, $V_{1}$ and $V_{2}$ go to $+\infty$ at infinity (and therefore faster than $c x^{2}$ for some $c>0$).
\end{assum}

We already know  that $X_{1}$ and $X_{2}$ commute asymptotically when $N$ and then $\beta$ go to infinity by Corollary \ref{homo}, since  it implies that the trace of the moments  of  the commutator $U(\bX)=-(X_{1}X_{2}-X_{2}X_{1})^{2}$ go to zero.
However there are many ways to insure this is the case: for instance $X_{1}$ and $X_{2}$ can have asymptotically the same basis of eigenvectors but a non trivial spectrum  or  $X_{1}$ or $X_{2}$ could converge to a multiple of the identity.
 In fact it turns out this two strategies have the same small entropy: Because the eigenvalues of $X_{1}$ follow a Coulomb gaz, the event $\{\|X_{1}-I\|_\infty\le \varepsilon\}$ has probability of order $\varepsilon^{N^2}$. But similarly if we diagonalize $X_{1}$ and write $X_{2}=UDU^*$ with $D$ and $X_{1}$ commuting, and $U$ a unitary matrix, the eigenvalues of $U$ follow a Coulomb gas and therefore the event $\{\|U-I\|_\infty\le \varepsilon\}$ has probability of order $\varepsilon^{N^2}$. As noticed by D. Shlyakhtenko in a private communication, this implies that if we have a 3(or more)-matrix model such as
$$V(X_1,X_2,X_3)=-\beta [X_1,X_2]^2-\beta [X_1,X_3]^2+W(X_1,X_2,X_3)$$
then the most likely strategy is  that $X_1$ is close to a multiple of identity rather than  to force $X_1$ to commute with both $X_2$ and $X_3$. We show in the sequel that with only one commutator the situation is more interesting and the matrices in general do not  converge to a multiple of identity. In fact, if the potential $V_{1}$ and $V_{2}$ have several minimizers, the eigenvalues of the matrices $X_{1}$ and $X_{2}$ will concentrates in a neighborhood of these minimizers with non trivial masses. Asymptotically as $N$ and then $\beta$ go to infinity, the matrices will commute and have ''independent'' non trivial  spectrum.

First we consider the case where one of the potential  is quadratic. We define the following functional on the space of probability measures on $\mathbb R$:
$$J^V_{\beta}(\mu)=\int\int \left( \frac{1}{2}(V(x)+V(y))-\ln\frac{\beta|x-y|^2}{1+\beta|x-y|^2}\right)d\mu(x)d\mu(y)$$
and its off diagonal version
$J^V_{\beta,\neq}(\mu)$  given by
$$J^V_{\beta,\neq}(\mu)=\int\int 1_{x\neq y} \left( \frac{1}{2}(V(x)+V(y))-\ln\frac{\beta|x-y|^2}{1+\beta|x-y|^2}\right)d\mu(x)d\mu(y)$$

\begin{prop}[{\bf Asymptotics if one  potential is quadratic}]\label{prop1}

	Let $X^N,Y^N$ be two $N\times N$  random Hermitian or symmetric  matrices with distribution:
$$d\Q^N_{V_\beta}(X,Y)=\frac{1}{\mathbb Z^N_{V_{\beta}}}\exp(-N\Tr (-\frac{\beta}{2}[X,Y]^2+V(X)+\frac{Y^2}{2}))dXdY$$
where $V$ is a polynomial going to $+\infty$ at infinity and $\beta$ a non-negative real number.
Then under $\Q^N_{V_\beta}$,
 the law of the eigenvalues $x_1,\dots,x_N$ of $X^{N}$ is given by:
$$Q^N_{V_\beta}(dx)=\frac{1}{ Z^N_{V_{\beta}}}
\exp(-N^2J^V_{\beta,\neq}(\frac{1}{N}\sum_{i=1}^N\delta_{x_i}))dx_1\dots dx_N$$

	\begin{enumerate}
	\item Moreover, 
	 \begin{eqnarray*}
	\lim_{N\rightarrow\infty }\frac{1}{N^2}\ln   Z^N_{V_{\beta}}&=&\lim_{N\rightarrow\infty }\frac{1}{N^2}\ln 
	 \int_{\RR^N}\exp(-N^2J^V_{\beta,\neq}(\frac{1}{N}\sum_{i=1}^N\delta_{x_i}))dx_1\dots dx_N\\
	 &=&-c^V_\beta:=-\inf J^{V}_{\beta}.\end{eqnarray*}
\item 
Under $\Q^N_{V_\beta}$ (or $Q^N_{V_{\beta}}$), the law of $\hat\mu^N_X$  satisfies a large deviation principle with speed $N^2$ and good rate function $\mu\to J^V_\beta(\mu)-c^V_\beta$.
In particular $\hat\mu^N_X$ converges almost surely  towards the unique minimizer $\mu^V_\beta$ of  $J^{V}_{\beta}$.
 
		\item {Assume that $V$ achieves its minimal value $0$  at the real numbers $\{x_{1},\ldots, x_{r}\}$  so that $V(x)=c_{i}(x-x_{i})^{2p_{i}}(1+o(1))$ for some integer numbers $p_{i}\ge 1$ and  positive real numbers $c_{i}$. Let $p=\max p_{i}$.}\color{black}
		Then, there  exists  a finite constant $A(V)$ such that 
		\begin{equation}\label{boundRadius}\limsup\frac{1}{N}\ln \Q^N_{V_\beta}(\|Y\|_{\infty}>A(V)\beta^{-\frac{p}{5(2p+1)}})<0 \,.\end{equation}

	\end{enumerate}
	
\end{prop}
\begin{proof} The fact that the law of $(x_{1},\ldots, x_{N})$ has distribution $Q^{N}_{V_{\beta}}$ follows from the fact that we can assume without loss of generality that $X$ is diagonal as the law of $Y$ is invariant under conjugation by the unitary group and then
$$\Tr(-[X,Y]^2)=\Tr(-(XY-YX)^{2})=2\sum_{i<j}(x_{i}-x_{j})^{2}|Y_{ij}|^{2}\,.$$
Integrating over the entries of the matrix $Y$ , we  get the distribution
$$\frac{1}{Z^{N}_{\beta}} \prod_{i< j}(1+\beta (x_{i}-x_{j})^{2})^{-1} e^{-N\sum_{i=1}^{N} V(x_{i})} \prod_{i<j} |x_{i}-x_{j}|^{2}\prod dx_{i}$$

	The large deviation principle and the convergence of the free energy  are a direct consequence of a general  large deviation principle proved by Garcia Zelada \cite{Garcia} and Ben Arous-Guionnet \cite{BAG,AGZ}.
	It is classical to prove that $J_{\beta}^{V}$ is a good rate function ( $J_{\beta}$ differs from 
	 the usual Beta-ensemble rate function by a smooth  smooth and its behavior at infinity is balanced by the logarithmic entropy). We show below that 
	it is  strictly convex and therefore has a unique minimizer. Because $J^{V}_{\beta}$ is a sum of the linear term $\mu\rightarrow \int Vd\mu$ and $J^{0}_{\beta}$, it is enough to show that $J^{0}_{\beta}$ is strictly convex. 
		The argument is similar to the one used for Beta-ensembles \cite[Lemma 2.6.2]{AGZ}. Let $\varepsilon$ be a postive real number.
		We define the regularization  $J_{\beta}^{\varepsilon}$ of $J_\beta^{0}$ by setting,  for any probability measure $\mu$, 
		$$J^{\varepsilon}_\beta (\mu)= -\int\int\ln\frac{\varepsilon+\beta|x-y|^2}{1+\beta|x-y|^2}d\mu(x)d\mu(y)$$
		Then, by monotone convergence theorem, for any probability measure $\mu$,
		$$J_\beta^{0}(\mu)= \lim_{\varepsilon\to 0}J^{\varepsilon}_\beta(\mu).$$
		Moreover, for $0<\alpha<1$,	we find :
		\begin{eqnarray*}
			&&(\alpha(1-\alpha))^{-1}(\alpha J^{\varepsilon}_\beta(\mu)+(1-\alpha) J^{\varepsilon}_\beta(\nu)-J^{\varepsilon}_\beta(\alpha\mu+(1-\alpha)\nu))\\
			&=& \int_{\RR^2} \left( \ln \frac{1+\beta|x-y|^2}{\varepsilon+\beta|x-y|^2}\right)d(\mu-\mu')(x)d(\mu-\mu')(y) \\
			&=& \int_{\RR^2} \int_{\varepsilon}^1  \frac{da}{a+\beta|x-y|^2} d(\mu-\mu')(x)d(\mu-\mu')(y) \\
			&=& \int_{\RR^2} \int_{\varepsilon}^{1}\left(\int_{0}^{\infty} \exp(-(a+\beta|x-y|^2)t)dt \right)  d(\mu-\mu')(x)d(\mu-\mu')(y) da\\
			&=& \int_{\RR^2} \int_{\varepsilon}^{1}\int_{0}^{\infty}\left(\int_{g\in\RR} e^{-at}e^{i g\sqrt{t\beta}(x-y)}d\gamma(g)\right) dt da d(\mu-\mu')(x)d(\mu-\mu')(y) \\
		\end{eqnarray*}
		where $\gamma$  is the standard Gaussian measure. Thanks to the regularization, we can use Fubini's theorem. As a consequence, we find that
				\begin{eqnarray*}
			&&(\alpha(1-\alpha))^{-1}(\alpha J^{\varepsilon}_\beta(\mu)+(1-\alpha) J^{\varepsilon}_\beta(\nu)-J^{\varepsilon}_\beta(\alpha\mu+(1-\alpha)\nu))\\
			&=& \int_{\varepsilon}^{1}\int_{0}^{\infty}\int_{g\in\RR} e^{-at}|\int_{\RR} e^{i g\sqrt{t\beta}x}d(\mu-\mu')(x)|^2d\gamma(g)dt da  \\
		\end{eqnarray*}
		The right hand side is always non-negative  so that $J^{\varepsilon}_\beta$ is   convex for every $\epsilon>0$  and therefore $J_\beta^{0}$ is  convex by passing to the limit.  $J_\beta^{0}$ is in fact strictly convex since, if $J^{0}_\beta$ is finite at $\mu$ and $\nu$,  $J^{0}_\beta(\alpha\mu+(1-\alpha)\nu)$ is also finite and
		\begin{eqnarray*}
			&&(\alpha(1-\alpha))^{-1}(\alpha J^{0}_\beta(\mu)+(1-\alpha) J^{0}_\beta(\nu)-J_\beta^{0}(\alpha\mu+(1-\alpha)\nu))\\
			&=& \int_{0}^{1}\int_{0}^{\infty}\int_{\RR} e^{-at}|\int_{\RR} e^{i g\sqrt{t\beta}x}d(\mu-\mu')(x)|^2d\gamma(g)dt da  \\
		\end{eqnarray*}
		The RHS can only vanish if $\mu$ and $\nu$ have the same Fourier transforms almost surely which implies
		 $\mu=\nu$. This proves the strict convexity of $J_\beta^{V}$ and the uniqueness of  its minimizer $\mu_{\beta}^{V}$. As a consequence, if $d$ is a distance compatible with the weak topology, for any $\delta>0$, there exists $c_{\delta}>0$ such that for $N$ large enough
		 $$\Q^{N}_{V_\beta}\left(d(\hat\mu_{X}^{N},\mu_{\beta}^{V})>\delta\right)\le e^{-c_{\delta}N^{2}}\,.$$
		 We take $d(\mu,\nu)=\sup\{|\int fd(\mu-\nu)|\}$ where the supremum is taken over functions  $f$ with Lipschitz norm bounded by one. 
		
{ We finally deduce the bound on $\|Y\|_{\infty}$.	 
 By Lemma \ref{energyproblempotential}, we know that there exists a finite constant $C$ such that for any interval $I_\beta$ of length $2\varepsilon_\beta$ with $\varepsilon_\beta^2\beta\to\infty$, {
	 $$\mu_{\beta}^{V}(I_\beta)\le C\beta^{\frac{1}{4(2p+1)}}\sqrt{\varepsilon_\beta}.$$}\color{black}
		 We take
		 $I_{\beta}=[x-\varepsilon_\beta,x+\varepsilon_\beta]$  for a real number $x$ and 
		   approximate  the indicator function of ${I_\beta}$ by a Lipschitz function with  support in $[x-2\varepsilon_\beta,x+2\varepsilon_\beta]$ and Lipschitz norm bounded by $1/\varepsilon_\beta$. 
		  Therefore, for any $\delta>0$, on $d(\hat\mu_{X}^{N},\mu_{\beta}^{V})\le \delta$, we have
		 		 \begin{equation}\label{lkj}\hat\mu_{X}^{N}([x-\varepsilon_\beta,x+\varepsilon_\beta])\le C\beta^{\frac{1}{4(2p+1)}} \sqrt{\varepsilon_\beta} +\frac{1}{\varepsilon_\beta} \delta\end{equation}
		 with probability greater than $e^{-c_{\delta}N^{2}}$ for some $c_{\delta}>0$. This upper bound can  be made uniform on $x$ by using a finite covering of $[-M,M]$ and the fact  that $\|X\|_{\infty}$ stays bounded by  some finite $M$ with overwhelming probability since $V_{\beta}$  is $(\eta,A)$-trapping for some $\eta>0$, $A$ finite and the trivial partition $\{\{1\},\{2\}\}$ by 
		 Theorem \ref{conf-to-T}.  We now fix $\delta= C\beta^{\frac{1}{2(2p+1)}} \varepsilon_\beta^{3/2}$.
		 As a consequence, we find a positive finite constant $c'$ such that, with probability greater than $1-e^{-cN}$ for some $c>0$ and $N$ large, if we order the $(x_{i})_{1\le i\le N}$ we have 
		 \begin{equation}\label{jh} 
		 \inf\left\{  |x_{i}-x_{j}|, |i-j|\ge C\beta^{\frac{1}{4(2p+1)}} \sqrt{\varepsilon_\beta} N\right\}\ge  2\varepsilon_\beta\,.\end{equation}
		 Indeed,  otherwise the interval $I_\beta=[x_i;x_j]$ with $x_i,x_j$ achieving the above minimum would contradict (\ref{lkj}).
		 Next, we consider the matrix $Y$ conditionally to the $(x_{i})_{1\le i\le N}$ satisfying \eqref{jh}. We first observe that, under this conditional law, the $Y_{ij}$ are  independent centered Gaussian variables with variance
		 $N^{-1}(1+2\beta(x_{i}-x_{j})^{2})^{-1}$.  Define $\eta_\beta =C\beta^{\frac{1}{4(2p+1)}} \sqrt{\varepsilon_\beta}$, we  make the following decomposition
		 $$Y= \Delta+ G$$
		 where $\Delta$ is the band matrix with entries $(Y_{ij}1_{|i-j|\le\eta_\beta N})_{i,j}$ and $G$ collects the remaining entries.  The variance of these entries is bounded by $1/N$ and therefore, by Lemma  \ref{normbandmatrix}, 
		 $$\mathbb P\left(\|\Delta\|_{\infty}\ge C\sqrt{\eta_\beta }\right)\le e^{-\eta_\beta N}\,.$$
		 On the other hand, when \eqref{jh} holds, {
		 the entries of $G$ have variance bounded from above by $N^{-1}(1+2\beta (\varepsilon_\beta)^{2})^{{-1}}$. Therefore, if $W$ is a $N\times N$ matrix from the GOE,
		  \cite{HuGu}[Lemma 5.4] shows that
		  $$\|G\|_{\infty}\le  (1+2\beta (\varepsilon_\beta)^{2})^{{-1/2}}\|W\|_{\infty}$$
		  so that the large deviation principle for $\|W\|_{\infty}$  (see \cite{BGD} or \cite[Theorem 2.6.6]{AGZ}) implies}
		   that there exists finite non negative constants $c,C$ such that 
		 $$\mathbb P \left(\|G\|_{\infty}\ge 4\left(1+2\beta (\varepsilon_\beta)^{2}\right)^{-1/2}\right)\le C\exp\{ -c N\}\,.$$
		 Taking  $\varepsilon_\beta=\beta^{-\frac{1}{5}\frac{8p+5}{2(2p+1)}}$ (it is easy to check that we always have $\varepsilon_\beta^2\beta\to\infty$) and 
		 putting these two bounds together gives a positive finite constant $C$ such that for $\beta$ and $N$ large enough
		 \begin{equation}\label{bsr}
		  \mathbb P\left(\|Y\|_{\infty}\ge C  \beta^{-\frac{p}{5(2p+1)}}\right)\le 		C  e^{-C \beta^{-\frac{p}{10(2p+1)}}N }\,.
		 \end{equation}

}
	
		\end{proof}

Amazingly, in the simpler case where both potentials are quadratic, the asymptotic distribution of the matrices can be computed explicitly. 

\begin{prop}[{\bf Asymptotics if both  potentials are quadratic}]
	Let $X^N,Y^N$ be two random matrices following 
	$$d\Q^N_{\beta}(X,Y)=\frac{1}{Z^N(\beta)}\exp(-N\Tr (-\beta[X,Y]^2+\frac{X^2}{2}+\frac{Y^2}{2}))dXdY.$$

		Then under the law $\Q^N_\beta$, the sequence of empirical distribution  $\mun_{\beta^{\frac 1 6}X, \beta^{\frac 1 6}Y}$ of $(\beta^{\frac 1 6}X,\beta^{\frac 1 6}Y)$  is tight almost surely. Besides, if $\tau_\beta$ is an accumulation point of  $\mun$
		 then
		$\tau_\beta$ converges when $\beta\to\infty$ towards the  commutative distribution  of the first two coordinates of a point taken uniformly in  $B_{\R^3}(0,2^{1/4}(3\pi)^{1/6})$.
		In other words, for any $P\in \mathbb C\langle X,Y\rangle$, 		
		 almost surely,
		$$\lim_{\beta\to\infty}\limsup_{N\to\infty}\frac{1}{N}\Tr P(\beta^{\frac{1}{6}}X^N, \beta^{\frac{1}{6}}Y^N)=\frac{\int_{B_{\R^3}(0,(3\pi)^{1/3})}P(x,y)dxdydz}{|B_{\R^3}(0,(3\pi)^{1/3})|}$$
		and the same result holds if we replace  the $\limsup$ by a $\liminf$.

\end{prop}
Remark that the uniform law on $B_{\R^3}(0,(3\pi)^{1/3})$   is in a way a 2-dimensional analog of the semi-circular law  seen as the first coordinate of a uniform point on $B_{\R^2}(0,2)$.

\begin{proof}
	This is a consequence of the previous results. We know that $\mun_{\beta^{\frac 1 6}X, \beta^{\frac 1 6}Y}$ is tight ( since both $X^{N}$ and $Y^{N}$ have a bounded spectral norm with large probability by \eqref{boundRadius})  and let $\tau_{\beta}$ be a limit point. By the previous proposition,  the distribution of $x$ and $y$ under $\tau_{\beta}$ is the minimizer $\mu^{x^{2}/2}_{\beta}$ of the rate function $J_{\beta}^{x^{2}/2}$. By the second point of Proposition \ref{energyproblempotential}, we know that the rescaling $\beta^{1/6}\#\mu^{x^{2}/2}_{\beta}$ of the equilibrium measure by $\beta^{1/6}$ is supported in $[-L_{V},L_{V}]$. Therefore the distribution of $(\beta^{\frac{1}{6}}x, \beta^{\frac{1}{6}}y)$ under $\tau_{\beta}$ is tight. Let $\tau_{\infty}$ be a limit point. 
	Moreover,  $-[X^N,Y^N]^2$ is a non-negative matrix  with a first moment converging to $0$ at speed $\frac{\ln \beta}{\beta}$ by Proposition \ref{bound-first-trace}. Therefore
	$\tau_{\beta}(-[\beta^{\frac{1}{6}}x,\beta^{\frac{1}{6}}y]^2)=O(\ln \beta \beta^{-\frac{1}{3}})$  still goes to $0$ when $\beta$ goes to infinity. Hence, $\tau_{\infty}$ is the distribution of commuting variables. 
	{ Lemma  \ref{energyproblempotential}.4) shows that the distribution of $x$ under $\tau_{\infty}$  is the probability distribution with density $(2\pi)^{-1}(A-x^2/2)_+$ with  $A=\frac{1}{2}(3\pi)^{2/3}$. }
	{
Finally there is only one law of two commutative variables  on $\RR^2$ with a distribution invariant by rotation with a fixed marginal distribution. Indeed, the law of two commutative variables is encoded by a probability measure $P$ on $\mathbb R^{2}$. Because of the invariance by rotation of $P$, we see that  for any vector $\lambda\in \mathbb R^{2}$, $\langle \lambda, (x, y)\rangle$ has the same law that $\|\lambda\|_{2}x$ if $(x,y)$ follows $P$, which is thus uniquely defined if the law of $x$ is prescribed.
Since the distribution of the first two coordinates of an uniform points $B_{\R^3}(0,(3\pi)^{1/3})$ is invariant by rotation and they have this marginal $(2\pi)^{-1}(A-x^2/2)_+$, this must be the distribution $\tau_{\infty}$.
	
	}

\end{proof}

\begin{theorem}\label{maincomm}We let  { $V_{1}$ and $V_{2}$ be two non-negative polynomials of one variable, going to $+\infty$  at infinity. Let
	 $z_1^{j}<z_2<\dots<z_{k_{j}}^{j}$ be the zeroes of $V_{j}$  and assume there exists  $c_1^{j},\dots,c_k^{j}\in\mathbb R^{+*}$ 
	such that, for every $i\in \{1,\ldots, k_{j}\}$,
	$$V_{j}(x)= c_i^{j}(x-z_i^{j})^2+o(x-z_i^{j})^2$$
	%We moreover assume   that $x\rightarrow W_{1}(x)=W(x)-x^2/2$ is Lipschitz
	Let
	$$d\Q^N_{V_\beta}(X_{1},X_{2})=(\mathbb Z^N_{V_\beta})^{-1}e^{-N\Tr(-\beta[X_{1},X_{2}]^2+V_{1}(X_{1})+ V_{2}(X_{2}))}dX_{1} dX_{2}\,.$$
	Then, for every $\varepsilon>0$, $\Q^N_{V_\beta}$ almost surely, we have
	\begin{equation}\label{res1}\lim_{\beta\rightarrow \infty}\liminf_{N\rightarrow\infty } \hat \mu_{X_{i}}^N(\cup_j[z_j^{i}-\varepsilon, z_j^{i}+\varepsilon])=1\end{equation}
Moreover,	 for every $j\in \{1,\ldots,k\}$ and $\varepsilon>0$, we have  almost surely
	$$\lim_{\beta\rightarrow\infty}\liminf_{N\rightarrow\infty} \hat \mu_{X_{i}}^N([z_j^{i}-\varepsilon, z_j^{i}+\varepsilon])=\alpha_{j}^{i}:=\frac{(c_j^{i})^{-\frac{1}{2}}}{\sum_\ell (c^{i}_\ell)^{-\frac{1}{2}}}$$}
	and the above $\liminf$ can be replaced by a $\limsup$. 
	Besides,
	$$\left(\frac{2\beta N}{\pi}\right)^{N(N-1)/2} Z^N_{V_\beta}=\exp\left(-N^2(c\beta^{-\frac{1}{3}}+o(\beta^{\frac{-1}{3}}))+o(N^2)\right)$$
	where $c=c_{V_{1}}(\sum_j c_j^{-\frac{1}{2}})^{-\frac{2}{3}}$ and $c_{V_{1}}$ is a constant depending only on $V_{1}$ defined in Lemma \ref{energyproblempotential}.

Besides when $N$ and then $\beta$ goes to $\infty$ the joint non-commutative distribution of $X_{1},X_{2}$ exists: they asymptotically commute and  their spectrum become independent: under $\Q^N_{V_\beta}$, almost surely
$$\lim_{\beta\rightarrow \infty}\liminf_{N\rightarrow\infty } \frac{1}{N}\Tr(1_{X_{2}\in[z_j^{2}-\varepsilon, z_j^{2}+\varepsilon]}1_{X_{1}\in[z^{1}_\ell-\varepsilon, z^{1}_\ell+\varepsilon]})=\alpha^{2}_j\alpha^{1}_\ell\,.$$

%with $\gamma^{*}_\ell$ the analog of $\alpha^{*}_j$ for $X$.

\end{theorem}

Theorem \ref{theocommutator} follows from Theorem \ref{maincomm} by assuming that the integer numbers $p_{i}^{j}$ are all equal to $2$.%Note here that $V_{2}$ is a priori not a polynomial (or we should only assume that $V_{2}$ is locally Lipschitz.:Should we extend to polynomial??

\begin{proof} In the proof we will denote $X_{1}=X$, $X_{2}=Y$, $V_{1}=W$ and $V_{2}=V$  to use the same notations as in the previous results. 
 Moreover $c_{j}^{2}=c_{j}$ and $z_{j}^{2}=z_{j}$. 
 
{\bf Compactification : }
{We first show that we can replace $W,V$ by uniformly Lipschitz finite variation  functions, up to a quadratic term. 
First notice that our assumption on $V$ and $W$ and Theorem \ref{conf-to-T} implies that the potential $V_{\beta}(X,Y)=-\beta[X,Y]^2+W(X)+ V(Y)$ is $(\eta,A)$-trapping for some $\eta>0$, $A$ finite and the trivial partition $\{\{1\},\{2\}\}$.  We may assume $A\ge 0$ without loss of generality. Therefore, Theorem \ref{Ttosupport} implies that the operator norms of both $X$ and $Y$ stay bounded by $L^{\frac{1}{\kappa}+\frac{1}{2}}$ with overwhelming probability, with some $L$ depending only on $\eta$ and $A$. As a consequence, taking $M\ge L^{\frac{1}{\kappa}+\frac{1}{2}}$, 
we can replace $\Q^N_{V_\beta}$ by $\Q^{N,M}_{V_\beta}=1_{B_{M}}\Q^N_{V_\beta}/\Q^N_{V_\beta}(B_{M})$ where $B_{M}=\{\|X\|_{\infty}\le M\}\cap \{\|Y\|_{\infty}\le M\}$. 
On $B_{M}$, we can replace $V$ and $W$ by $\tilde V$ and $\tilde W$ given for any real numbers $x,y$ by: 
$$\tilde V(y)=\frac{\eta}{4}y^{2}+V(\varphi_{M}(y))-\frac{\eta}{4} \varphi_{M}(y)^{2},  \tilde W(x)=\frac{\eta}{4}x^{2}+W(\varphi_{M}(x))-\frac{\eta}{4} \varphi_{M}(x)^{2}$$
where $\varphi_{M}(x)=x$ on $|x|\le M$. We choose next  $\varphi_{M}$ outside of the ball so that  that $\varphi_{M}(x) x$ is non negative for every real number $x$, $\varphi_{M}$ is  increasing  and converging to  $\pm M'$ when $x$ goes to $\pm\infty$ with some $M'>M$ (e.g $M+1$).
We now wish to remove the restriction to the ball $B_{M}$ to consider instead $\Q^N_{\tilde V_\beta}$ with $\tilde V_{\beta}(X,Y)=-\beta[X,Y]^2+\tilde W(X)+ \tilde V(Y)$. To this end we need to check that we can choose $M,M'$ large enough so that $\Q^N_{\tilde V_\beta}(B_{M}^{c})$ is very small (in particular summable so that Borel-Cantelli's lemma applies) . It is enough to show that $\tilde V_{\beta}$ is also $(\eta',A')$-trapping again by Theorem \ref{Ttosupport}.$(\eta',A')$ must not depend on $M$ but can be different from $(\eta,A)$ (up to take $M$ large enough).  We have seen that the commutator is $(0,0)$-trapping so we only need to check that the functions $\tilde V$ and $\tilde W$ are $(\eta/2,A)$-trapping. This amounts to show that for every real number $x$, 
\begin{equation}\label{traps} x\tilde V'(x)=\frac{\eta }{2} x^{2}+(V'(\varphi_{M}(x))-\frac{\eta}{2} \varphi_{M}(x)) x \varphi_{M}'(x)
\ge \frac{\eta}{2} x^{2}-A\,.\end{equation}
But our choice of $\varphi_{M}$ insures that  this is true for $|x|\le M$. Moreover for every $x$, $x\varphi_{M}(x) \varphi_{M}'(x)$ is non-negative and moreover, since $V$ is $(\eta,A)$-trapping
$$\left(V'(\varphi_{M}(x))-\frac{\eta}{2} \varphi_{M}(x)\right)\varphi_{M}(x) \ge \frac{\eta}{2} \varphi_{M}(x)^{2}-A$$
Hence the LHS is non-negative as soon as $ M^{2}\ge 2A/\eta$ which shows that \eqref{traps} is true also for $|x|\ge M$. Assuming this, we conclude that $\tilde V_{\beta}$ is $(\eta/2,A)$-trapping so that again up to choose $M$ slightly bigger, we conclude that for $M$ large enough (depending on $A$ and $\eta$ only), $\Q^N_{\tilde V_\beta}(B_{M}^{c})$ goes to zero exponentially fast by \eqref{zx}. 
 Finally, observe that $V(\varphi_{M}(Y))-\frac{\eta}{4} \varphi_{M}(Y)^{2}$ and
$W(\varphi_{M}(y))-\frac{\eta}{4} \varphi_{M}(y)^{2}$ are uniformly Lipschitz and with finite variation.
Hence, we may and shall assume in the sequel that $V$ and $W$ are the sum of a quadratic potential and a Lipschitz function with finite variations.}

\medskip

{\bf Lower bound : }
Let $z_1<\dots<z_k$ be the minimizers of $V$.
	Choose $\eta=\frac{1}{2}\min_{i\neq j} |z_i-z_j|$ and $x_\beta=\min(\frac{1}{\ln \beta},\eta/2)$.
	Thus $2x_\beta +\eta\leqslant\min_{i\neq j} |z_i-z_j|$ and $x_\beta$ goes to zero when $\beta$ goes to infinity. We first condition by the number of eigenvalues $(\lambda_{i})_{1\le i\le N}$ of $Y$ close to each of these minimizers. 
	For ${\mathbf {n}}=(n_1,\dots,n_k)$ such that $\sum n_i=N$, we denote  by $$I_{\bf{n}}:=\cap_{r}\{ \sharp\{i : |\lambda_i-z_r|<x_\beta\}=n_r\}.$$
	 Then for any  ${\mathbf {n}}$ such that $\sum n_i=N$, 
	$$\left(\frac{2\beta N}{\pi}\right)^{N(N-1)/2} Z^N_{V_{\beta}}\geqslant Z_{V_{\beta}}^{\bf n}$$
	where we have set :
	\begin{eqnarray*}
	 Z_{V_{\beta}}^{\bf n}
	&:=&\left(\frac{2\beta N}{\pi}\right)^{N(N-1)/2}   \int_{I_{\mathbf{n}}}e^{-N\Tr(-\beta[X,Y]^2+W(X)+ V(Y))}dX dY\,.\end{eqnarray*}
Next we write the integral as an integral on the eigenvalues of $Y$: 
	$$Z^{\bf n}_{V_{\beta}}\!=\!\!
	\int_{I_{\bf n}}
		\prod_{i<j}\frac{2\beta N(\lambda_j-\lambda_i)^2}{\pi}e^{-N(\beta\sum_{i, j}(\lambda_i-\lambda_j)^2X_{ij}^2+\Tr W(X)+\sum V(\lambda_i))}dX  \prod_id\lambda_i$$
	
	We denote $i\sim j$ if for some $r\in \{1,\ldots, k\}$ we have $\lambda_i,\lambda_j\in[z_{r}-x_{\beta},z_{r}+x_{\beta}]$. We write $i\nsim j$ otherwise. We set
	$$\tilde X_{ij}=1_{i\sim j} X_{ij}.$$
	$\tilde X$ is a  block  diagonal matrix whose blocks on the diagonal $X(1),\dots,X(k)$ are of size $n_1,\dots,n_k$.
	We can also recover directly  $X(r)$ by writing $X(r)=\pi_r^* X \pi_r$ with $\pi_r$ the $N\times n_r$ matrix whose coefficients are $(\pi_r)_{ij}=1$ iff $\lambda_i$ is the $j$th eigenvalue in $[z_{r}-x_{\beta},z_{r}+x_{\beta}]$ and $0$ otherwise.
	
	Because we assumed that $W(x)= a x^{2}+\vartheta(x)$ with some uniformly Lipschitz function $\vartheta$ and some $a>0$, we know by Hoffman-Wielandt's lemma \cite{AGZ}[Lemma 2.1.19] that there exists a finite constant $C$ such that
		\begin{eqnarray*}
		\frac{1}{N}\Tr  W(X) - \frac{1}{N}\Tr W(\tilde X) &\leqslant  & a  \frac{1}{N}\Tr \left( (X)^{2}-(\tilde X)^2\right)+
		C( \frac{1}{N}\Tr  (X-\tilde X)^2)^{\frac{1}{2}} \\
		&\leqslant& \frac{C^2}{\sqrt{\beta}-a/2}+ \frac{\sqrt{\beta}}{N}\Tr  (X-\tilde X)^2\end{eqnarray*}
%(Here $\sqrt{\beta}$ could have been replaced by any constant but it turns out that this will be the best choice)
	where we finally assumed without loss of generality that $\beta> a^{2}/4$ and used that $\Tr X^{2}-\Tr \tilde X^{2}=\Tr(X-\tilde X)^{2}$. As a consequence, we get
		$$Z^{\bf n}_{V_{\beta}}\ge  e^{O(\frac{N^2}{\sqrt{\beta}})}\int_{I_{\bf n}
	}\prod_{i<j}\left(\frac{2\beta N(\lambda_j-\lambda_i)^2}{\pi}\right)\qquad\qquad\qquad \qquad\qquad$$
	$$\qquad \times  e^{-N(\beta\sum_{i, j}(\lambda_i-\lambda_j)^2X_{ij}^2+\Tr( W(\tilde X)+\sqrt{\beta}
	(X-\tilde X)^2)+\sum V(\lambda_i))} dX  \prod_id\lambda_i$$
Therefore, integrating over the entries of $X-\tilde X$, we deduce:

	\begin{align*}Z^{\bf n}_{V_{\beta}}
		&\ge 
	e^{O(\frac{N^2}{\sqrt{\beta}})}\int_{I_{\bf n}
	}\prod_{i\sim j}\left(\frac{2\beta N(\lambda_j-\lambda_i)^2}{\pi}\right)\prod_{i\nsim j}\left(\frac{\beta N(\lambda_j-\lambda_i)^2}{\beta N(\lambda_j-\lambda_i)^2+N\sqrt{\beta}}\right)\\
	&\times \exp\left(-N\beta \sum_{i\sim j}(\lambda_i-\lambda_j)^2X_{ij}^2-N\tr( W(\tilde X)+V(Y))\right)d\tilde X \prod_id\lambda_i
	\end{align*}

	Then we break $\tilde X$ in  the blocks of size $n_r$ along the matrices $X(r)$ 
	observing that the factor $\frac{\beta N(\lambda_j-\lambda_i)^2}{\beta N(\lambda_j-\lambda_i)^2+N\sqrt{\beta}}$ is uniformly bounded below by  $e^{-\frac{1}{\sqrt{\beta}\eta^2}}$ when $i\nsim j$. Moreover, there is $\varepsilon_{\beta}=O(x_{\beta})$ so that   if $c_{i}(\beta)=(1+\varepsilon_{\beta})c_{i}$, for any $\lambda\in [z_{i}-x_\beta, z_{i}+x_\beta]$, $V(\lambda)\le c_{i}(\beta)(z_{i}-\lambda)^{2}$. Therefore, we find
	
	\begin{align*}
		&Z^{\bf n}_{V_{\beta}}\ge e^{O(\frac{N^2}{\eta\sqrt{\beta}})}\prod_{r}\int_{[z_{r}-x_{\beta},z_{r}+x_{\beta}]^{n_{r}}}d X(r) \prod_id\lambda_i
		\prod_{1\le i<j\le n_{r}}\left(\frac{2\beta N(\lambda_j-\lambda_i)^2}{\pi}\right)\\
		&\ts\exp\{-N \beta\!\! \sum_{1\le i<j\le n_{r}}\!\!(\lambda_i-\lambda_j)^2X(r)_{ij}^2\!-\!N\tr W(X(r))-Nc_r(\beta)\sum_{i=1}^{n_{r}}(\lambda_i-z_r)^2\}.
	\end{align*}
	Now, we diagonalize the matrices $X(r)$ with eigenvalues $\{\mu_{i}\}_{1\le i\le n_{r}}$ and integrate over 
	  the  matrices $Y(r)$ with eigenvalues $\{\lambda_i-z_{r}, 1\le i\le n_{r}\}$ to obtain:
	  \begin{align*}
	Z^{\bf n}_{V_{\beta}}&\ge e^{O(\frac{N^2}{\sqrt{\beta}})}\prod_r  \int_{||Y(r)||_{\infty}<x_\beta}d Y(r) \prod_i d\mu_i \prod_{1\le i<j\le n_{r}}\left(\frac{2\beta N(\mu_j-\mu_i)^2}{\pi}\right)\\
	&\times \exp\{-N \beta \sum_{i,j}(\mu_i-\mu_j)^2Y(r)_{ij}^2-Nc_r (\beta) \tr Y(r)^2-N\sum_i W(\mu_i)\}\\\
\end{align*}

To compute the RHS note that  we can remove 
the indicator function on the set $||Y(r)||<x_\beta$ by 
	Proposition \ref{prop1} (\ref{boundRadius})  as long as {$x_\beta\gg \beta^{-\frac{p}{5(2p+1)}}$ with $p=\max p_{i}$. }This  is the case since  $x_{\beta}$ goes to zero like $1/\ln \beta$.
Then, we can compute the integral over $Y(r)$, getting, with the notations of Proposition \ref{prop1}, 
	\begin{align*}
		Z^{\bf n}_{V_{\beta}}&\ge e^{O(\frac{N^2}{\sqrt{\beta}})}\prod_r  \int \prod\frac{\beta(\mu_j-\mu_i)^2}{c_r(\beta)+\beta(\mu_j-\mu_i)^2}\exp\left(-N\sum W(\mu_i)\right)\prod_id\mu_i\\
		&=e^{O(\frac{N^2}{\sqrt{\beta}})}\prod_r \int\exp(-n_r^2J^{\frac{N}{n_r}W}_{\beta/c_r(\beta),\neq}(\frac{1}{n_{r}}\sum_{i=1}^{n_{r}} \delta_{\mu_{i}}))\prod_id\mu_i
	\end{align*} 
Finally if for every $r$, $n_r/N\to \alpha_r\in]0;1[$, we get  by Lemma \ref{energyproblempotential}.1) that 

	\begin{align}
	Z^{\bf n}_{V_{\beta}}&\ge  e^{O(\frac{N^2}{\sqrt{\beta}})
	+o(N^2)}\prod_r \exp(-n_r^2 c^{W/\alpha_r}_{\beta/c_r(\beta)})\nonumber\\
	%&= d_Ne^{O(\frac{N^2}{\sqrt{\beta}})+o(N^2)} \exp(-N^2 \sum_r \alpha_r^2 c_{W/x_r}(\beta/c_r(\beta))^{-\frac{1}{3}})\\
	&= e^{O(\frac{N^2}{\sqrt{\beta}})+o_{N} N^2 } {e^{-N^2 (1+o_{\beta})\beta^{-\frac{1}{3}}c_{W}\sum_r \alpha_r^{2-\frac{1}{3}} c_r(\beta)^{\frac{1}{3}}}} \,,\label{cv}
\end{align}
{ with $o_{\beta}$ (resp. $o_{N}$)  going to zero as $\beta$ (resp. $N$) goes to infinity.}
	At this stage recall that $c_{r}(\beta)$ goes to $c_{r}$ as $\beta$ goes to infinity.
Moreover observe that if $u_i$ and $k$ are positive constant, 
	the minimizer of $\alpha\rightarrow \sum_1^n \alpha_i^{1+k}u_i$ in the simplex is such that the $\alpha_i$'s are proportional to $u_i^{-1/k}$ and then the value of the minimum is $(\sum_i u_i^{-1/k})^{-k}$.

We apply this with $k=1-\frac{1}{3}=\frac{2}{3}$ and $u_r=c_r(\beta)^{\frac{1}{3}}\simeq c_{r}^{\frac{1}{3}}$ to get that
 the best lower bound is taken  at  $\alpha_r^{*}$ proportional to $c_r^{-\frac{1}{2}}$.
%This is intuitive since it gives more weight if $c_r$ is small i.e. if the potential is small. 
To conclude we obtained the lower bound
\begin{equation}\label{lblb}
\left(\frac{2\beta N}{\pi}\right)^{N(N-1)/2} Z^N_{V_{\beta}}
\geqslant 
\exp\left(
-N^2(1+o_{\beta}(1)) \beta^{-\frac{1}{3}}c_{W}
\left(\sum_r c_{r}^{-1/2}\right)^{{-2/3}}
\right).
\end{equation}
{\bf Upper bound : }
To get the corresponding upper bound, we first show that, under $\Q^{N}_{V_{\beta}}$,  the eigenvalues of $Y$ concentrate in a  neighborhood of the minimizers of $V$.
We recall  that $V_{\beta}$ is trapping by Theorem \ref{conf-to-T} and Remark \ref{remarktrapping} so that we can restrict ourselves to the ball  $B_M=\{\|X_1\|_\infty\le M, \|Y\|_\infty \le M\}$. Next, 
we take $\kappa>0$ and choose an interval $I$ such that $V|_I>\kappa>0$. We fix  a positive constant $\varepsilon$ and  bound from above
$$\Lambda_{N}(\varepsilon):=\left(\frac{2\beta N}{\pi}\right)^{N(N-1)/2} \int_{\{ \hat\mu^{N}_{Y}(I)\ge \varepsilon \}}e^{-N\Tr(-\beta[X,Y]^2+W(X)+ V(Y))}dX dY$$
by noticing that, since $V$ is non negative, $\tr V(Y)\ge N\varepsilon\kappa$  on $\{ \hat\mu_{Y}(I)\ge \varepsilon \}$. Hence, 
after diagonalizing $X$, we find that $\Lambda_{N}(\varepsilon)$ is bounded from above by 
$$ e^{- N^{2}\varepsilon\kappa}
\int \prod_{i<j}\left(\frac{2\beta N}{\pi}(\lambda_j-\lambda_i)^2\right) e^{-N\beta\sum_{i, j}(\lambda_i-\lambda_j)^2|Y_{ij}|^2-N\sum W(\lambda_{i})}dY  \prod_id\lambda_i.$$
Integrating over the $Y_{ij}$'s gives 
\begin{equation}\label{boundto}\Lambda_{N}(\varepsilon)\le e^{- N^{2}\varepsilon\kappa}
\int  \prod_id\lambda_i e^{-N\sum W(\lambda_{i})}\le e^{- N^{2}\varepsilon\kappa+O(N\ln N)}\end{equation}
Hence, together with \eqref{lblb}, 
we deduce that
$\Q^N_{V_\beta}(\hat\mu_{Y}^{N}(I)>\varepsilon)$ goes to zero  exponentially fast as soon as $\varepsilon\kappa \ge C \beta^{-1/3}$ if $C$ is large enough (and fixed hereafter).
 We take {$\kappa=\kappa(\varepsilon)$} small but fixed so that $I^{c}$ is included in a $\varepsilon$-neighborhood of the minimizers  and assume $\beta$ small enough $\varepsilon\kappa (\varepsilon) \ge C \beta^{-1/3}$.  { If there were $o(\beta^{-1/3} N)$ eigenvalues inside $I$, we could just remove these eigenvalues up to a negligible cost  (compared to the partition function) and find again ourselves with eigenvalues in independent wells.  To find ourselves in such a situation we shall use a discretization of the set $I$.}
Let $\eta$ be a small positive real number and  partition $[z_{i}+\varepsilon, z_{i}+2\varepsilon]$  and $[z_{i}-2\varepsilon, z_{i}-\varepsilon]$ into intervals of width $2\eta\ll \varepsilon$ for $i\in \{1,\ldots, k\}$. There are about $k \varepsilon/\eta$ such intervals and their union lies inside $I$ which has at most $C\beta^{-1/3}N$ eigenvalues. Therefore, 
on the set $\{ \hat\mu_{Y}(I)\le C\beta^{-1/3} \}$, we know that there must exists an interval $J_{i}^{+}=[z_{i} +\varepsilon+2\ell^{+}_{i}\eta ,z_{i}+\varepsilon+2(\ell^{+}_{i}+1)\eta]$  and $J_{i}^{-}=[z_{i}-\varepsilon-2\ell^{-}_{i}\eta ,z_{i}-\varepsilon-2(\ell_{i}^{-}-1)\eta]$ for some integer $\ell_{i}^{\pm}$  in which there are at most $2C\eta \beta^{-1/3} N$ eigenvalues. Even though $\ell_{i}^{\pm}$ are random, 
 we have at most $(C\varepsilon/2\eta)^{2k}$ choices for these numbers and therefore can just optimize over their choice at the end. We hence consider these intervals fixed. We let $V_{i}=[z_{i}-\varepsilon-2(\ell_{i}^{-}-1)\eta, z_{i} +\varepsilon+2\ell^{+}_{i}\eta]$ and $F_{i}= [z_{i}+\varepsilon+2(\ell^{+}_{i}+1)\eta, z_{i+1}-\varepsilon-2\ell^{-}_{i+1}\eta] $,  $F_{0}=(-\infty, z_{1}-\varepsilon-2\ell^{-}_{1}]$ and  $F_{k}=[z_{k}+\varepsilon+2\ell^{+}_{k}\eta,+\infty)$.  
 We   let $k\simeq j$ (resp. $k\sim j$) iff there exists $i$ such that  both $\lambda_{k},\lambda_{\ell}$ belong to the same  $V_{i}$ or the same $F_{i}$
  (resp.  $\lambda_{k},\lambda_{\ell}$ both belong to the same  $J_{i}^{+}$ or in $J_{i}^{-}$). Then we set

$$\tilde X_{k\ell}=X_{k\ell}1_{k\simeq \ell}, \bar X_{k\ell}=  X_{k\ell}1_{k\sim \ell} \,.$$
Then, $\bar X$ has at most rank  { $C\eta \beta^{-1/3}N$ } so that since $W(x)=a x^2+\vartheta (x)$ with $\vartheta$ with finite variations, 
$$\Tr W(X)-\Tr W(X-\bar X)\ge a \Tr \bar X^2 - \|\vartheta\|_{TV} N C\eta \beta^{-1/3}\,.$$
Moreover, because $\vartheta$ is Lipschitz,
$$\Tr W(X-\bar X)-\Tr W(\tilde X)\ge  a \Tr (X-\bar X-\tilde X)^2- \|\vartheta\|_L\sqrt{N}\left(\Tr | X-\bar X-\tilde X|^2\right)^{1/2}\,.$$
Therefore if we let $I_{\bf{n},\bf{m}}$ be the set where there are $n_i$ eigenvalues in $V_{i}$, $m_{i}$ eigenvalues in $ F_{i}$, $\sum (n_i+m_{i}) \ge N-C\eta \beta^{-1/3}N$, we get after integrating over the entries of $\bar X$ (which are very few) and the entries of $X-\tilde X-\bar X$ (which are such that $|\lambda_{i}-\lambda_{j}|\ge 2\eta$. Recall we only imposed that $\varepsilon \gg \eta$ with $\varepsilon$ fixed  such that $\varepsilon \kappa(\epsilon)\gg \beta^{-1/3}$. We choose $\eta\gg C \beta^{-1/12}$ so that the contribution of the expectations over the entries of $ X-\bar X-\tilde X$
is negligible (i.e such that  $(\sqrt{\beta}\eta^{2})^{-1}\ll \beta^{-1/3}$, that is $\eta\gg\beta^{-1/12}$ )))

\begin{align}&Q^N_{V_\beta}(I_{\bf n,m }
)	 \le 
	e^{cN^{2}\beta^{{-1/3}}(1+o_{\beta,\eta})}
	\int_{I_{\bf n,\bf m}}
	\prod_{i\simeq j}\left(\frac{2\beta N(\lambda_j-\lambda_i)^2}{\pi}\right)\label{defR}\\
	&\,\times \exp\left(-N\beta \sum_{i\simeq j}(\lambda_i-\lambda_j)^2X_{ij}^2-N\tr( W(\tilde X)) -N\sum V(\lambda_{i}))\right)d\tilde X \prod_id\lambda_i\nonumber
	\end{align}
where $o_{\beta,\eta}$ goes to zero when $\beta$ goes to infinity and $\eta$ to zero. We can next decompose $\tilde X$ in block diagonal matrices corresponding to entries with eigenvalues either both in a $V_{i}$ or in a $F_{i}$ and remark that the above integral factorizes and so for $B_{r}=V_{r}$ or $F_{r}$, we have

 \begin{eqnarray*}&&Q^N_{V_\beta}(I_{\bf n,m})
 \le
	e^{cN^{2}\beta^{{-1/3}}(1+o_{\beta,\eta})} \prod_{1\le r\le k}  R_{n_r, r}(V_{r})\prod_{0\le r\le k}R_{m_r, r}(F_{r})\\
&&R_{n_r, r}(B_r):=	
	\int_{\mbox{spec}(Y)\in B_r }
	\prod_{1\le i<j\le n_r}\left(\frac{2\beta N(\mu_j-\mu_i)^2}{\pi}\right)\\
	&&\,\times \exp\left(-N\beta \sum_{}(\mu_i-\mu_j)^2Y_{ij}^2-N\sum W(\mu_i) -N\Tr (V(Y))\right)dY \prod_id\mu_i\\
	\end{eqnarray*}
	We can bound $R_{n_r, r}(F_r)$ exactly as in \eqref{boundto} as $V$ is lower bounded by $\delta_{\varepsilon}=\inf_{|x-z_{i}|\ge \varepsilon}V(x)$ on $F_{k}$. On $V_{r}$, the eigenvalues are close to $z_{r} $ and we can therefore approximate $V$ by $c_{r}(\varepsilon)(x-z_{r})^{2}$ with some $c_{r}(\varepsilon)$ going to $c_{r}$ as $\varepsilon$ goes to zero and use Proposition \ref{prop1}. Hence, we find
	 $$Q^N_{V_\beta}(I_{\bf n,m})
 \le
	e^{cN^{2}\beta^{{-1/3}}(1+o_{\beta,\eta})} {e^{-\delta_{\varepsilon}\sum m^{2}_{i}}}\prod_{1\le r\le k} e^{-n_r^2 c^{NW/n_r}_{2\beta/c_r(\varepsilon)}}$$
	For $\beta$ large enough,  the above right hand side is maximized when $m_r=0$ (as $\delta_\varepsilon$ is bounded below independently of $\beta$) and the probability that $\max m_{i}\ge 1$ goes exponentially fast to zero. Moreover, 
thanks to Lemma \ref{energyproblempotential}.1) 

$$Q^N_{V_\beta}(I_{\bf n})\le e^{cN^{2}\beta^{{-1/3}}(1+o_{\beta,\eta})}  e^{-N^2\beta^{-1/3} c_W\sum_r c_r^{1/3}(\frac{n_r}{N})^{5/3}}$$
from which we see by summing over all possible filling fractions $n_r/N$ that it converges towards the optimizer of $\alpha\rightarrow \sum_r c_r^{1/3}(\alpha_r)^{5/3}$ on $\sum \alpha_{i}=1$, which is given by $(c_{r}^{-1/2}/\sum c_{r}^{-1/2})_{1\le r\le k}$.

This settles the asymptotics for the  filling fractions of the eigenvalues of $X=X_{1}$. We can do exactly the same proof for the matrix $Y=X_{2}$ to obtain the asymptotics of its filling fractions. Moreover, by Corollary \ref{homo}, for every $\varepsilon>0$
$$\limsup_{\beta\rightarrow \infty}\limsup_{\beta\rightarrow \infty}Q^N_{V_\beta}(\frac{1}{N}\Tr (|[X,Y]|^{2})>\varepsilon)=0$$
so that $X$ and $Y$ asymptotically commute. 
 To get information on the joint behavior of the spectrum of $X$ and $Y $,  we improve the upper bound to control the behavior of $\tilde X$. 
  This is sufficient to get information on $\tilde X$ since $X-\bar X-\tilde{X}$ is negligable and $\bar X$ of small rank (when $\beta$ is large). Now we use the same machinery, let us call $\tilde{X}(1),\dots,\tilde{X}(k)$ the blocks of $\tilde{X}$:
$$\tilde{X}(r)=1_{Y\in V_r}\tilde{X}1_{Y\in V_r}=1_{Y\in V_r}{X}1_{Y\in V_r}$$
where $1_{Y\in V_r}$ is the orthogonal projector on the eigenspace associated with all eigenvalues of $Y$ inside $V_r$. 
 Indeed, by definition, provided $0\notin [z_{i}^{1}-\varepsilon,z_{j}^{1}+\varepsilon]$, 
 \begin{eqnarray*}
 \frac{1}{N}\Tr( 1_{Y\in [z_{r}^{2}-\varepsilon,z_{r}^{2}+\varepsilon]}1_{X\in [z_{i}^{1}-\varepsilon,z_{j}^{1}+\varepsilon]})&=&\frac{1}{N}\Tr(
1_{\tilde X(r)\in [z_{i}^{1}-\varepsilon,z_{j}^{1}+\varepsilon]})\\
&=&\mu^N_{\tilde{X}(r)}([z_{i}^{1}-\varepsilon,z_{j}^{1}+\varepsilon]) \,.\end{eqnarray*}
 Then a slight modification of the bound \eqref{defR} shows that for any measurable set $A$ in the set of probability measures on the real line, 
$$Q^N_{V_\beta}(I_{\bf n,m}\cap_r \mu^N_{\tilde{X}(r)}\in A)
 \le
	e^{cN^{2}\beta^{{-1/3}}(1+o_{\beta,\eta})} \prod_{1\le r\le k}  \tilde R_{n_r, r}(V_{r})\prod_{0\le r\le k}R_{m_r, r}(F_{r})
$$
with
\begin{align*}
&\tilde R_{n_r, r}(B_r):=	
	\int_{\mbox{spec}(Y)\in B_r,\mu^N_{\tilde{X}(r)}\in A }
	\prod_{1\le i<j\le n_r}\left(\frac{2\beta N(\mu_j-\mu_i)^2}{\pi}\right)\\
	&\,\times \exp\left(-N\beta \sum_{}(\mu_i-\mu_j)^2Y_{ij}^2-N\sum W(\mu_i) -N\Tr (V(Y))\right)dY \prod_id\mu_i
\end{align*}
Again we can approximate $V$ by a quadratic potential on the set $V_{r}$ and use Proposition \ref{prop1} which shows that it $A$ does not contains a neighborhood of $\mu^{NWN/n_r}_{2\beta/c_r(\varepsilon)}$ then $\tilde R_{n_r, r}(V_r)$ will be much smaller than $R_{n_r, r}(V_r)$ so that 
$Q^N_{V_\beta}(I_{\bf n,m}\cap_r \mu^N_{\tilde{X}(r)}\in A)$ will go to zero.  This completes the proof of the Theorem.

\end{proof}
\section{Appendix}

\subsection{Random matrix inequalities by flow}\label{appineq}
In this paper, many controls  are based on change of variables. The most basic one is the following:
\begin{lemma}[{\bf Inequality by mapping}]\label{change-variable}Let $f:\RR^d\mapsto\RR^{+}$ be a measurable function with finite $L^1$ norm $\int f(x)dx$.
Let $X$ be a $\RR^d$ valued random variable  with distribution $\mu$ absolutely continuous with respect to Lebesgue measure and  density proportional to $f$.
If $\phi:A\to\RR^d$ is a $\mathcal C^1$ diffeomorphism onto its image then
$$\P(X\in A)\leqslant \sup_{x\in A} \frac{f(x)}{f\circ\phi (x) J_{\phi}(x)}$$
where $J_\phi$ is the Jacobian of $\phi$ : $J_\phi(x)=|\det ( \partial_i \phi_j(x))|$. 
\end{lemma}
\begin{proof}
If $f$ is the density of $\mu$ (up to multiplication by a constant $c$)
\begin{align*}
\P(X\in A)= \int_{A} cf(x)dx 
&\leqslant \sup_{x\in A} \frac{f(x)}{f(\phi(x))J_{\phi}(x)}\int_{A} cf(\phi(x))J_{\phi}(x)dx\\
&\leqslant \sup_{x\in A} \frac{f(x)}{f(\phi(x))J_{\phi}(x)}\P(\phi^{-1}(X)\in A)
\end{align*}
where we did the change of variable $x\to\phi(x)$. The lemma follows since the Jacobian and $f$ are non negative. 
\end{proof}

Change of variables were already used to derive concentration of measure, see e.g. \cite[Section 6.2]{GN}, in the context of unitary matrices. However, this strategy may be difficult to apply in { non compact manifolds } as the reminders may be large. The strategy to use flows allows to bypass this issue.

\begin{lemma}[{\bf Inequality by flow}]\label{change-variable-flow}Let 
	$X$ be a $\RR^d$ valued random variable  with distribution $\mu(dx)=c e^{-W(x)}dx$  with $W:\RR^{d}\rightarrow\RR$ a $\mathcal{C}^1$ function going to infinity fast enough so that $c^{-1}=\int e^{-W(x)}dx$ is finite.
	Assume that there exists a flow $(Y^x_t)_{x\in\RR^d,t\in \RR^+}$ on $\RR^d$ such that
	$Y^x_0=x$ and $\partial_t Y^x_t = Z_{Y^x_t}$ with $Z_x$ a $\mathcal{C}^1$ vector field  from { $\RR^{d}$ into itself}. We fix $t>0$ and assume  that for $s\le t$,  $x\rightarrow Y^x_s$ is a diffeomorphism onto its image.
	Then, for any finite time $t\ge 0$, 
	$$\P(X\in A)\leqslant \sup_{x\in A} \exp\left(\int_0^t(\langle \nabla_{Y^x_s} W.Z_{Y^x_s}\rangle-\Tr \Diff_{Y^x_s} Z)ds\right)$$
	where $\langle x.y\rangle$ is the scalar product on $\R^d$. 
\end{lemma}
\begin{proof}
	This is a direct application of the previous Lemma with $f(x)=\exp(-W(x))$ and $\phi(x)=Y^x_t$. Because we assumed that  $x\rightarrow Y^x_s$ is a diffeomorphism onto its image, for any given $x$ in a bounded subset of  $ \RR^{d}$, $Y^{x}_{t}$ stays bounded and therefore we may assume without loss of generality that $x\to Z_x$ is uniformly Lipschitz. 
	
	We just need to compute
\begin{align*}	\frac{f(x)}{f(\phi(x))}&=\exp(-W(x)+W(Y^x_t))
	=\exp(\int_{0}^{t} \partial_{s}W(Y^x_s) ds)\\
	&=\exp(\int_0^t\langle \nabla_{Y^x_s} W.Z_{Y^x_s}\rangle ds)
	\end{align*}

	%All estimates are uniform since the trajectory $Y^x_{[0;t]}$ is compact.
	Moreover, we can compute the Jacobian of this change of variables by noticing that
	$$ \partial_t \Diff Y^x_t=\Diff \partial_t  Y^x_t =\Diff Z_{Y^x_t}=\Diff_{Y^x_t} Z \Diff Y^x_t$$
	Now since for $A$ invertible:
	$\Diff_A \det . H = \det(A)\Tr(A^{-1}H)$, we deduce that
	\begin{align*}
			\partial_t \det \Diff Y^x_t &=\Diff_{Y^x_t} \det . \partial_t \Diff Y^x_t\\
			&=\det \Diff Y^x_t\Tr ((\Diff_{Y^x_t})^{-1}\Diff_{Y^x_t} Z \Diff Y^x_t)\\&=\det \Diff Y^x_t\Tr \Diff_{Y^x_t} Z 
	\end{align*}
This computation is true whenever $ \Diff Y^x_t $ is invertible i.e. $\det \Diff Y^x_t \neq 0$, which is true at $t=0$ and thus for small time $s$ till some time $u>0$.
Thus,  we can solve the ODE :
$$\det \Diff Y^x_u  =\exp(\int_0^u\Tr \Diff_{Y^x_s} Z ds)\,.$$
Because the RHS is uniformly bounded below as $Y^{x}_{s}$ stays uniformly bounded, we conclude that we can take $u=t$.
Finally, putting things together we get:

$$ \frac{f(x)}{f\circ\phi (x) J_{\phi}(x)}= \exp\left(\int_0^t(\langle \nabla_{Y^x_s} W.Z_{Y^x_s}\rangle -\Tr \Diff_{Y^x_s} Z)ds\right)$$
and we can apply the previous Lemma.

\end{proof}

We deduce from this an inequality by flow for matrix models.

\begin{lemma}[{\bf A random matrix inequality}]\label{change-variable-matrix}
	Let $X^N=(X^N_1,\dots,X^N_\ell)$ be random Hermitian matrices with law $\P$  proportional to
	$$\exp(-N\Tr V(X^N))dX^N.$$
	Let $X^N_t$ be the solution of $X^N_0=X^N$ and $\partial_t (X^N_i)_t = - g_i(X^N_t)$. Then
	\begin{equation}\label{cont1}\P(X^N\in A)\leqslant \sup_{X^N_0\in A} \exp(-\int_0^t(N\Tr \Da V(X^N_s) . g(X^N_s){-}\Tr\otimes \Tr (\partial g(X^N_s)))ds)\,.\end{equation}
	Here 
	 $A.B=\sum_{i=1}^\ell A_i B_i$ and 	$g_{i}(X)=P_{i}(f_{1}(X_{1}),\ldots, f_{\ell}(X_{\ell}))$ with 
	 self-adjoint non-commutative polynomials  $P_{1},\ldots,P_{\ell}$ and $C^{1}$ bounded  functions $f_{i}:\mathbb R\rightarrow \mathbb R$. Moreover, we set for $f$ $C^{1}$, $\partial f(X)=\int_{0}^{1} f'(\alpha X\otimes 1+(1-\alpha)1\otimes X) d\alpha$ 
	 and with $A\otimes B\#(C\otimes D)=AC\otimes DB$, 
	 
	 $$\partial g(X)=\sum_{i=1}^{\ell}(\partial_{X_{i}} P_{i})(f(X))\#\partial  f_{i}(X_{i})\,.$$

\end{lemma}

\begin{proof}
	This is again a direct application of the previous Lemma since $g$ is uniformly Lipschitz so that the solution of
	$$\partial_{t } X^{N}_{t}= Z_{X^{N}_{t}}$$
	with $Z_{X}=g(X)$ is such that $X^{N}_{0}\rightarrow X^{N}_{t}$ 
		 is a diffeomorphism onto its image. One notices that
	$$\partial_s N\Tr V(X^N_s)= -\nabla_{X^N_s} (N\Tr V). Z_{X^N_s}= -N\Tr \Da V (X^N_s) . g(X^N_s),$$
	whereas
	\begin{eqnarray*}
		\Tr\Diff_{X^N_s}Z_{X^N_s}&=&\sum_{p=1}^\ell \sum_{i,j=1}^N2^{-1_{i\neq j}}( \partial_{ \Re(X_p(i,j))} +i \partial_{\Im(X_p(i,j))})g_{p}(X)({i,j})(X^N_s)\\
		&=&
		\Tr\otimes \Tr \partial g(X^N_s)\end{eqnarray*}
	where the term $\partial_{\Im(X_p(i,i))}$ should be put to zero. Observing that $( \partial_{ \Re(X_p(i,j))} +i \partial_{\Im(X_p(i,j))})X_{k\ell}=2^{1_{i\neq j}}1_{ij=\ell k}$ gives the last equality and completes the proof. 
\end{proof}

\begin{lemma}\label{matrix-ODE} For matrices $X^N=(X^N_1,\dots,X^N_\ell)$ we define the following  dynamic in the space of $N\times N$  Hermitian matrices :
	$$
	\left\{\begin{array}{cll}
		X^N_i(0)&=&X^N_i\\
		Z^N(t) &=& (X^N_1(t))^2+\dots+ (X^N_\ell(t))^2\\
		\partial_t X_i^N(t) &=& -((Z^N(t))^kX_i^N(t)+X_i^N(t)(Z^N(t))^k)
	\end{array}
	\right.
	$$
	Then,  for any $t>0$, $X^{N}_{0}\rightarrow X^{N}_{t}$ is a diffeomorphism onto its image so that \eqref{cont1} extends to the polynomial $g_{i}(X)=(\sum X_{j}^{2})^{k} X_{i}+X_{i }(\sum X_{j}^{2})^{k} $. 
	Moreover,  for every $n$, $\frac{1}{N}\Tr (Z^N(t))^n$ is a decreasing function of  $t$. { Furthermore, for any positive real number $B$ such that $\frac{1}{N}\Tr (Z^N(0))^{k+1}>B$}, for $t\in [0,\frac{1}{N^{\frac{k}{k+1}}}]$,
	$$\frac{1}{N}\Tr (Z^N(t))^{k+1}\geqslant B- 4(k+1)t N^{\frac{k}{k+1}} B^{\frac{2k+1}{k+1}}\,.$$

\end{lemma}
\begin{proof} Because $g$ is locally Lipschitz, $X^{N}_{t}$ is well defined until a short time from any given initial condition. To show it is globally well-defined it is enough to show that the spectral radius of the solution remains bounded in time. But
	\begin{align*}
		\frac{\partial}{\partial t}\Tr (Z^N(t))^n
		&=-2n\Tr (Z^N(t))^{n+k}\\
		&-2n \sum_{i=1}^\ell \Tr \left(X^N_i(t)(Z^N(t))^kX^N_i(t)(Z^N(t))^{n-1}\right)
	\end{align*}
	which is non-positive since $\Tr X A X B=\Tr (B^{1/2}X A^{1/2})(B^{1/2}X A^{1/2})^* $ is non-negative  if $A$ and $B$ are non-negative and $X$ self-adjoint.
	Therefore $\Tr (Z^N(t))^n$ is decreasing in $t$, and hence the spectral radius of the solution is bounded by the one of the initial condition. This is enough to guarantee that $X^{N}_{0}\rightarrow X^{N}_{t}$ is a diffeomorphism onto its image and therefore extend the inequality \eqref{cont1} to this case.
	Besides, using Lemma \ref{matrix-inequ} for the second term in the above  inequality, we get 
	
	\begin{align*}
		\frac{\partial}{\partial t}\Tr (Z^N(t))^{k+1}&\geqslant -4(k+1)\Tr (Z^N(t))^{2k+1}\\
		&\geqslant -4(k+1)(||Z^N(t)||_\infty^{k+1})^{\frac{k}{k+1}}\Tr (Z^N(t))^{k+1}\\
		&\geqslant -4(k+1)(\Tr(Z^N(t))^{k+1})^{\frac{k}{k+1}}\Tr (Z^N(t))^{k+1}\\
		&= -4(k+1)(\Tr (Z^N(t))^{k+1})^{\frac{2k+1}{k+1}}.
	\end{align*}

	As a consequence,
	$y(t)=\Tr (Z^N(t))^{k+1}$ satisfies  the one dimensional differential inequality
	$$y'\geqslant-4(k+1)y^{\frac{2k+1}{k+1}}$$
	starting from $y(0)\geqslant NB$.
	Therefore, solving this differential inequality gives,
	$$y(t)\geqslant (4kt+y(0)^{-\frac{k}{k+1}})^{-\frac{k+1}{k}}= y(0) (1+ 4kt y(0)^{\frac{k}{k+1}})^{-\frac{k+1}{k}}\,.$$
	As a consequence, if $y(0)\ge BN$ and $t\le 1/N^{\frac{k}{k+1}}$, using that for $x\ge 0$,  $(1+x)^{-\frac{k+1}{k}} \ge 1-\frac{k+1}{k}x$, we get 
	\begin{eqnarray*}
		\Tr (Z^N(t))^{k+1}&=& y(t)\geqslant BN (1+ 4k t (BN)^{\frac{k}{k+1}})^{-\frac{k+1}{k}}\\
		&\geq& BN (1- 4(k+1)t (BN)^{\frac{k}{k+1}})\,.\\
		\end{eqnarray*}
	%Here we see that we can not go much further than $t=N^{-k/k+1}$ to get a non trivial bound.
\end{proof}

\subsection{Matrix inequalities}

\begin{lemma}\label{matrix-inequ}
	If $X$, $Y$ are two $N\times N$ hermitian matrices and $u$ and $v$ are integers then if either $u+v$ is even or $X$ is positive.
	$$\Tr(X^{u+v}Y^2)\geqslant \Tr(X^uYX^vY).$$
\end{lemma}
\begin{proof}
	To prove this, we may assume that $X$ is  a diagonal matrix. If the $x_a$'s are the eigenvalues of $X$ then,
	\begin{align*}
		\Tr(X^{u+v}Y^2)-\Tr(X^uYX^vY)&=\sum_{1\leqslant a,b\leqslant N}|Y_{ab}|^2x_a^{u+v}- \sum_{1\leqslant a,b\leqslant N}|Y_{ab}|^2x_a^{u}x_b^v\\
		&=\frac{1}{2}\sum_{1\leqslant a,b\leqslant N}|Y_{ab}|^2(x_a^u-x_b^u)(x_a^v-x_b^v)
	\end{align*}
	which is non-negative if the $x_a$'s are positive or if $u$ and $v$ have the same parity.
\end{proof}

\begin{lemma}\label{matrix-inequ-holder}
	For all $D\in \NN$, $\ell\in\NN^*$ for any Hermitian matrices, $X^N_1,\dots,X^N_\ell$ if $Z^N=(X^N_1)^2+\dots+(X^N_\ell)^2$,
	$$\Tr(( Z^N)^{k+D})\leqslant \ell^D\Tr ((Z^N)^k\sum_{i=1}^\ell (X^N_i)^{2D} ).$$
\end{lemma}

\begin{proof}
	A difficulty of this result is that in general $(X+Y)^{2D}$ is not smaller than $4^D(X^{2D}+Y^{2D})$ as in the commutative case.
	
	{\bf First step : $\mathbf{k=0}$. }We first prove the case $k=0$ by using the non-commutative H\"older's inequality: for any choice of $i_{1},\ldots,i_{2D}\in \{1,\ldots,\ell\}^{2D}$, 
	$$\Tr (X^N_{i_1}\dots X^N_{i_{2D}})\leqslant \max_i\|X^N_i\|_{2D}^{2D}.$$
	Therefore after expanding the product we get : 
	$$\Tr (Z^N)^D\leqslant \ell^D\max_i\Tr (X^N_i)^{2D}.$$
	{\bf Second step : general case }
	We apply the case $k=0$ to the new matrices
	$$\tilde X^N_i = ((Z^N)^{\frac{k}{2D}} (X^N_i)^2 (Z^N)^{\frac{k}{2D}})^{\frac{1}{2}}$$
	$$\tilde Z^N = (Z^N)^{\frac{k}{2D}} ((X^N_1)^2+\dots+(X^N_\ell)^2) (Z^N)^{\frac{k}{2D}}=(Z^N)^{\frac{k+D}{D}}$$
	Therefore,
	$$\Tr(( Z^N)^{k+D})=\Tr((\tilde Z^N)^D)\leqslant \ell^D\sum_{i=1}^\ell\Tr ((Z^N)^{\frac{k}{2D}} (X^N_i)^2 (Z^N)^{\frac{k}{2D}})^D.$$
	Finally we use the Araki-Lieb-Thirring inequality  \cite{kosaki} on the right hand side 
	$$\Tr(( Z^N)^{k+D})\leqslant \ell^D\sum_{i=1}^\ell\Tr ((Z^N)^{\frac{k}{2}} (X^N_i)^{2D} (Z^N)^{\frac{k}{2}}).$$
	
\end{proof}

\begin{lemma}\label{normbandmatrix}
	If $X^N$ is an Hermitian matrix with Gaussian independent centered entries with  variance uniformly bounded by $ K$ and such that
	$X^N_{ij}=0$ if $|i-j|>uN$, then  with probability greater than $1-\frac{2}{u}e^{-uN}$
	$$ \|X^N\|_{\infty}\leqslant CK\sqrt{uN}$$
	for some universal constant $C$.
		
\end{lemma}

\begin{proof}
Define $a^N_k =	\lfloor kuN\rfloor$
and  by induction (with convention $a^N_{-1}=-\infty$) define for $k\geqslant 0$,
$$A^{N,k}_{ij} = X^N_{ij}(1_{a^N_k\leqslant i,j< a^N_{k+2}}-1_{a^N_{k-1}\leqslant i,j< a^N_{k+1}})_+$$
Since $A^{N,k}_{ij}$ has only non zero coefficients in a block of size at most $2uN+1$ we deduce from \cite[Corollary 4.4.8]{vershynin} that
$$\P(\|A^{N,k}\|\geqslant CK(\sqrt{2uN}+t))\leqslant 4e^{-t^2}$$
Now 
$$X^{N}=A^{N}+B^{N}$$
with 
$$A^N=A^{N,0}+A^{N,2}+\dots+A^{N,2\lfloor (2u)^{-1}\rfloor+2}, B^N=A^{N,1}+A^{N,3}+\dots+A^{N,2\lfloor (2u)^{-1}\rfloor+3}$$
Moreover 
$A^{N}$
has independent blocks coming from each $A^{N,2k}$
\begin{eqnarray*}
&&P(\|A^{N}\|_{\infty}\geqslant CK(\sqrt{2uN}+t))%&=&\P(\sup_k \|A^{N,2k}\|_{\infty}\geqslant CK(\sqrt{2uN}+t))\\
\\
&&= 1-\prod_{k}(1- \P( \|A^{N,2k}\|_{\infty}\geqslant CK(\sqrt{2uN}+t)))\\
&&\quad \le  1-(1-\sup_{k } \P( \|A^{N,2k}\|_{\infty}\geqslant CK(\sqrt{2uN}+t))^{\lfloor (2u)^{-1}\rfloor  }
\le \frac{1}{u} e^{-t^2}\end{eqnarray*}

The same result holds for 
$B_{N}$ and therefore the result follows for 
 $X^N=A^N+B^N$.
Taking $t=\sqrt{uN}$  allows to conclude.

\end{proof}

\subsection{Change of variables in one-matrix models}

We  compute the Jacobian of a change of variables acting only on the eigenvalues of matrices.
\begin{lemma}[{\bf Jacobian of a function of hermitian matrices}]\label{Jacobian1} Let $f$ be a $C^1$ function from $\mathbb R$ into $\mathbb R$.
	Let $\phi$ be a function  from $\HNC$ into $\HNC$ acting on the eigenvalues
	$\lambda =( \lambda_1>\dots>\lambda_N)$ so that if $\Diag(\lambda)$ is the diagonal matrix with entries $\lambda$ and $U$ is any unitary matrix,
	$$\phi(U^*\Diag(\lambda)U)=U^*\Diag (f(\lambda))U\,.$$
	Then if  $\Delta(x)=\prod_{i<j}(x_j-x_i)$ is the VanderMonde determinant, then the Jacobian of $\phi$ is
	$$J_\phi=\left(\frac{\Delta(f(\lambda))}{\Delta(\lambda)}\right)^2e^{\Tr\ln f'(X)}.$$
\end{lemma}
\begin{remark} In fact we only need $f$ to be $C^1$ in a neighborhood of the spectrum.
\end{remark}
\begin{proof}
	We decompose the map $\phi$ in three parts : first the reduction of $X$ into its eigenvectors, eigenvalues. The Jacobian of this change of variable is $\Delta(\lambda)^{-2}$. Then $f$ is applied to each eigenvalue which creates the factor $\prod_i f'(\lambda_i)$. Finally the last factor comes from the map $(U,f(\Lambda))\to U^*f(\Lambda) U$.
	Roughly, if $W_X\subset V_X$ is the subset of hermitian matrices where $f$ is still $\mathcal C^1$ in each eigenvalue,
	\begin{align*}
		\int_{W_X} h(\phi(X)) J_\phi dX&=\int_{\phi(W_X)} h(X) dX
		=\int_{\phi(W_X)} h(U^*\Lambda U)\Delta(\Lambda)^2 d\Lambda dU\\
		&=\int_{W_X} h(U^*f(\Lambda) U)\Delta(f(\Lambda))^2 \prod f'(\lambda_i)d\Lambda dU\\
		&=\int_{W_X} h(f(U^*\Lambda) U))\Delta(f(\Lambda))^2 \prod f'(\lambda_i)d\Lambda dU\\
		&=\int_{W_X} h(f(X))\Delta(f(\Lambda))^2 \prod f'(\lambda_i)/\Delta(\Lambda)^2dX
	\end{align*}
	And we identify $J_\phi$ by comparing the first and last integral.
\end{proof}

The main result of this section is to construct a change of variables mapping the eigenvalues outside some interval $I$  to the neighborhood of some point and bound its Jacobian from below. 

\begin{lemma}\label{Jacobian2} Let $I
$ be an open set.
	For $n\in \{0,\ldots,N\}$ and $(x_1,\dots,x_n)\in I^n$, let 
	$\mathcal{H}^I_L(x_1,\dots,x_n)$ be the set of $X\in\HNC$  with eigenvalues in the interior of $I^c$ except for $n$ of them which are equal to $(x_1,\dots,x_n)$. Let $\alpha$ be a real number and fix $\varepsilon\in (0,4)$.
	Then,  there exists an injective map $\phi$ from  $\mathcal{H}^I_L(x_{1},\ldots,x_{n})$ into $\HNC$ such that if $$X=U^*\Diag(x_1,\dots,x_n,\lambda_{n+1},\dots,\lambda_N)U$$ with $U$ unitary and $\lambda_{n+1},\dots,\lambda_N$ belongs to the interior of $I^c$ then
	$$\phi(X)=U^*\Diag(x_1,\dots,x_n,\phi_{n+1},\dots,\phi_N)U$$
	with the $\phi_i$ in $]\alpha-\vep;\alpha+\vep[$ and 
	
	$$J_\phi \geqslant \left(\frac{\vep}{4e}\right)^{3N(N-n)}e^{-2N\Tr \ln (1+|X|1_{X\notin I})-2(N-n)\Tr \ln (1+|X|1_{X \in I})}$$
	
	As an immediate application with $I=\emptyset$ we see that there is an injective map $\phi:\HNC\to\{X\in\HNC :||X-\alpha||<\vep\}$ such that 
	$$J_\phi \geqslant \left(\frac{\vep}{4e}\right)^{3N^2}e^{-2N\Tr \ln (1+|X|)}$$
\end{lemma}
A key step of the proof is the next lemma.
\begin{lemma}\label{add-points}
	Let $I$ be a bounded close  interval and  $x_1,\dots,x_n$ in $\RR$  be fixed. Let $0<\eta<1$. Then,  there exists $y$ in $I$  at distance greater than $\eta\frac{|I|}{2n}$ from $\{x_1,\ldots,x_n\}$
	such that 
	$$\sum_{i=1}^n\ln\left(|x_i-y|-\eta\frac{|I|}{2n}\right)\geqslant n\ln\frac{|I|(1-\eta)}{2e}.$$
	More generally, for any integer number $N\ge n+1$, there exists $(x_{n+1},\dots,x_N)$ at distance greater than $\eta\frac{|I|}{2N}$ from $\{x_1,\ldots,x_n\}$
	such that for every $j\in [n+1,N]$,
	$$\sum_{i=1}^{j-1}\ln\left(|x_i-x_j|-\eta\frac{|I|}{2N}\right)\geqslant (j-1)\ln\frac{|I|(1-\eta)}{2e}.$$
\end{lemma}
\begin{proof}
	Let $\vep = \eta\frac{|I|}{2n}$ and  define
	$$V_{I,x,\vep} = \sup_{y\in I}\sum_{i=1}^n\ln(|x_i-y|-\vep)\,.$$
	Observe that by a simple dilatation of the space by a factor $\alpha>0$ we have
	$V_{\alpha I,\alpha x,\alpha\vep}=V_{I,x,\vep}+n\ln\alpha.$
	
	Now since $\vep<\frac{|I|}{2n}$, we can consider $Y$ to be the  random variable uniformly distributed on 
	$$E=I\backslash\bigcup_{i=1}^n]x_i-\vep;x_i+\vep[.$$
	and 
	$U=\sum_{i=1}^n\ln(|x_i-Y|-\vep)$.
	Since $\ln|x|$ is locally integrable, $U$ is $L^1$ and its expectation is bounded from below
	
	$$\E[U]= \frac{1}{|E|}\int_E\sum_{i=1}^n\ln( |x_i-y|-\vep)dy
	\geqslant \frac{n}{|E|}\int_{-1}^1 \ln( |y|)dy
	\geqslant \frac{-2n}{|I|-2n\vep}
	$$
	As a consequence
	$$V_{I,x,\vep}\ge \E[U]\ge \frac{-2n}{|I|-2n\vep}
	$$
	But due to the scaling argument we see that
	$$V_{I,x,\vep}\geqslant\sup_{\alpha>0}V_{\alpha^{-1} I,\alpha^{-1} x,\alpha^{-1}\vep}+n \ln\alpha\geqslant\sup_{\alpha>0}\{\frac{-2n\alpha}{|I|-2n\vep}+n\ln\alpha\}$$
	Taking the optimal value $\alpha =(1/2)( |I|-2n\vep)$ shows that
	$$V_{I,x,\vep} = \sup_{y\in I}\sum_{i=1}^n\ln(|x_i-y|-\vep)\geqslant n\ln\frac{|I|(1-\eta)}{2e}.$$
	Finally, since $y\mapsto \sum_{i=1}^n\ln(|x_i-y|-\vep)$ is continuous and concave on each of the interval when it is not equal to $-\infty$, it achieves its maximal value on the closed set $I$. We take $y$ to be such a maximizer.

	The generalization is straightforward if we assume that the logarithm in the right hand side is negative (because in that case we can add the points one by one using the above argument). To conclude we again use the scale invariance which allow us to go back to the case where the right hand side is negative.

\end{proof}

\begin{proof}[Proof of Lemma \ref{Jacobian2}.] In the following, $(x_1,\ldots,x_n)$ are fixed and we wish to map $(\lambda_{n+1},\cdots,\lambda_N)$ to some neighborhood of $\alpha$ at sufficient distance from the $x_i$'s to bound from below the associated Jacobian. We first choose a family of  points $\{x_{n+1},\ldots, x_{N}\}$ as in   Lemma \ref{add-points} with $I=]\alpha-\varepsilon;\alpha+\varepsilon[$ and fix them: they are some measurable functions of $\{x_1,\ldots,x_n\}$. We construct $\phi_{n+k}$ close to $x_{n+k}$  for $k\in [1,N-n]$ by putting
	$$\phi_{n+k} = x_{n+k} + \frac{\vep}{8N}\frac{\lambda_{n+k}}{1+|\lambda_{n+k}|}=x_{n+k}+\frac{\vep}{8N}\psi(\lambda_{n+k})$$
	with $\phi(x)=x(1+|x|)^{-1}$. Since $\phi$ is a diffeomorphism from $\RR$ into $[-1,1]$, $\psi: (x_1,\ldots,x_n,\lambda_{n+1},\cdots,\lambda_N)\mapsto (x_1,\ldots x_n,\phi_{n+1},\dots,\phi_N)$ is also a diffeomorphism. 
	Now we use Lemma \ref{Jacobian1} to compute the Jacobian $J_\phi$ of $\phi$ which is equal to 
	
	\begin{align*}
		&\left(\prod_{\substack{1\leqslant i\leqslant n\\ 1\le k\leqslant N-n}}\frac{|\phi_{n+k}-x_i|}{|\lambda_{n+k}-x_i|}\prod_{ 1\le k\le k'\le N-n}
		\frac{|\phi_{n+k}-\phi_{n+k'}|}{|\lambda_{n+k}-\lambda_{n+k'}|}\right)^2e^{\sum\Tr\ln \frac{\vep}{8N(1+|\lambda_{n+k}|)^2}}.
	\end{align*}
	To bound the denominator, we use the inequality $\ln|x-y|\leqslant \ln(1+|x-\gamma|)+\ln(1+|y-\gamma|)$ so that
	$$\left(\prod_{\substack{1\leqslant i\leqslant n \\ 1\le k\leqslant N-n}}|\lambda_{n+k}-x_i|\prod_{1\leqslant k\leqslant k'\leqslant N-n}|\lambda_{n+k}-\lambda_{n+k'}|\right)^2\\
	\leqslant e^{2(N-1)\Tr \ln (1+|X|1_{X\notin I}))}\,.
	$$
	To bound the numerator we use our construction of the $x_{n+k}$, since $|\phi_{n+k}-x_{n+k}|<\frac{\vep}{8N}$ and Lemma \ref{add-points} holds,  to find that for every $k\in [1,N-n]$:
	$$
	\prod_{i\leqslant n}|\phi_{n+k}-x_i|\prod_{k'<k}|\phi_{n+k}-\phi_{n+k'}|\geqslant
	\prod_{i\leqslant n+k-1}(|x_{n+k}-x_i| -\frac{\vep}{4N})\geqslant \left(\frac{\vep}{4e}\right)^{(n+k-1)}
	$$
	Putting everything together we get,
%	$$
%	J_\phi\geqslant \exp\left(-2N\Tr \ln (1+|X|1_{X\notin I})+\frac{1}{2}(N-n)(N+n+1)
%	\ln\frac{\vep}{4e}+(N-n) \ln \frac{N}{4 e}   \right)
%	$$
		$$
	J_\phi\geqslant e^{\left(-2N\Tr \ln (1+|X|1_{X\notin I})+\frac{1}{2}(N-n)(N+n+1)
	\ln\frac{\vep}{4e}+(N-n) \ln \frac{N}{4 e}   \right)}
	$$
	Finally the bound follows from $\frac{1}{2}(N+n-1)\ln\frac{\vep}{4e}+\ln \frac{N}{4 e}\leqslant 3N$.

\end{proof}

\subsection{Study of the energy functional related to the commutator model}

\begin{lemma}[{\bf Energy problem without a potential}]\label{Energywithout}
	Let $\beta$ be a positive  real number. Define for a measure $\mu$ on the real line  its energy:
	$$J_\beta^0(\mu)=-\int\int\ln\frac{\beta|x-y|^2}{1+\beta|x-y|^2}d\mu(x)d\mu(y)\,.$$
	\begin{enumerate}
		\item  $J_\beta^0$ is a lower continuous strictly convex function {on the set of probability measures  on $\mathbb R$ equipped with the weak topology.}
		\item{
		Let $I=[a,b]$, $a<b$,  be a closed non empty bounded interval.
		Then, there is a unique probability measure $\mu^I_\beta$  supported on  $I$ such that }
		$$J_\beta^0(\mu^I_\beta)=\inf\{J_\beta^0(\mu):\mu(I)=1\}=:c^{|I|}_\beta.$$
		{ $\mu^I_\beta$ has no atoms.}
		
		\item For every positive real numbers $\alpha, a$, we have $c^a_{\alpha^2\beta}=c^{\alpha a }_{\beta}$ so that $c^{\alpha}_{\beta}=c^{1}_{\alpha^{2}\beta}$. Moreover,  for every real number $b$ and any measurable subset  $B$ of $I$ {
		$$\mu_{\alpha^2\beta}^I(B)=\mu_{\beta}^{\alpha {I+b}}( \alpha(B+b))\,.$$}
		
		\item The application 
		$\beta\to c^1_\beta$ is a strictly positive decreasing continuous function on $(0,\infty)$ and
		$$2\pi-O(\beta^{-\frac{1}{4}})\leqslant \sqrt{\beta}c^1_\beta\leqslant 2\pi$$
		In particular,  for  any  positive real number $a$,{
		$$c^a_\beta\sim_{\beta\to+\infty}\frac{2\pi a}{\sqrt{\beta}}.$$}

	\end{enumerate}
	
\end{lemma}

\begin{proof}
	The first point was proved in the proof of Proposition \ref{prop1}  and the second point is clear as a simple change of variable $h_{\alpha,b}(x)=\alpha x+b$ shows that
		$J_{\beta}^{0}(h_{\alpha,b}\#\mu)=J^{0}_{\beta\alpha^{2}}(\mu)$ if $h\#\mu$ denotes the push-forward of $\mu$ by $h$. This  insures by uniqueness of the minimizers that
		$h_{\alpha,b}\#\mu_{\alpha^2\beta}^I=\mu_{\beta}^{\alpha {I+b}}  $ and $c^a_{\alpha^2\beta}=c^{\alpha a }_{\beta}$. Hereafter we denote by
		$$\phi_\beta(t):=- \ln \frac{\beta|t|^2}{1+\beta|t|^2}=\ln (1+\frac{1}{\beta|t|^2}) $$ and observe that $\phi_\beta\ge 0$ is continuous except at $0$ where it blows up. {Moreover $\beta\rightarrow \phi_{\beta}(t)$ is decreasing for every real number $t$. }
	As a consequence,
	$$\beta\to c^1_\beta= \inf_{\mu\in \mathcal P([0,1])}J^{0}_{\beta}(\mu)=\inf_{\mu\in \mathcal P([0,1])}\int\int \phi_{\beta}(x-y) d\mu(x)d\mu(y)
	$$
	 is a  decreasing  function from $(0,\infty)$ into itself. 
	  $c^{1}_{.}$ is moreover continuous. Indeed, it is 
	 is lower semi-continuous since it is decreasing. The 
  upper semi-continuity can be checked by noticing that { for every $\varepsilon\in [0,\beta)$
  \begin{eqnarray*}
  J_\beta^0(\mu)&=&J^{0}_{\beta-\varepsilon}(\mu)+\int \ln\left( \frac{(1+\beta(x-y)^{2})(\beta-\varepsilon)}{(1+(\beta-\varepsilon)(x-y)^{2})\beta}\right)d\mu(x)d\mu(y)\\
  &\ge& J_{\beta-\varepsilon}(\mu)+\ln(1-\frac{\varepsilon}{\beta})\end{eqnarray*}
  which implies
	$$c^1_{\beta}=  J_\beta^0(\mu^{[0;1]}_\beta) \ge J^{0}_{\beta-\varepsilon}(\mu^{[0;1]}_\beta)+\ln(1-\frac{\varepsilon}{\beta})
	\geqslant c^{1}_{\beta-\varepsilon}+O(\frac{\varepsilon}{\beta})\,.$$}
	We next show that the function $\psi(\beta):= (c^1_{\beta^{2}})^{-1}$ is subadditive.  To this end,
	we cut the interval $[0,1]$ in two subintervals $[0;u]$ and $[u;1]$. Let us denote $\gamma:=\mu^{[0;1]}_\beta([0;u])$, since $\mu^{[0;1]}_\beta$ has no atoms and we integrate a non-negative function, 
	\begin{align}
		c^1_\beta&= J_\beta^0(\mu^{[0;1]}_\beta)\geqslant J_\beta^0(\mu^{[0;1]}_\beta|_{[0;u]})+J_\beta^0(\mu^{[0;1]}_\beta|_{[u;1]})\nonumber\\
		& \geqslant \gamma^2c^u_\beta+(1-\gamma)^2c^{1-u}_\beta\label{b1}\\
		&\geqslant\inf_{0\le \delta\le 1}\{ \delta^2c^1_{u^{2}\beta}+(1-\delta)^2c^1_{(1-u)^{-}\beta}\}=((c^1_{u^{2}\beta})^{-1}+(c^1_{(1-u)^{2}\beta})^{-1})^{-1}\,.\nonumber
	\end{align}
	Therefore $\psi$ is subbaditive since for every $u\in [0,1]$
	$$\psi(\beta)\le \psi(\beta u)+\psi(\beta(1-u))\,.$$
	 As a consequence there exists a constant $c$ such that 
	 $c^1_{\beta^2}\sim c/\beta$,  as $\beta$ goes to infinity.
	The aim of the rest of the proof is to show that $c=2\pi$, which  requires extra work. As in \eqref{b1}, we find
	$$c^1_\beta \geqslant J_\beta^0(\mu^{[0;1]}_\beta|_I)\geqslant \mu^{[0;1]}_\beta(I)^2c^{|I|}_\beta\geqslant \mu^{[0;1]}_\beta(I)^2c^1_{|I|^2\beta}\,.$$
	{
	Therefore, we deduce that if $|I_\beta|^2\beta\to\infty$ then
	\begin{equation}\label{holder}
	\limsup_\beta \frac{ \mu^{[0;1]}_\beta(I_\beta)}{\sqrt{|I_\beta|}}\leqslant \limsup_{\beta}\sqrt{\frac{c^1_\beta}{{|I_\beta|}c^1_{|I_\beta|^2\beta}}}=1\end{equation}
	%(here we used the fact that $\beta\to c^1_\beta=c^{\sqrt{\beta}}_1$ is continuous)
	As a  consequence, for every measurable  decreasing function $f:[0;1]\to\RR^+$, we have:
	\begin{equation}\label{sqrtequ}\limsup_\beta\int_0^1 fd \mu^{[0;1]}_\beta \le 2
	  \int_0^1 f(x) \frac{1}{\sqrt{x}}dx\end{equation}
	Indeed it is trivially true for $f(x)=1_{x<t}$ and then one can extend the result to all decreasing $f$ by linearity and monotone limit. 
	}
	
	Now let us prove that the support of $\mu^{[0;1]}_\beta$ is  the whole interval $[0;1]$.
	Since $J_\beta^0$ is left constant under translation of the measure and decreases by  dilatation we already know that the convex closure of the support of $\mu^{[0;1]}_\beta$ is ${[0;1]}$.	We therefore only need to make sure that there is no hole in the support.
We know that there exists a constant $C_\beta$ such that $\mu^{[0;1]}_\beta$ p.p. :
	$$U(x)=2\int\ln \frac{1+\beta|x-y|^2}{\beta|x-y|^2}d\mu^{[0;1]}_\beta(y)= 2\int \phi_\beta(x-y) d\mu^{[0;1]}_\beta(y) 
	=C_\beta$$
	and everywhere on ${[0;1]}$, $U(x)\geqslant C_\beta$. Indeed, this can be deduced from the fact that $J_\beta^0(\mu^{[0;1]}_\beta+\delta \nu)\ge J_\beta^0(\mu^{[0;1]}_\beta)$ for 
	any measure $\nu$ on $[0,1]$ such that $\mu^{[0;1]}_\beta+\delta \nu$ is a probability measure.
But $U(x)$ is strictly convex on the complementary of the support of $\mu^{[0;1]}_\beta$ (since $\phi_\beta$ is strictly convex)
and therefore the support  of $\mu^{[0;1]}_\beta$ is connected and therefore equals $[0,1]$. So  $U$ is constant equal to $C_\beta$ on ${[0;1]}$. Therefore its integral with respect to $d\mu^{[0;1]}_\beta$ is the same than its integral with respect to the uniform measure on $I$, from there we get the new formula for the energy:
	$$J_\beta^0(\mu^{[0;1]}_\beta)=\int\int_0^1\ln \frac{1+\beta|x-y|^2}{\beta|x-y|^2}dyd\mu^{[0;1]}_\beta(x)= \int(\int_0^1  \phi_\beta(x-y)dy)d\mu^{[0;1]}_\beta(x)$$
	Now for any $0<x<1$, because $\phi_\beta$ is non negative, {
	\begin{align}\sqrt{\beta}\int_0^1  \phi_\beta(x-y)dy&\leqslant \sqrt{\beta}\int_{\RR}\phi_\beta(z)dz
		=\int_{\RR}\phi_1(z)dz\nonumber\\
		&=[z\ln\frac{z^2}{1+z^2}+2\arctan(z)]_{-\infty}^\infty=2\pi\label{b2}
	\end{align}}
	Therefore $c^1_\beta\leqslant\frac{2\pi}{\sqrt{\beta}}$ and we have the upper bound in the fourth point of the proposition.
	We finally prove the converse inequality and notice that for all real number $x$, 
	
	\begin{align*}0&\leqslant 2\pi+\sqrt{\beta}\int_0^1 \ln \frac{\beta|x-y|^2}{1+\beta|x-y|^2}dy= \sqrt{\beta}\int_{[0,1]^{c}} \phi_{\beta}(x-y)dy\\
	&\le \int_{-\infty}^{-\sqrt{\beta}x}\phi_{1}(y) dy+\int_{-\sqrt{\beta} x+\sqrt{\beta}}^{+\infty}\phi_{1}(y) dy\\
		&
		\leqslant 2\pi (1\wedge (\frac{1}{\sqrt{\beta}x}+ \frac{1}{\sqrt{\beta}(1-x)}))
	\end{align*}
	where we used that the left hand side is bounded by $2\pi$ and $\phi_{1}(x)\le x^{-2}$.
	From this we deduce, using \eqref{sqrtequ} that
	\begin{align*}
	\limsup_\beta \beta^{\frac{1}{4}}(2\pi- \sqrt{\beta}c^1_\beta)&\leqslant \limsup_\beta \beta^{\frac{1}{4}} 4\pi \int_0^{\frac{1}{2}} \left(1\wedge \frac{1}{\sqrt{\beta}x}\right)d\mu^{[0;1]}_\beta\\
	&\leqslant\limsup_\beta \beta^{\frac{1}{4}}  8\pi \int_0^{\frac{1}{2}} \left(1\wedge \frac{1}{\sqrt{\beta}x}\right)\frac{1}{\sqrt{x}}dx = 32\pi\,.
	\end{align*}
	%which yields the lower bound in the fourth point of the proposition. 
	
\end{proof}
We next consider the case where the energy functional to minimize contains a non trivial potential.  \begin{lemma}[{\bf Energy problem with a potential}]\label{energyproblempotential} Let
	$V:\RR\to\RR$ be a continuous functions such that
	$\lim_{|x|\to\infty}V(x)=+\infty$ and $V$ achieves its minimum value  (equal to zero) exactly on the finite set $x_1,\dots,x_r$ with $V(x)=c_i(x-x_i)^{2p_i}(1+O(|x-x_{i}|))$ near $x_i$ for each $i\in\{1,\ldots r\}$ for some positive integers $p_i$. We set  $p=\max p_i$  and for a measure $\mu$, we recall that
	\begin{align*}
		J^V_\beta(\mu)&
		&=\int \left( \frac{1}{2}( V(x)+V(y))-\ln\frac{\beta|x-y|^2}{1+\beta|x-y|^2}\right)d\mu(x)d\mu(y)\,.
	\end{align*}
	
	{We define $\beta_p=\beta^{\frac{p}{2p+1}}$  and  $\gamma_p = \beta^{-\frac{1}{2(2p+1)}}$.}
Then:
	\begin{enumerate}
		\item  $J^V_\beta$ is a lower continuous stricly convex function on the set $\mathcal P(\mathbb R)$ of probability measures on the real line. It achieves its minimum value 
		at a unique probability measure $\mu^V_\beta$ and there exists a constant $0<c_V<\infty$ such that {
		$$c^V_\beta {:=\inf J^{V}_{\beta}}= \beta_{p}^{-1} c_V+o(\beta_{p}^{-1})$$
		where $o(\beta_{p}^{-1})$ goes to zero as $\beta$ goes to infinity.}
		The constant $c_{V}$ satisfies the scaling property $c_{\alpha V} = \alpha^{\frac{1}{2p+1}}c_V$  for every $\alpha>0$. Besides we have an explicit expression for $c_V$:
		$$c_V=\left(1+\frac{1}{2p}\right)^{\frac{4p+1}{2p+1}}\pi^{1-\frac{1}{2p+1}}(\sum_{j:p_j=p} c_j^{-1/(2p)})^{\frac{1}{2p+1}-1}$$
		\item
		There exists a constant $L_V<\infty$ such that for every $\beta>1$,
		$\mu^{V}_\beta$ is compactly supported in
		$$S(L_{V}):={\bigcup}_{i=1}^r \{ x_i+\beta^{\frac{-p}{(2p_i)(2p+1)}}[-L_V;L_V]\}$$
	{	Moreover, for every interval  $I_\beta$ of length $|I_\beta|$ such that $|I_\beta|^2\beta\to\infty$ then
	$$	
\mu^V_\beta(I_\beta)=O(\beta^{\frac{1}{4(2p+1)}})\sqrt{|I_\beta|}.$$ }
		\item
		There exists a constant $0<C<\infty$ such that for every $i$:
		$$\lim_{\beta\rightarrow \infty} \mu^V_\beta(x_i+\beta^{\frac{-p}{(2p_i)(2p+1)}}[-L_V;L_V])=1_{p_i=p}C^{-1}c_i^{-\frac{1}{2p}}$$
		
		\item {Recall that $h_{\alpha}(x)=\alpha x$.
		If $V$ has a single  minimum at  $x_1=0$ and $V(x)\simeq c x^{2p}$ in a neighborhood of the origin,  then
		$$h_{\beta^{-\frac{1}{4p+2}}}\#  \mu^V_\beta|_{\{\beta^{\frac{-1}{2(2p+1)\}}}[-L_V;L_V]\}}\to d\mu(z):= \frac{1}{2\pi}(A-{c}\color{black} z^{2p})_+dz$$
		where $A$ is the only constant such that the RHS is a probability measure. More precisely, the Wasserstein distance $W_1$ between these two measures is 
		 bounded by 
		$\beta_p^{-\frac{1}{11}}$.}
		
	\end{enumerate}
	
\end{lemma}
In other words,   the energy is of order  $\beta_{p}^{-1}$  and   $\gamma_p$ is the typical length of the support of the minimizer around the minima.\color{black}

\begin{remark}
	
	As a special case of interest we get that if all global minima of $V$ are quadratic ($p=p_i=1$) then the minimizer has its support with endpoints at distance of order $\beta^{-\frac{1}{6}}$ of the minima of $V$ and $ \inf J^{V}_\beta $ is of order $\beta^{-\frac{1}{3}}$.
\end{remark}
\begin{proof}
	
	The strict convexity of $J_\beta^V$ is a direct consequence of the strict convexity of $J_\beta^{0}$ as they differ only by a linear term. 
	We next estimate $c^V_\beta$.
	We will assume without loss of generality that $0$ minimizes $V$ and that near $0$, $V(x)=cx^{2p}(1+o(1))$ with $c>0$. Recall that $\mu_\beta^I$ denotes the minimizer of $J_\beta$ on $\mathcal P(I)$.  We have:
	$$c^V_\beta\leqslant \inf_{\gamma>0}J^V_\beta(\mu^{[-\gamma;\gamma]}_\beta)\leqslant \inf_\gamma \{c\gamma^{2p}(1+o(\gamma))+c^{2\gamma}_\beta\}= \inf_\gamma\{ c\gamma^{2p}(1+o(\gamma))+c^{1}_{4\gamma^2\beta}\}$$
	where we used Lemma \ref{Energywithout}. Recalling that $c^{1}_{4\gamma^2\beta}\simeq \pi/ \gamma\sqrt{\beta}$, 
 and taking $\gamma = \gamma_p$, we get
	\begin{equation}\label{cub}c^V_\beta\le O(\beta_p^{-1})\,.\end{equation}
	where $\beta_p,\gamma_p$ are defined in the statement of the Theorem.
	
	Let us now prove that the support of a minimizer of $J^V_\beta$ is concentrated in a neighborhood of  the minima of $V$. Let $\mu$ be a measure such that there exists  $z$ in the support of $\mu$ such that $V(z)>2 L \beta_p^{-1}$ for some constant $L$ to be chosen later.
	Choose $\varepsilon$ such that $V>L \beta_p^{-1}$ on $I_z^\varepsilon:=[z- \varepsilon, z+\varepsilon]$ and define $\mu_z =\mu(I_z^\varepsilon)^{-1} \mu|_{I_z^\varepsilon}$.  Now we can define:
	$$\mu_t = \mu+t(\lambda_{\gamma_p}-\mu_z)$$
	with $\lambda_{\gamma_p}$ the uniform measure on $[- \gamma_p,\gamma_p]$.
	$\mu_{t}$ is a probability measure for $t\in [0,\mu(I_z^\varepsilon)]$
	 and $\mu_0=\mu$.
	Besides
	
	\begin{align*}
		\partial_t J^V_\beta(\mu_t) |_{t=0}&=\int V d(\lambda_{\gamma_p}-\mu_z) +2 \int\phi_\beta(x-y)
		d\mu(x)d(\lambda_{\gamma_p}-\mu_z) (y)\\
		&\leqslant \sup_{[-\gamma_p, \gamma_p]}V - L \beta_p^{-1}+2 \sup_u  \int \phi_\beta(u-y)d\lambda_{\gamma_p}(y)\\
	\end{align*}
	where we again used that $\phi_\beta$ is non-negative.
	It is easy to see that the supremum in the last term is achieved at $u=0$ and then by similar computations as in \eqref{b2} to 
	 see that this term is  at most of order 
	$O( \beta_p^{-1})$. Since the first term as the same order of magnitude, we conclude that there exists  a finite  constant $L_{V}$ such that if $L\ge L_{V}$ the above right hand side is non-negative.Then,  $\mu$ can not be the minimizer of $J_\beta^V$.
	Hence, the support of  $\mu^V_\beta$ is included in 
	$S(L_{V}):={\bigcup}_{i=1}^r \{ x_i+\beta^{\frac{-p}{(2p_i)(2p+1)}}[-L_V;L_V]\}$.
	For $\beta$ sufficiently large, the restriction of $V$ to  each of the interval  $x_i+\beta^{\frac{-p}{(2p_i)(2p+1)}}[-L_V;L_V]$ is convex (and strictly convex except at the point $\{x_{i}\}$)  and an argument identical to the one for $\mu^{[0;1]}_\beta$ in the proof of the previous Lemma \ref{Energywithout}  shows  that the restriction of the support of $\mu^V_\beta$ in $x_i+\beta^{\frac{-p}{(2p_i)(2p+1)}}[-L_V;L_V]$  is an interval. Let us call $\gamma^i_\beta$ the length of this interval and  define $\alpha^i_\beta= \mu^V_\beta(x_i+\beta^{\frac{-p}{(2p_i)(2p+1)}}[-L_V;L_V])$.
	Then, we get the complementary lower bound to \eqref{cub} by using again that $\phi_{\beta}$ and $V$ are non-negative:
		\begin{eqnarray}\label{eqc}
	c^V_\beta&\geqslant& \sum_i  \inf_{\mu\in\mathcal P(x_i+\beta^{\frac{-p}{(2p_i)(2p+1)}}[-L_V;L_V]))} J^{0}_{\beta}(\alpha_{\beta}^{i}\mu)
	\nonumber\\
	&=&(1+o(1))\sum_i(\alpha_{\beta}^{i})^{2} \frac{2\pi}{L_{V}\sqrt{\beta}\beta^{-\frac{p}{2p_{i}(2p+1)}}}
	\end{eqnarray}
	where we finally used Lemma \ref{energyproblempotential}.  We notice that for $i$ such that $p_{i}=p$
	{$$\sqrt{\beta}\beta^{-\frac{p}{2p_{i}(2p+1)}}=\beta^{\frac{p}{2p+1}}=\beta_{p}$$ whereas for $i$ such  that $p_{i}<p$,
	{$\sqrt{\beta}\beta^{-\frac{p}{2p_{i}(2p+1)}}\ll \beta_{p}$.}  Since we have seen  in \eqref{cub} that $c^{V}_{\beta}$ is at most of order $\beta_{p}^{-1}$ we conclude that  for $i$ such that $p_{i}<p$, $\alpha_{\beta}^{i}$ goes to zero with $\beta$ (is in fact bounded by $\beta^{(\frac{p}{p_{i}}-1)\frac{p}{4(2p+1)}}$)
	and also that $0<\liminf \beta_{p}c^{V}_{\beta} \leqslant \limsup \beta_{p}c^{V}_{\beta}<\infty$, i.e.
	\begin{equation}\label{cublb}
	c^{V}_{\beta} \asymp \beta_{p}^{-1}.\end{equation}
	{
	This allows us to complete the proof of  Lemma  \ref{energyproblempotential} (2) and show that
$\mu^{V}_{\beta}$  can not concentrate too much on small intervals. Indeed, let $I_\beta$ be a sequence of intervals of Lebesgue measure $|I_\beta|$ such that $|I_\beta|^2\beta\to\infty$, then, since $V$ and $\phi_{\beta}$ are non-negative 
$$c^V_\beta\geqslant J^0_\beta(\mu^V_\beta|_{I_\beta})=
\mu^V_\beta(I_\beta)^2c_\beta^{|I_\beta|}$$
and therefore by Lemma
\ref{Energywithout}(4) we get 
$$\mu^V_\beta(I_\beta)\leqslant \sqrt{\frac{c^V_\beta}{c_{|I_\beta|^2\beta}^{1}}}=\sqrt{O(\beta^{-\frac{p}{2p+1}}\sqrt{\beta}|I_\beta|)}=O(\beta^{\frac{1}{4(2p+1)}})\sqrt{|I_\beta|}.$$

	We next  complete the proof of the first point of Lemma  \ref{energyproblempotential}, as well as its third and fourth points. First notice that we can remove the interaction between wells since they are at distance of order one, 
	up to an error of order $1/\beta$:
	
	$$J^V_\beta(\mu_\beta)=\sum_i J_\beta^V(\alpha^\beta_i\mu_i^\beta)+O(1/\beta)$$
where 
	$\mu_i^\beta$ is   the restriction of $\mu^{V}_{\beta}$ to  the neighborhood $x_i+\beta^{-\frac{p}{2p_{i}(2p+1)}}[-L_V;L_V]$ of the minimizer $\{x_{i}\}$. 	We estimate now the energy  in each of these wells. 
		}

	{\bf The energy problem  with one well}
		
	Assume that $V(x)=c x^{2p}$.
	With the previous discussion in mind we look at the rescaled energy functional for a probability measure $\mu$  on $[-L_V;L_V]$:
	
	%{\color{blue}
	\begin{eqnarray*}
	\tilde J_\beta^{V}(\mu)&:=&\beta_pJ^{V}_\beta(h_{ \beta^{-\frac{1}{2(2p+1)}}}\#\mu)\\
	&=&\int c x^{2p}d\mu(x)+\beta_p\int\int\ln\phi_{\beta_{p}^2}(x-y)d\mu(x)d\mu(y).\end{eqnarray*}
	%}
		We are going to show that the infimum of $\tilde J_\beta^{V}$ over the set of probability measures on $[-L_V;L_V]$ is converging to a  constant depending only on $p$ and $c$. More precisely,  with $a_p=\left(1+\frac{1}{2p}\right)^{\frac{4p+1}{2p+1}}\pi^{1-\frac{1}{2p+1}}$, we show:

\begin{equation}\label{toprove} \inf \tilde J^{cx^{2p}}_\beta =a_pc^{\frac{1}{2p+1}}+O(\beta_p^{-\frac{1}{11}}). \end{equation}

		To this end, let $m$  be a large integer  (to be chosen later) and cover $[-L_V;L_V]$ by the union of segments $[z_i;z_{i+1}]$ of width $m^{-1}$ with $z_i=-L_V+ (i+1)/m$, $-1\leqslant i\leqslant \lfloor 2L_V\rfloor $.

	We first lower bound the infimum of $ \tilde J_\beta^{V}$. Again  this infimum is achieved at a unique probability measure that we denote  $\nu^V_\beta$. 
	We let  $\alpha^m_i(\beta)=\nu^V_\beta([z_i;z_{i+1}])$ and notice that, because $\phi_{\beta_{p}^2}$ is non-negative,
	\begin{align}
		\tilde J^V_\beta(\nu^V_\beta)&\geqslant \sum_{i=-1}^m \tilde J^V_\beta(\nu^V_\beta|_{[z_i;z_{i+1}]})\nonumber\\
		&
		\geqslant \sum_{i=-1}^{ \lfloor 2L_V\rfloor}\left\lbrace  \alpha^m_i(\beta) c\inf_{[z_i;z_{i+1}]} z^{2p} +(\alpha^m_i(\beta))^2 \tilde J^0_\beta((\alpha^m_i(\beta))^{-1}\mu|_{[z_i;z_{i+1}]})\right\rbrace\nonumber\\
		& \geqslant (1-O(\frac{1}{m}))  \sum_{i=-1}^{ \lfloor 2L_V\rfloor} \left\lbrace \alpha^m_i(\beta)c z_i^{2p}+(\alpha^m_i(\beta))^2\beta_p c^{1/m}_{\beta_p^2}\right\rbrace\nonumber \\
		&=(1-O(\frac{1}{m}))\sum_{i=-1}^{ \lfloor 2L_V\rfloor} \left\lbrace \alpha^m_i(\beta)c z_i^{2p}+(\alpha^m_i(\beta))^2 m2\pi (1-O((\frac{\beta_p}{m})^{-1/2}))\right\rbrace\nonumber \\
		&=(1-O(\frac{1}{m})-O((\frac{\beta_p}{m})^{-1/2}))
		 \phi_m(\alpha^m(\beta))\label{lb1}
	\end{align}
	where 
	$$\phi_m(\alpha)=\sum_{i=-1}^{ \lfloor 2L_V\rfloor} \{\alpha_i cz_i^{2p}+2\pi m(\alpha_i)^2 \}$$
	and we assumed that $\beta_{p}/m$ is large to use the fourth point of Lemma \ref{Energywithout}.	Now let $\bar\alpha^m=(\bar\alpha^m_i)_{i=1,..r}$ be the value which achieves the minimum of the strictly convex functional $\phi_m$  on the convex set $\{\alpha_{i}\in [0,1]: \sum\alpha_{i}=1\}$ and $\mu^{[z_i;z_{i+1}]}_{ \beta^2_p}$ the probability on $[z_{i};z_{i+1}]$ minimizing  { $\tilde J_{\beta}^{0}$. }From the above we deduce that
	\begin{eqnarray}
	 \inf \tilde J^{cx^{2p}}_\beta &\ge &(1-O(\frac{1}{m})-O((\frac{\beta_p}{m})^{-1/2}))
		\inf  \phi_m(\alpha^m(\beta))\nonumber\\
		&=& (1-O(\frac{1}{m})-O((\frac{\beta_p}{m})^{-1/2})) \tilde J^{cx^{2p}}_\beta(\mu^{m})\label{lbv}
		\end{eqnarray}}
		where we 
define $\mu^{m}$ to be  the probability measure
	$$\mu^m = :\sum_i\bar \alpha^m_i\mu^{[z_i;z_{i+1}]}_{ \beta^2_p}.$$
	We next prove the complimentary upper bound. In fact:
	\begin{align*}
		 \tilde J^V_\beta(\nu^V_\beta)
		&\leqslant \tilde J^V_\beta(\mu^m)\\
		&\leqslant \sum_{i=1}^m \bar\alpha^m_i c\sup_{[z_i;z_{i+1}]} z^{2p}+(\bar\alpha^m_i)^2 \tilde J^0_\beta(\mu^{[z_i;z_{i+1}]}_{ \beta^2_p})\\
		&-\sum_{i\neq j}\bar\alpha^m_i\bar\alpha^m_j\beta_p\int\int\ln\frac{\beta_p^2|x-y|^2}{1+\beta_p^2|x-y|^2}d\mu^{[z_i;z_{i+1}]}_{ \beta^2_p}(x)d\mu^{[z_j;z_{j+1}]}_{ \beta^2_p}(y)\\
	\end{align*}
	
	Our goal is to get rid of the last term. Define the function : $$g^i_\beta(y)=-\beta_p\int\ln\frac{\beta_p^2|x-y|^2}{1+\beta_p^2|x-y|^2}d\mu^{[z_i;z_{i+1}]}_{ \beta^2_p}(x).$$ 
	The function $g^i_\beta$ is positive and decreases when $y$ goes away from $[z_i;z_{i+1}]$. Moreover,  for  $y\in [z_i;z_{i+1}]$, $g^i_\beta(y)$ is constant and its value is 
	{
	$$ g^{i}_{\beta}(y)=a_m^p(\beta):=\beta_pJ^{0}_{\beta_p^2}(\mu^{[z_i;z_{i+1}]}_{ \beta^2_p})=\beta_{p}c^{\frac{1}{m}}_{\beta_{p}^{2}}=
2m\pi (1+O(\frac{\beta_{p}}{m})^{-1/2})\,,$$}
where we again assumed that $\beta_{p}/m$ is large.
	This is sufficient to get a bound on our sum. Indeed, for $j> i$, we get for every $\eta\in [0,\frac{2}{m}]$,
	
	{
	\begin{align*}
		&\int g^i_\beta(y) d\mu^{[z_j;z_{j+1}]}_{ \beta^2_p}(y)\leqslant \int g^i_\beta(y) d\mu^{[z_{i+1};z_{i+2}]}_{ \beta^2_p}(y)\\
		&\leqslant \int_{z_{i+1}}^{z_{i+1}+\eta} g^i_\beta(z_{i+1})  d\mu^{[z_{i+1};z_{i+2}]}_{ \beta^2_p}(y)+g^i_\beta(z_{i+1}+\eta)\\
		&\leqslant \mu^{[0;1/m]}_{ \beta^2_p}([0;\eta])g^0_\beta(z_{1}) -\beta_p\ln\frac{\beta_p^2\eta^2}{1+\beta_p^2\eta^2}.
	\end{align*}	
	Now according to \eqref{holder}, $ \mu^{[0;1/m]}_{ \beta^2_p}([0;\eta]) =  \mu^{[0;1]}_{ \beta^2_p/m^2}([0;m\eta])=O((m\eta)^{1/2})$ as long as $\eta\beta_p\to\infty$.}
	 Therefore, we deduce that if $\eta\beta_p\to\infty$,
	$$\sum_{i\neq j}\bar\alpha^m_i\bar\alpha^m_j
	\int g^{i}_{\beta}(y)
	d\mu^{[z_j;z_{j+1}]}_{ \beta^2_p}(y)= O( m^{3/2}\eta^{1/2}(1+(\frac{\beta_{p}}{m})^{1/4}))+\beta_{p}^{-1}\eta^{-2})$$
	
	We finally choose $m,\eta$ so that $\eta = 	 \beta_p^{-\frac{1}{10}}, m=\lfloor\beta_p^{\frac{1}{11}}\rfloor$.
With this choice, together with \eqref{lb1}, we have proved 
\begin{equation}\label{bor1}\phi_m( \alpha^m(\beta))+O(\beta_p^{-\frac{1}{11}})\leqslant \tilde J^V_\beta(\nu^V_\beta)\leqslant \phi_m( \bar\alpha^m)+O(\beta_p^{-\frac{1}{11}})\end{equation}
Because $\phi_m( \bar\alpha^m)\le \phi_m( \alpha^m(\beta))$ we conclude that 
	 \begin{equation}\label{conv1}\lim_{\beta\rightarrow\infty}\inf \tilde J^V_\beta=\lim_{m\rightarrow\infty}\inf \phi_{m}\end{equation}
	 if these limits exists.
	 We finally show that  the above limits are positive and finite, and moreover that $\alpha^{m}_{\beta}-
	 \bar\alpha^m$ goes to zero.
	 Looking at the critical points  $\bar\alpha^{m}$ of  $\phi_m$ on $\{\alpha_{i}\in [0,1]:\sum_{i}\alpha_{i}=1\}$ we see that 
	 for the indices in 
	  $S_m:=\{i : \bar\alpha^m_i>0\}$ we must have
	$$cz_i^{2p}+2\pi m\bar \alpha^m_i=A_m$$
	where $A_m$ is a constant.  Hence, the minimizer of $\phi_{m}$ is given by
	$$\bar\alpha^m_i = \frac{1}{2\pi m}(A_m-cz_i^{2p})_+$$
	where $A_m$ is the only positive solution to $$\sum_i \frac{1}{2\pi m}(A_m-cz_i^{2p})_+=1\,.$$

	The continuum limit of this minimizer
	is given by  the probability measure 
	\begin{equation}\label{eqmeas} \nu_{c}(dx)= \frac{1}{2\pi }(A-cx^{2p})_+ dx\end{equation}
	where $A$ is the only positive solution to $\int \frac{1}{2\pi }(A-cz^{2p})_+dz=1$.
	This can be solved directly and gives :
	$$A=c^{\frac{1}{2p+1}}\left(\frac{(2p+1)\pi}{2p}\right)^{\frac{2p}{2p+1}}\,.$$
	Plugging this formula
	in $\phi_{m}$ we see that
	\begin{eqnarray}
	\label{infphi}
	\inf \phi_{m}&=&\sum_{i=1}^m \{\frac{1}{2\pi m}(A_m-cz_i^{2p})_+ cz_i^{2p}+\frac{1}{4 \pi m}(A_m-cz_i^{2p})_+^{2}\}\nonumber\\
	&\rightarrow& \int c x^{2p}\nu_{c}(dx)+ \pi \int \frac{d\nu_{c}(x)}{dx} d\nu_{c}(x)=:F(c x^{2p})\end{eqnarray}
	A direct computation proves that
	$$F(c x^{2p}) = \left(1+\frac{1}{2p}\right)^{\frac{4p+1}{2p+1}}\pi^{1-\frac{1}{2p+1}}c^{\frac{1}{2p+1}}.$$
	This completes the proof of \eqref{toprove} . 
{ We can also see that  $A_m=A+O(1/m)$ and  $W_1(\nu,\sum_i \bar\alpha^m_i\delta_{z_i})=O(1/m)$.

	We next can quantify this convergence as follows.
	Indeed, for any sequence $\alpha^m_i$ on the simplex:
	\begin{align*}
		&\phi_m(\alpha^m)-\phi_m(\bar\alpha^m)\\&=\sum_i(\alpha^m_i -\bar\alpha_i^m)(cz_i^{2p}+2\pi m \bar\alpha^m_i)+\sum_i\pi m (\alpha^m_i-\bar\alpha_i^m)^2\\
		&=\sum_{i}(\alpha^m_i -\bar\alpha_i^m)A_m+\sum_{i\notin S_m}\alpha^m_i(cz_i^{2p}-A_m)+\sum_i\pi m (\alpha^m_i-\alpha_i^m)^2\\
		&\geqslant \sum_i\pi m (\alpha^m_i-\alpha_i^m)^2=\pi m \|\alpha^m-\bar\alpha^m\|^2_2
	\end{align*}
	Therefore \eqref{bor1} implies that
	
	$$\pi m \|\alpha^m(\beta)-\bar\alpha^m\|^2_2\leqslant O(\beta_p^{-\frac{1}{11}})\,.$$
	
	This allows us to bound the Wasserstein distance between the  corresponding measures:
	
		\begin{align*}W_1(\sum_i \bar\alpha^m_i\delta_{z_i},\sum_i\alpha^m_i(\beta) \delta_{z_i})&\leqslant
	2d_{TV}(\sum_i \bar\alpha^m_i\delta_{z_i},\sum_i\alpha^m_i(\beta) \delta_{z_i})\\
	&=
	2\|\alpha^m(\beta)-\bar\alpha^m\|_1\leqslant 2\|\alpha^m(\beta)-\bar\alpha^m\|_2\\
	&\leqslant 2\sqrt{O(m^{-1}\beta_p^{-\frac{1}{11}})} = O(\beta_{p}^{-\frac{1}{11}}).
	\end{align*}
		Finally : 
\begin{align*}	W_1(\nu_{c},\nu^V_\beta)
		&\leqslant W_1(\nu_c,\sum_i \bar\alpha^m_i\delta_{z_i}) +W_1(\sum_i \bar\alpha^m_i\delta_{z_i},\sum_i\alpha^m_i(\beta) \delta_{z_i})\\ &+W_1(\sum_i\alpha^m_i(\beta) \delta_{z_i},\nu^V_\beta) 	
\end{align*}	
and 
the first and last term are bounded by  $ O(m^{-1})=O(\beta_p^{-\frac{1}{11}})$
So we conclude that $W_1(\nu_c,\nu^V_\beta)=O(\beta_p^{-\frac{1}{11}})$.
}

\noindent
{\bf{End of the proof  of the Lemma.}}
We have already proved that:
$$\inf J^V_\beta = \inf \{J_\beta(\mu) : \mu(	{\bigcup}_{i=1}^r \{ x_i+\beta^{\frac{-p}{(2p_i)(2p+1)}}[-L_V;L_V]\}) = 1\}.$$

But on $\{ \mu : \mu(	{\bigcup}_{i=1}^r \{ x_i+\beta^{\frac{-p}{(2p_i)(2p+1)}}[-L_V;L_V]\}) = 1\}$,
\begin{align*}
&\sup_\mu \left|J^V_\beta(\mu)-\sum_i J^{c_i(x-x_i)^{2p_i}}_\beta(\mu|_{x_i+\beta^{\frac{-p}{(2p_i)(2p+1)}}[-L_V;L_V]})\right|\\
&\leqslant  O(\frac{1}{\beta}) + O(\max_i \beta^{\frac{-p(2p_i+1)}{(2p_i)(2p+1)}}) = O(\beta^{-1/2})
\end{align*}
Indeed neglecting the interaction term gives an error of order $\beta^{-1}$ and replacing in the $i$ th well $V$ by $c_i(x-x_i)^{2p_i}$ gives an error of order $\beta^{\frac{-p(2p_i+1)}{(2p_i)(2p+1)}}$ whose maximal value is obtained for $p_i=p$. Therefore,
$$\inf J^V_\beta = \inf\{\sum_i J^{c_i(x-x_i)^{2p_i}}_\beta(\mu_i) : \sum_i \alpha_i = 1\}+O(\beta^{-1/2})$$
where in the above sum we take the infimum on $\alpha$ in the simplex and $\mu_i$ is a measure of mass $\alpha_i$ (The minimisation problem forces the optimal $\mu_i$ to have a support in  
$x_i+\beta^{\frac{-p}{(2p_i)(2p+1)}}[-L_V;L_V]$ if $L_V$ was taken sufficiently large). We next notice that we can remove the terms with small $\alpha_{i}'$s. Indeed, if for $\epsilon>0$ we set  $T_{\epsilon}=\{i: \alpha_{i}\ge\epsilon\}, $ we have 
$$ \inf\{\sum_{i\in T_{\epsilon}} J^{c_i(x-x_i)^{2p_i}}_\beta(\mu_i) : \sum_i \alpha_i = 1\} \le  \inf\{\sum  J^{c_i(x-x_i)^{2p_i}}_\beta(\mu_i) : \sum  \alpha_i = 1\}\quad $$
$$\qquad\qquad \qquad\qquad \qquad  \le
 \inf\{\sum_{i\in T_{\epsilon}} J^{c_i(x-x_i)^{2p_i}}_\beta(\mu_i) : \sum_{i\in T_{\epsilon}} \alpha_i = 1\}$$
 where in the left hand side we just used that $J^{c_i(x-x_i)^{2p_i}}_\beta$ is non-negative and in the right hand side that we restrict the set of measures where we take the infimum.   We set $\eta(\epsilon)=\sum_{i\in T^{c}_{\epsilon}}\alpha_{i}\le r\epsilon$.
Rescaling in the left hand side $\{\alpha_{i}, i\in T_{\epsilon}\}$ by $\{\alpha_{i}(1-\eta(\epsilon))^{-1}, i\in T_{\epsilon}\}$ while using  that $J^{c_i(x-x_i)^{2p_i}}_\beta(\alpha \mu_i)\ge \alpha^{2} J^{c_i(x-x_i)^{2p_i}}_\beta(\mu_i)$  for every $\alpha\in [0,1]$ since the potential is non negative, we conclude
$$ (1-r\epsilon)^{-2}\inf\{\sum_{i\in T_{\epsilon}} J^{c_i(x-x_i)^{2p_i}}_\beta(\mu_i) : \sum_{i\in T_{\epsilon}} \alpha_i = 1\} \qquad$$
$$\qquad\qquad \le  \inf\{\sum  J^{c_i(x-x_i)^{2p_i}}_\beta(\mu_i) : \sum  \alpha_i = 1\}$$
We have seen that the masses $\alpha_{i}$ for $i$ such that $p_{i}<p$ go to zero with $\beta$ like $\beta^{(\frac{p}{p_{i}}-1)\frac{p}{4(2p+1)}}$ and so we may assume that they are in $T_{\epsilon}$. Then, factorizing the masses $\alpha_{i}$ gives for $\beta$ large,
\begin{align*}
\inf J^V_\beta &=(1+o(\epsilon))\left(\inf\{ \sum_{i\in T_{\epsilon}} \alpha_i^2J^{\frac{c_i}{\alpha_i}(x-x_i)^{2p_i}}_\beta(\mu_i): \sum_{i\in T}\alpha_i = 1\}+O(\beta^{-1/2})\right)\\
%&= \inf\{ \sum_{i:\alpha_i\neq 0} \alpha_i^2\beta_{p_i}^{-1}(F(\frac{c_i}{\alpha_i}x^{2p_i})+O(\beta_{p_i}^{-\frac{1}{11}})): \sum_i \alpha_i = 1\}+O(\beta^{-1/2})\\
&= (1+o(\epsilon))\beta_p^{-1}\inf\{ \sum_{i \in T_{\epsilon}} \alpha_i^{2-\frac{1}{2p_i+1}} (a_{p_i}c_i^{\frac{1}{2p_i+1}}+O(\beta_{p_i}^{-\frac{1}{11}})): \sum_i \alpha_i = 1\}\\
\end{align*}
where we finally used { \eqref{toprove} assuming that $c_{i}/\alpha_{i}$ does not blow up, namely $\alpha_{i}$ does not go to zero. }
\color{black}
Finally, we see that the above infimum is taken at $\alpha_{i}^{*}=c_{i}^{-\frac{1}{2p}}/\sum_{i}c_{i}^{-\frac{1}{2p}}$ yielding

\begin{align*}
\inf J^V_\beta
%&=(1+o(1))\left( \beta_p^{-1}a_p\inf\{ \sum_{i:\alpha_i\neq 0} \alpha_i^{2-\frac{1}{2p+1}}c_i^{\frac{1}{2p+1}}: \sum_{i:p_i=p} \alpha_i = 1\}+O(\beta^{-1/2})\right) \\
&=(1+o(1))\left( \beta_p^{-1}a_p(\sum_{j:p_j=p} c_j^{-1/(2p)})^{\frac{1}{2p+1}-1}+O(\beta^{-1/2})\right) \end{align*}

The expression for $c_V$ and its scaling property $c_{\alpha V}=\alpha^{\frac{1}{2p+1}}c_V$ is a trivial consequence from this formula. Moreover, our proof shows that $\{\alpha^{*}_{i}\}_{1\le i\le r}$  is the limit of the mass of the neighborhood of the minimizers  $\{x_{i}\}_{1\le i\le r}$ since it is the unique minimizer of the above approximation of the energy. This proves the third point of the Lemma.

%
%
%\begin{itemize}
%\item
%For the first point of the proposition, we see from \eqref{conv1} and \eqref{infphi} that for every $i$
%$$J_\beta^{c_{i}x^{2p}/\alpha_{i}^{\beta}}(\nu_i^\beta)= \beta^{-\frac{1}{2p+1}}F(c_{i}(\alpha_{i}^{\beta})^{-1} x^{2p})(1+o(1))$$
%As a consequence, we deduce from \eqref{eqJ} that
%$$\inf J_{\beta}^{V}= \beta^{-\frac{1}{2p+1}}\inf_{\alpha}\sum \alpha_{i}^{2}F(c_{i} x^{2p}/\alpha_{i})(1+o(1))$$
%{ Check error is indeed small compared to the first term and also the discrepancy with the announced result which does not say anything about the $\alpha_{i}$ }
% {\color{blue}
%
%The asymptotic for $c^V_\beta$ should follow from:
% $$c^V_\beta=J^V_\beta(\mu^V_\beta)=\beta^{-\frac{1}{2p+1}}\tilde J^V_\beta(\nu^V_\beta)$$
% and
% $$\tilde J^V_\beta(\nu^V_\beta)\sim\phi_m(\alpha^m(\beta))\sim\phi_m(\bar\alpha^m)$$
% and we can easily see that this last expression has well dfined limit ($\bar\alpha^m$ is explicit). This limit should be by the way $\sim \tilde J^V_\beta(\nu)$
%	 }
%	
%	
%	
%	{\bf Scaling of $c_{\alpha V}$}
%	This comes from the fact that a linear change of variable proves that:
%	
%	$$c^{\alpha x^{2p}}_\beta = c^{x^{2p}}_{\alpha^{-\frac{1}{p}}\beta}$$
%	and we deduce the scaling factor for the potential $x^{2p}$. But since everything behave only according to the behavior in the flattest well we can show that this can be generalized.
%	
%	
%	
%	\end{itemize}

\end{proof}

\bibliographystyle{acm}
\bibliography{lowT}

\end{document}